\documentclass[10pt,twoside, a4paper,reqno]{amsart}

\usepackage{amsmath,amssymb,eucal,graphicx,mathtools}
\usepackage[foot]{amsaddr}
\usepackage{amsthm}
\usepackage{hyperref}
\usepackage{graphicx}
\usepackage{caption}
\usepackage{subcaption}
\usepackage{listings}
\usepackage[figuresleft]{rotating}
\usepackage{lscape}
\usepackage{tikz}
\usetikzlibrary{shapes.geometric, positioning,calc, patterns,decorations.markings,backgrounds}
\usepackage{contour}
\usetikzlibrary{arrows}
\usetikzlibrary{decorations.pathmorphing}
\usepackage{tikz-cd}
\usepackage{stmaryrd}
\usepackage[capitalise]{cleveref}
\usepackage{comment}
\usepackage{verbatim}
\usepackage{tkz-euclide}
\usepackage{adjustbox}
\usepackage{mathdots}
\usepackage{circuitikz}
\usetikzlibrary{decorations.markings}
\usepackage{cite}
\usepackage{todonotes}
\usepackage{mathrsfs}
\usepackage{mathtools}
\usepackage{enumitem}
\usepackage[margin=2.5cm]{geometry}
\usepackage{varwidth}

\usepackage{microtype}

\mathchardef\mhyphen="2D 
\DeclareMathOperator{\mods}{mod}
\newcommand{\stautilt}{\mathrm{s\tau\mhyphen tiltp}\,}
\newcommand{\staurigid}{\mathrm{\tau\mhyphen rigidp}\,}
\newcommand{\stauinvtilt}{\mathrm{s\tau^{-1}\mhyphen tiltp\,}}
\newcommand{\stauinvrigid}{\mathrm{s\tau^{–1}\mhyphen rigid}}
\newcommand\taurigid{
\operatorname{\tau\mhyphen rigid}}

\newcommand{\twopresilt}{2\mhyphen\operatorname{presilt}}
\newcommand{\twosilt}{2\mhyphen\operatorname{silt}}

\DeclareMathOperator{\Ext}{Ext}
\DeclareMathOperator{\Hom}{Hom}

\DeclareMathOperator{\End}{End}

\DeclareMathOperator{\projdim}{proj.dim}
\DeclareMathOperator{\add}{add}
\DeclareMathOperator{\thick}{thick}
\DeclareMathOperator{\inj}{inj}
\DeclareMathOperator{\proj}{proj}

\DeclareMathOperator{\gldim}{gl.\!dim}
\DeclareMathOperator{\brick}{brick}
\DeclareMathOperator{\sbrick}{sbrick}
\DeclareMathOperator{\flsbrick}{f_L-sbrick}
\DeclareMathOperator{\frsbrick}{f_R-sbrick}
\DeclareMathOperator{\Filt}{Filt}

\DeclareMathOperator{\Fac}{Fac}
\DeclareMathOperator{\Sub}{Sub}
\DeclareMathOperator{\Hasse}{Hasse}
\DeclareMathOperator{\wide}{wide}

\DeclareMathOperator{\torf}{torf}
\DeclareMathOperator{\ftorf}{f-torf}
\DeclareMathOperator{\charac}{char}
\DeclareMathOperator{\spann}{span}
\DeclareMathOperator{\ind}{ind}

\DeclareMathOperator{\rad}{rad}
\newcommand{\op}{^{\mathrm{op}}}

\DeclareMathOperator{\soc}{soc}

\DeclareMathOperator{\tauint}{\tau-itv}
\DeclareMathOperator{\ftors}{f-tors}

\newcommand{\sqbinom}[2]{\left\llbracket \!\begin{smallmatrix} #1 \\ #2 \end{smallmatrix}\!\right\rrbracket}

\DeclareMathOperator{\tors}{tors}

\newcommand{\plus}[1]{\underline{#1}}
\newcommand{\minus}[1]{\overline{#1}}

\newcommand{\calB}{\mathcal{B}}

\newcommand{\calL}{\mathcal{L}}

\newcommand{\calU}{\mathcal{U}}
\newcommand{\calT}{\mathcal{T}}
\newcommand{\calW}{\mathcal{W}}
\newcommand{\calF}{\mathcal{F}}

\newcommand{\calS}{\mathcal{S}}
\newcommand{\calR}{\mathcal{R}}
\newcommand{\calX}{\mathcal{X}}
\newcommand{\calY}{\mathcal{Y}}
\newcommand{\Wfrak}{\mathfrak{W}}
\newcommand{\R}{\mathbb{R}}
\newcommand{\C}{\mathbb{C}}
\newcommand{\calC}{\mathcal{C}}

\newcommand{\bfg}{\mathbf{g}}

\newcommand{\Z}{\mathbb{Z}}

\newcommand{\T}{\mathcal{T}}
\newcommand{\U}{\mathcal{U}}
\newcommand{\V}{\mathcal{V}}
\newcommand{\W}{\mathcal{W}}

\newcommand{\covered}{{\,\,<\!\!\!\!\cdot\,\,\,}}

\DeclareMathOperator{\fbrick}{\mathrm{f-brick}}
\DeclareMathOperator{\intheart}{\mathrm{int}\text{-}\mathrm{heart}}
\DeclareMathOperator{\hall}{\widehat{\mathscr{H}}}

\counterwithin*{equation}{section}

\newtheorem{theorem}{Theorem}[section]
\newtheorem{corollary}[theorem]{Corollary}
\newtheorem{lemma}[theorem]{Lemma}
\newtheorem{proposition}[theorem]{Proposition}
\newtheorem{definitionproposition}[theorem]{Definition-Proposition}
\newtheorem{conjecture}[theorem]{Conjecture}

\theoremstyle{definition}
\newtheorem{definition}[theorem]{Definition}
\newtheorem{example}[theorem]{Example}
\newtheorem{remark}[theorem]{Remark}

\makeatletter
\newtheorem*{rep@theorem}{\rep@title}
\newcommand{\newreptheorem}[2]{%
	\newenvironment{rep#1}[1]{%
		\def\rep@title{#2 \ref{##1}}%
		\begin{rep@theorem}}%
		{\end{rep@theorem}}}
\makeatother
\newreptheorem{theorem}{Theorem}

\newcommand{\support}[1]{\color{red}}

\title{Bricks and \texorpdfstring{$\tau$}{tau}-tilting theory under base field extensions}
\author[E. D. Børve]{Erlend D. Børve}
\address{Institut für Mathematik und Wissenschaftliches Rechnen, Universität Graz, Mozartgasse 14, 8010 Graz, Austria}
\email{erlend.borve@uni-graz.at}
\thanks{E.D.B. gratefully acknowledges support from the French ANR grant CHARMS (ANR-19-CE40-0017-02), the Deutsche Forschungsgemeinschaft (DFG, German Research Foundation) -- Project ID 281071066 -- TRR 191, and from the NAWI Graz Postdoc fellowship.}

\author[M. Kaipel]{Maximilian Kaipel}
\address{Abteilung Mathematik, Department Mathematik/Informatik der Universität
zu Köln, Weyertal 86-90, 50931 Cologne, Germany}
\email{mkaipel@uni-koeln.de}
\thanks{M.K. acknowledges support from the Deutsche Forschungsgemeinschaft (DFG, German Research Foundation) -- Project ID 281071066 -- TRR 191.}

\makeatletter
\let\@wraptoccontribs\wraptoccontribs
\makeatother
\contrib[with an appendix by]{Eric J. Hanson}
\address{Department of Mathematics, North Carolina State University, Raleigh, NC 27695, USA}
\email{ejhanso3@ncsu.edu}
\thanks{E.J.H. is supported by an AMS-Simons travel grant.}


\begin{document}

\begin{abstract}
    Let $K:k$ be a field extension and let $\Lambda$ be a finite-dimensional $k$-algebra. We investigate the relationship between $\Lambda$ and $\Lambda_K {\coloneqq} \Lambda \otimes_k K$ with particular emphasis on various aspects of $\tau$-tilting theory and bricks. We show that many types of objects for $\Lambda$ lift injectively to the same type of object for $\Lambda_K$, and many common constructions in $\tau$-tilting theory commute with the process of extending the base field. One of our main applications is the construction of a faithful functor from the $\tau$-cluster morphism category $\Wfrak(\Lambda)$ of $\Lambda$ to the $\tau$-cluster morphism category $\Wfrak(\Lambda_K)$ of $\Lambda_K$. In particular, this establishes a faithful functor from $\Wfrak(\Lambda)$ to a group whenever $k$ is of characteristic zero which has many important consequences. In the appendix, E. J. Hanson shows the analogous result whenever $k$ is a finite field. Moreover, we give some nontrivial examples to illustrate the behaviour of $\tau$-tilting finiteness under base field extension.
\end{abstract}

\maketitle

\section{Introduction}

Let $K:k$ be a field {extension} and fix a finite-dimensional $k$-algebra $\Lambda$. The endeavours of many people have aimed to understand the relationship between the representation theory of the finite-dimensional $k$-algebra $\Lambda$ and that of the finite-dimensional $K$-algebra $\Lambda_K \coloneqq \Lambda \otimes_k K$. These investigations were initiated in \cite{JL82}, where various homological dimensions of $\Lambda$, like the global dimension and the finitistic dimension, were compared with those of $\Lambda_K$. For this purpose, the scalar extension functor $- \otimes_k K: \mods \Lambda \to \mods \Lambda_K$, between the categories of finite-dimensional right modules, plays a central role. For many of the homological questions in \cite{JL82}, the best behaviour is exhibited when $K:k$ is MacLane separable, see \cref{defn:MacLaneseparable}. With this assumption, the global dimensions of $\Lambda$ and $\Lambda_K$ coincide. 

Moreover, an important feature of an algebra is its representation type. The celebrated trichotomy theorem of \cite{Drozd1977} states that a finite-dimensional algebra over an algebraically closed field is either representation finite, tame or wild. Over arbitrary base fields, the notion of tameness has to be adapted \cite{CB1991,CB1992}. The relationship between the representation type of $\Lambda$ and the representation type of $\Lambda_K$ was first considered {in \cite{JL82}}. It was shown that when $K:k$ is MacLane separable, then $\Lambda$ is representation finite if and only if $\Lambda_K$ is representation finite. The goal of establishing a similar relationship for generalised representation tame algebras was pursued in \cite{Kas2001,MendezPerez2013,Perez2016}.

One important tool for investigating the category of finite-dimensional modules of $\Lambda$ is Auslander--Reiten theory \cite{AuslanderReiten1975}. In \cite{Kas00}, another proof was given of the fact that representation finiteness is preserved by MacLane separable field extension by relating the Auslander--Reiten theory of $\Lambda$ with that of $\Lambda_K$. 
In fact, representation finiteness of $\Lambda$ is characterised by properties of the scalar extension functor when $k$ is an infinite field \cite{Zayed94}. Moreover, different classes of algebras have been shown to be preserved and reflected under base field extensions in certain cases \cite{Li2021,Li2023derivediscrete}. 

Even if an investigation of the behaviour under base field extensions is not the primary objective of the research, it is often an essential tool, which enables more geometric insights into the objects of interest \cite{DIJ2019,IK2024}. This last idea is also common in algebraic geometry, being applied in descent theory, see \cite{Hohl2023survey, Milne2024survey} and the references therein. A particularly important case is the extension of Gabriel's theorem beyond algebraically closed fields, where the setting of species replaces that of quivers \cite{DlabRingel75,DlabRingel76}. The relationship between these two settings is deeply rooted in the theory of field extensions, see also \cite{DengDu2006,DengDu2007,Hubery2004}. Moreover, the proof of the Second Brauer--Thrall Conjecture can be extended to perfect fields using the theory of base field extensions \cite{BautistaSalmeron2005}.

The work in the present paper is both intrinsically as well as extrinsically motivated. We begin by studying the behaviour of many objects within representation theory under the scalar extension functor, but with the ultimate aim of gaining new insights into an important category of interest, namely the \textit{$\tau$-cluster morphism category}. The general setting of our work is $\tau$-tilting theory \cite{AIR2014}, which uses Auslander--Reiten theory to extend classical tilting theory.

\subsection{General results about \texorpdfstring{$\tau$}{tau}-tilting theory and bricks}

The study of $\tau$-tilting theory has gained widespread popularity since its introduction in \cite{AIR2014} and has widely been accepted as the correct framework for generalising statements about hereditary algebras to arbitrary finite-dimensional algebras. In particular, it can be seen as a generalisation of (classical) tilting theory to higher projective dimensions, since whenever the $\tau$-tilting modules, the central objects of $\tau$-tilting theory, have projective dimension at most one, they coincide with (classical) tilting modules \cite[Cor. IV.4.7]{ARS1995}.

As its name suggests, $\tau$-tilting theory makes use of the Auslander--Reiten translation $\tau$, whose behaviour under base field extension has already been studied in \cite{Kas00}. Modules which do not have nonzero homomorphisms to their Auslander--Reiten translation are called \textit{$\tau$-rigid}, and give rise to $\tau$-perpendicular subcategories of $\mods \Lambda$ \cite{Jas15}. Similar to perpendicular subcategories for hereditary algebras \cite{GL91}, these are of utmost importance within $\tau$-tilting theory. Notably, they are equivalent to module categories of other finite-dimensional algebras. Closely related to this are 
$\tau$-exceptional sequences \cite{BM18t}{, which}
generalise exceptional sequences over hereditary algebras \cite{cbw92,ringel_exceptional} to guarantee that sequences of maximal length always exist. 

Moreover, $\tau$-tilting theory is closely related with the study of bricks, since the endotop of each $\tau$-rigid module decomposes into bricks \cite{DIJ2019,Asa21}. These indecomposable direct summands are $\Hom$-orthogonal{,} which makes them (left-finite) semibricks \cite{Asai2020}. It is a well-known result that semibricks correspond bijectively to wide subcategories \cite[Thm. 1.2]{Rin76} (recited in \Cref{eq:Ringelbij}), and the strategy of lifting individual objects using the scalar extension functor appears to be the correct way to lift subcategories. Most of the terms in the following theorem have thus been explained. For the
{remaining ones}, we refer to later sections for precise definitions. 

\begin{theorem}(simplified)
    Let $K:k$ be a field extension and $\Lambda$ be a finite-dimensional $k$-algebra. Then the scalar extension functor $- \otimes_k K: \mods \Lambda \to \mods \Lambda_K$ induces injective maps on the following objects:
    \begin{enumerate}
        \item $\tau$-rigid modules, which restricts to support $\tau$-tilting modules and tilting modules; 
        \item $\tau$-rigid pairs, which restricts to support $\tau$-tilting pairs;
        \item functorially finite torsion classes;
        \item $\tau$-perpendicular intervals of torsion classes;
        \item complete signed $\tau$-exceptional sequences;
        \item {complete $\tau$-exceptional sequences;}
        \item left-finite semibricks;
        \item right-finite semibricks.
    \end{enumerate}
    If $k$ is perfect, then the same holds for the following objects:
    \begin{enumerate}
        \item[(9)] semibricks;
        \item[(10)] wide subcategories.
    \end{enumerate}
    Moreover, the partial orders on (3), (4) and (10) are preserved. 
\end{theorem}
\begin{proof}
    (1) and (2) are both a combination of \cref{lem:taurigidlift} and \cref{cor:tautiltlift} and \cref{lem:tiltinglift}; (3) is a combination of (2) and \cite[Thm. 2.7]{AIR2014}, the fact that the poset structure is preserved is \cref{lem:posetlift}; (4) is \cref{cor:tauperpitvlift}, the fact that the poset structure is preserved follows from the additivity of $- \otimes_k K$; (5) is \cref{lem:signedexcplift}; {(6) is \cref{cor:tauexceptionallift};} (7) and (8) follow from (2) and \cite[Thm. 1.3]{Asai2020} and the fact that $- \otimes_k K$ is injective-on-objects; (9) is \cref{lem:semibrickslift}; (10) is a combination of (9) and \cite[Thm. 1.2]{Rin76} (the latter result is recited in \Cref{eq:Ringelbij}). 
    {T}he fact that the poset structure is preserved follows from the exactness of $- \otimes_k K$.
\end{proof}

Besides individual objects and subcategories, it is important for our purpose to consider additional constructions and their interaction with base field extension. We briefly describe the terms appearing here. The inverse Auslander--Reiten translation $\tau^{-1}$ gives rise to $\tau^{-1}$-tilting theory. Any support $\tau^{-1}$-tilting module can be constructed explicitly from a support $\tau$-tilting module, as we recall in \Cref{prop:AIRp2.15}. Moreover, each $\tau$-rigid module can be completed to a support $\tau$-tilting module. In the poset of support $\tau$-tilting modules there is a unique maximal completion and a unique minimal completion, which are known as the Bongartz and co-Bongartz completions, respectively \cite{AIR2014}. We have already discussed that taking the endotop of a $\tau$-rigid module gives a left-finite semibrick. The dual construction for $\tau^{-1}$-rigid modules gives right-finite semibricks \cite{Asai2020}. Finally, constructing tilted \cite{HappelRingelTilted1982} and preprojective algebras \cite{BGL1987} from hereditary algebras is classical. We refer to later sections for precise statements of the claims in the following theorem. 

\begin{theorem}(simplified)
    Let $K:k$ be a field extension and let $\Lambda$ be a finite-dimensional $k$-algebra. Then the following constructions for $\Lambda$ commute with scalar extension along $K:k$:
    \begin{enumerate}
        \item The bijection from $\tau$-rigid pairs to $\tau^{-1}$-rigid pairs;
        \item Taking co-Bongartz completions of $\tau$-rigid pairs;
        \item Taking Bongartz completions of $\tau$-rigid pairs{.}
    \end{enumerate}
    If $K:k$ is MacLane separable, then the following constructions also commute with taking base field extensions:
    \begin{enumerate}
        \item[(4)] Turning $\tau$-rigid modules into left-finite semibricks;
        \item[(5)] Turning $\tau^{-1}$-rigid modules into right-finite semibricks;
        \item[(6)] Constructing the tilted algebra for a tilting module of a hereditary algebra.
        \item[(7)] Constructing the preprojective algebra for a projective module of a hereditary algebra.
    \end{enumerate}
\end{theorem}
\begin{proof}
    (1) is \cref{lem:dualtaucommutes}; (2) is \cref{lem:cobongartzcommute}; (3) is \cref{lem:bongartzcommute};
    (4) is \cref{lem:flbrickcommute}; (5) is \cref{lem:frbrickcommute}; (6) is \cref{lem:tiltedlift} and
    (7) is \cref{preprojcommute}.
\end{proof}

Geometrically, the $\tau$-tilting theory of $\mods \Lambda$ is encoded in the $g$-vector fan $\Sigma(\Lambda)$ \cite{DIJ2019}. The idea of considering a fan for such a purpose originates in the theory of cluster algebras \cite{FZ07}{,} and its properties have been studied in various settings, see \cite{AHIKM2022,AokiYurikusa2023} and the references therein. Moreover, the $g$-vector fan embeds into the wall-and-chamber structure \cite{Asai2020,BST19}{,} which is the support of the stability scattering diagram \cite{Bri17}. Since all stable modules are bricks, this relationship highlights another facet of the connection between bricks and $\tau$-tilting theory. In contrast to the $g$-vector fan, investigating the wall-and-chamber structure goes beyond the domain of $\tau$-tilting theory. Nonetheless, we are able to describe the behaviour of both under base field extensions.

\begin{theorem}\label{thm:wallcintro}
    Let $K:k$ be a field extension and let $\Lambda$ be a finite-dimensional $k$-algebra. Then the $g$-vector fan $\Sigma(\Lambda)$ embeds into the $g$-vector fan $\Sigma(\Lambda_K)$. Moreover, if $K:k$ is a finite and separable extension, then the wall-and-chamber structure of $\Lambda$ embeds into the wall-and-chamber structure of $\Lambda_K$.
\end{theorem}
\begin{proof}
    The first part of the statement is \cref{lemma:gvectorfan} and the moreover part is \cref{thm:WACembed}. 
\end{proof}

{Wall-and-chamber structures are intrinsically linked with stability conditions \cite{King1994, KS08, Bri17, BPP16}. It has been studied how the stability manifold is altered by extension of scalars \cite{Sos12,ChangQiu2024}. See also \cite{ChenMandelQin2024} for a similar result about quiver coverings and scattering diagrams.}

\subsection{\texorpdfstring{$\tau$}{tau}-cluster morphism category}

Generally speaking, the $\tau$-cluster morphism category $\Wfrak(\Lambda)$ of a finite-dimensional algebra $\Lambda$ encompasses many of the aforementioned objects of the $\tau$-tilting theory of $\Lambda$. More specifically, its objects are $\tau$-perpendicular subcategories of $\mods \Lambda$ and morphisms are given by $\tau$-rigid pairs (which represent corresponding $\tau$-tilting reductions \cite{Jas15}). These categories were originally introduced for hereditary algebras, where they are closely connected to exceptional sequences, cluster algebras and maximal green sequences \cite{IT17, IgusaTodorovMGS2021}. In these cases, the classifying space often forms a $K(\pi,1)$ space for the picture group,  which is of independent interest \cite{ITW16, IgusaTodorov22}. Similar questions about arbitrary finite-dimensional algebras have received much attention \cite{BarnardHanson2022,BH21, BM18w, HI21,HI21p, Kai23, Kai24, Bor21}, but the understanding of this category is limited to a few special cases. Part of a sufficient condition for the classifying space to be a $K(\pi,1)$ space is the existence of a faithful functor $\Wfrak(\Lambda) \to G$ for some group $G$ viewed as a groupoid with one object \cite{Igu14} (see \Cref{Igu14.3.5}). We apply the theory developed in this article to obtain new families of algebras for which the $\tau$-cluster morphism category admits a \textit{faithful group functor}, that is, a faithful functor from $\Wfrak(\Lambda)$ into a group.

\begin{theorem}(\cref{thm:basefieldmain})\label{repthm:basefieldmain}
    Let $K:k$ be a MacLane separable field extension and let $\Lambda$ be a finite-dimensional $k$-algebra. There exists a faithful functor $\Wfrak(\Lambda) \to \Wfrak(\Lambda_K)$. Therefore, if $\Wfrak(\Lambda_K)$ admits a faithful group functor, so does $\Wfrak(\Lambda)$.
\end{theorem}

This result establishes a well-behaved relationship between the $\tau$-cluster morphism categories of different algebras. Such a connection is quite special and only few other settings are known to admit such a relationship between $\tau$-cluster morphism categories of distinct algebras \cite{BarnardHanson2022,Kai24}.

In concurrent work, the first-named author shows that $\Wfrak(\Lambda)$ admits a faithful group functor whenever the base field is algebraically closed and of characteristic 0 \cite{BorMot}. Since any field extension of a field of characteristic 0 is MacLane separable (by definition), one can clearly use \Cref{repthm:basefieldmain} to remove the assumption of algebraic closure on $k$ (see \Cref{cor:mainbasefield+BorMot}).

\begin{corollary}(\cref{cor:mainbasefield+BorMot})
    Let $k$ be a field of characteristic 0 and let $\Lambda$ be a finite-dimensional $k$-algebra. Then $\Wfrak(\Lambda)$ admits a faithful group functor.
\end{corollary}

With a similar goal in mind, it is shown in the appendix that the analogous result holds for algebras over finite fields.

\begin{theorem}[Theorem~\ref{thm:finite_field}, simplified]
    Let $k$ be a finite field and let $\Lambda$ be a finite-dimensional $k$-algebra. Then $\mathfrak{W}(\Lambda)$ admits a faithful group functor.
\end{theorem}

The proof of Theorem~\ref{thm:finite_field} is based upon the theory of \emph{picture groups} and \emph{completed (finitary) Ringel--Hall algebras}, see Appendix~\ref{sec:appendix} for the relevant definitions and historical context. There are also cases where these arguments can be lifted to infinite fields:

\begin{corollary}[Corollary~\ref{cor:field_change}, simplified]\label{cor:introfieldchange}
    Let $\Lambda$ be a finite-dimensional algebra over an arbitrary field, and suppose there exists a lattice isomorphism $\eta: \tors\Lambda \rightarrow \tors\Lambda'$ for some finite-dimensional algebra $\Lambda'$ over a finite field. Then $\Wfrak(\Lambda)$ admits a faithful group functor.
\end{corollary}

For example, \cref{cor:introfieldchange} implies that, over any field, the $\tau$-cluster morphism categories of the following classes of $k$-algebras admit faithful group functors: representation finite hereditary algebras (as already shown in \cite{IgusaTodorov22}), preprojective algebras of representation finite hereditary algebras (\cref{prop:preproj_distinct_generators}), and representation finite string algebras (\cref{prop:string}).

Moreover, if an algebra $\Lambda$ defined over a finite field or a field of characteristic zero has at most three isomorphism classes of simple objects, then the results of \cite{BarnardHanson2022,Igu14} imply that the classifying space $\calB \Wfrak(\Lambda)$ is in fact a $K(\pi,1)$ space.

\subsection{\texorpdfstring{$\tau$}{tau}-tilting finiteness under base field extension}

Inspired by representation finite algebras, the class of $\tau$-tilting finite algebras consists of algebras admitting only finitely many isomorphism classes of {basic} support $\tau$-tilting modules \cite{DIJ2019}. For hereditary algebras, these two notions coincide, but even representation wild algebras may be $\tau$-tilting finite in general. The importance of $\tau$-tilting finiteness is highlighted by the plethora of equivalent conditions to being $\tau$-tilting finite, see \cite{AHMV19,DIJ2019, Gar24,IRTT2015,MS17,SchrollTreffinger2022,Sentieri2023}. Given a Maclane separable field extension $K:k$ and a $k$-algebra $\Lambda$, it turns out that $\Lambda$ is representation finite if and only if $\Lambda_K$ is, by \cite[Thm. 3.3]{JL82}, as recited in \Cref{thm:JL}\eqref{thm:JL2}. Similar relationships for representation tame algebras have been investigated \cite{Kas2001,MendezPerez2013,Perez2016}, but generally more care needs to be taken because the definition of tameness relies on algebraic closure of the ground field as mentioned previously. In the setting of $\tau$-tilting theory one possible analogue of tameness has been proposed in \cite{AokiYurikusa2023}. An algebra $\Lambda$ is \textit{$g$-tame} if its $g$-vector fan is dense in the ambient space. Tame algebras (naturally defined over algebraically closed fields) were shown to be $g$-tame in \cite{PY23}.

It is a natural and important question to ask whether $\tau$-tilting finiteness behaves similarly with respect to base field extension. In this direction, we are able to show the following.

\begin{theorem}\label{thm:introtaufinite}
    There exists a $\tau$-tilting finite $k$-algebra $\Lambda$ and a field extension $K:k$ such that $\Lambda_K$ is not $g$-tame and thus not $\tau$-tilting finite. On the other hand, let $K:k$ be MacLane separable, and $\Pi(\Gamma)$ the preprojective algebra of a representation finite hereditary $k$-algebra $\Gamma$. Then $\Pi(\Gamma_K) \cong \Pi(\Gamma) \otimes_k K$. Thus, the generally representation infinite algebra $\Pi(\Gamma)$ remains $\tau$-tilting finite under MacLane separable field extensions.
\end{theorem}
\begin{proof}
    Such an example is given in \cref{exmp:Acounterexample}. The isomorphism $\Pi(\Gamma_K) \cong \Pi(\Gamma) \otimes_k K$ is shown in \cref{preprojcommute}. By \cref{thm:JL} this means that $\Pi(\Gamma) \otimes_k K$ is again a preprojective algebra of a representation finite hereditary $K$-algebra. The conclusion follows from \cite[Thm. 7.9]{AHIKM2022} which implies that preprojective algebras over representation finite hereditary algebras are $\tau$-tilting finite. This is because the poset of support $\tau$-tilting modules is isomorphic to a weak order on the Weyl group corresponding to the Dynkin type of the hereditary algebra, which is a finite group if the algebra is representation finite. 
\end{proof}

We consider it an interesting problem to find further examples or conditions for representation infinite but $\tau$-tilting finite algebras to remain $\tau$-tilting finite under base field extension, see \cref{lem:taufinbrickcondition}. See also \cite[Sec. 3.2]{KKKMM2024} for similar results about skew-group algebras. \\

\textbf{Acknowledgements.} Some of the main ideas of this text were conceived during a visit of E.D.B. at the University of Cologne in May 2024. E.D.B. thanks M.K. for an extraordinarily productive and enjoyable week, as well as Sibylle Schroll for hosting him. Both authors are grateful to Sondre Kvamme, Calvin Pfeifer and Håvard U. Terland for useful discussions. 

E.J.H. is thankful to E.D.B. and M.K. for encouraging him to write Appendix~\ref{sec:appendix}, for including it in their paper, and for offering suggestions for improvement. He also thanks Kiyoshi Igusa and Hipolito Treffinger for useful conversations pertaining to Ringel--Hall algebras.

\section{Preliminaries on field extensions}

Let $k$ be a field. {Let $\overline{k}$} denote {an} algebraic closure of $k$. Recall that $\Lambda$ will always denote a finite-dimensional $k$-algebra. Given a field extension $K:k$, denote by $\Lambda_K$ the $K$-algebra $\Lambda \otimes_k K${.} The \textit{scalar extension functor} is then defined as:
\begin{align*} 
- \otimes_k K: \mods \Lambda &\to \mods \Lambda_K \\
M &\mapsto  M_K \coloneqq M \otimes_k K.
\end{align*}
Since the field extension $K$ of $k$ is a faithfully flat $k$-module, the scalar extension functor is an exact and faithful functor. Moreover, it admits a right adjoint (forgetful) \textit{restriction functor} $(-)|_\Lambda: \mods \Lambda_K \to \mods \Lambda$. We begin by collecting some essential properties of the scalar extension functor $- \otimes_k K$.

In the introduction, we defined $\mods \Lambda$ to be the (abelian) category of finite-dimensional right $\Lambda$-modules. Throughout the entire text, let $\proj \Lambda$ denote the additive subcategory of $\mods \Lambda$ spanned by finite-dimensional projective $\Lambda$-modules. Given a finite-dimensional right $\Lambda$-module, let $\add(M)$ denote the full subcategory of $\mods \Lambda$ spanned by all direct summands of modules of the form $M^{m}$, for some $m\geq 1$ (in particular, one could define $\proj \Lambda$ as $\add(\Lambda_{\Lambda})$, where ${\Lambda_{\Lambda}}$ denotes $\Lambda$ considered as a right $\Lambda$-module).

\begin{lemma}\label{lem:adddeterminesbasic}\label{lem:intersectlift}\label{lem:Kasprojs}
    Let $K:k$ be a field extension. Let $M,N \in \mods \Lambda$.
    \begin{enumerate}
        \item Suppose that $M$ is indecomposable. Then $M$ is a direct summand of $N$ if and only if $M_K$ and $N_K$ have a common nonzero direct summand.
        \item If $M$ and $N$ are indecomposable, then $M_K$ and $N_K$ share a common nonzero direct summand if and only if $M \simeq N$. 
        \item $M \simeq N$ if and only if $M_K \simeq N_K$. In other words, the functor $- \otimes_k K$ is injective-on-objects up to isomorphism.
        \item {Let $P \in \proj \Lambda$. Then $P_K \in \proj \Lambda_K$. Moreover, each projective $\Lambda_K$-module arises as a direct summand of $P_K$ for some $P \in \proj \Lambda$.}
        \item Suppose that $M$ and $N$ are basic. If $\add(M_K) = \add(N_K)$, then $M \simeq N$. 
        \item $\bigcup_{X \in \add(M) \cap \add(N)} \add(X_K) = \{ X: X \in \add(M_K) \cap \add (N_K)\}$. 
    \end{enumerate} 
\end{lemma}
\begin{proof}
    (1)-(3) are \cite[Lem. 2.5]{Kas00}, where (2) is also known as the Noether--Deuring Theorem, see \cite[Thm. 19.25]{Lam2001}. The asserion in (4) is taken from \cite[Lem. 2.1]{Kas00}.

    (5) Let $X$ be an indecomposable direct summand of $M$. From $\add(M_K) = \add(N_K)$ we have that $X_K$ is a direct summand of $(N_K)^r$ for some $r \geq 1$. By (1), it follows that $X$ is a direct summand of $N^r$, and since $X$ is indecomposable, we further deduce that $X$ is a direct summand of $N$. Repeating this for all indecomposable direct summands of $M$, we obtain that every indecomposable direct summand of $M$ is a direct summand of $N$. Reversing the argument, we find similarly that every indecomposable direct summand of $N$ is a direct summand of $M$. It follows that $\add(M) = \add(N)$. Since $M$ and $N$ are basic, the isomorphism $M \simeq N$ follows.

     (6) To show the inclusion $(\subseteq)$, let $X \in \add(M) \cap \add(N)$ be indecomposable. Then $X$ is a direct summand of $M$ and a direct summand of $N$. It follows from (1) that $X_K$ is a direct summand of $M_K$ and of $N_K$, hence $X_K \in \add(M_K)  \cap \add(N_K)$. Thus $\add(X_K) \subseteq \add(M_K) \cap \add(N_K)$, whence the inclusion $(\subseteq)$ holds. Conversely, to show the inclusion $(\supseteq)$, let $Y \in \add(M_K) \cap \add(N_K)$ be indecomposable. It is easy to see that $Y \in \add(Z_K) \subseteq \add(M_K)$ and that $Y \in \add(Z_K') \subseteq \add(N_K)$, for some indecomposable modules $Z,Z' \in \mods \Lambda$. From (1) it follows that $Z \in \add(M)$ and that $Z' \in \add(M)$. It now follows from (2) that $Z \simeq Z'$. In conclusion, $Y \in \add (Z_K)$, where $Z \in \add(M) \cap \add(N)$, which is enough to show that the inclusion $(\supseteq)$ also holds.
\end{proof}

Most importantly, the homomorphisms and extensions between two modules behave well.

\begin{lemma}{\cite[Lem. 2.2]{Kas00}}\label{lem:Kas00.2.2}
    Let $X,Y \in \mods \Lambda$ and $i$ a nonnegative integer. Then the canonical homomorphism of $K$-vector spaces
    \begin{equation*}
         \Ext^i_{\Lambda}(X,Y) \otimes_k K \to \Ext^i_{\Lambda_K}(X_K,Y_K) 
    \end{equation*}
    is an isomorphism which is natural in both arguments.
\end{lemma}

The remainder of this section recalls different types fields and field extensions which we frequently use throughout this paper. Standard notions of Galois theory will be presumed well-known to the reader. Recall that a field extension $K:k$ is a \textit{Galois extension} if it is separable, normal and algebraic. If $K:k$ is a finite, we denote its degree by $[K:k]$.

\begin{definitionproposition}\label{defprop:perfect} 
A field $k$ is said to be \textit{perfect} if the following equivalent assertions hold:
    \begin{enumerate}
        \item The field extension $\overline{k}:k$ is separable, where  $\overline{k}$ denotes an algebraic closure of $k$.
        \item Every algebraic field extension of $k$ is separable. 
        \item Every (possibly nonalgebraic) field extension $k$ is separable.
    \end{enumerate}
\end{definitionproposition}
\begin{proof}
    It is obvious that (3) $\implies$ (2) $\implies$ (1). Also, since every algebraic extension of $k$ is intermediate to $\overline{k}:k$, one easily proves that (1) $\implies$ (2), whence the first two assertions are equivalent. To complete the proof, we refer to {\cite[Thm. V.15.5.3(b)]{BourbakiAlgII}}\label{prop:Bou03.V53b}, where it is shown that (2) $\implies$ (3).
\end{proof}

For example, by \cite[Thm. IV.3]{Jac53}, all fields of characteristic zero and all finite fields are perfect. It is obvious from the definition that algebraically closed fields are perfect.

{
\begin{remark}\label{lem:nicerestriction}
    Let $K:k$ be a finite field extension. There is a functorial isomorphism 
\[(M_K)|_{\Lambda} \simeq M^{[K:k]}. \]  
Consequently, for any indecomposable direct summand $M_i \in \add (M_K)$ we have $M_i|_{\Lambda} \simeq M^{r}$ for some $r \leq [K:k]$. Consider for instance the Galois extension $\C:\R$ which satisfies $[\C: \R]=2${,} and consider $\R$ as a 1-dimensional $\R$-algebra. {It is then clear that} ${(\R \otimes_\R \C)}|_{\R} \simeq \R^2$.
\end{remark}
}

Of particular importance is the following slightly less well-known type of field extension.

\begin{definition}\label{defn:MacLaneseparable}
    Given a field extension $L:k$, let $L_1\colon k$ and $L_2\colon k$ be intermediate field extensions of $L:k$. We say that $L_1\colon k$ and $L_2\colon k$ are \textit{linearly disjoint over $k$ in $L$} if the $k$-homomorphism 
    \begin{equation*}
    \begin{tikzcd}
        L_1 \otimes_k L_2 \arrow[r,"\mu"] & L_1 L_2 \arrow[r,phantom, sloped,"\subseteq"] & L \\
        \lambda_1 \otimes \lambda_2 \arrow[r,mapsto]\arrow[u,phantom, sloped,"\in"] & \lambda_1 \lambda_2  \arrow[u,phantom, sloped,"\in"] &
    \end{tikzcd} 
    \end{equation*}
    is an isomorphism of $k$-algebras, where $L_1 L_2$ is the smallest subalgebra of $L$ containing $L_1$ and $L_2$. 
    A field extension $K:k$ is said to be \textit{MacLane separable} if either $\charac(k)=0$ or $\charac(k)>p\neq 0$ and the field extensions $K$ and $k^{p^{-1}}$ are linearly disjoint over {$k$} in $\overline{K}$. Here $k^{p^{-1}}$ denotes the subfield $\{\lambda \in {\overline{K}} : \lambda^p \in k\}$ of ${\overline{K}}$.
\end{definition}

For example, an algebraic field extension is separable if and only if it is MacLane separable \cite[Thm. IV.9]{Jac53}. In particular, a finite field extension is MacLane separable if and only if it is separable. Moreover, every separable field extension is MacLane separable \cite[p. 163, MacLane's Criterion]{Jac53}, but the converse is not true in general, see \Cref{eg:nonsep} below.

\begin{example}\label{eg:nonsep}
    There are examples of MacLane separable field extensions that are not separable. By the proceeding discussion, the larger field must be infinite and of prime characteristic. One can construct a general class of examples by fixing a field $k$ of characteristic $p\neq0$ and letting $K = k(\xi, \xi^{p^{-1}}, \xi^{p^{-2}}, \dots)$, where $\xi$ is transcendental over $k$. Then $K:k$ is MacLane separable but not separable \cite[Exercise IV.5.1]{Jac53}. The field $K$ is indeed infinite and of characteristic $p$. It happens to also be perfect, since every element in $K$ is a $p$th power.
\end{example}

\begin{theorem}{\cite[Thms. 2.4 and 3.3]{JL82}}\label{thm:JL}
Let $K:k$ be a field extension.
\begin{enumerate}
    \item\label{thm:JL1} The field extension $K:k$ is MacLane separable if and only if scalar extension along $K:k$ preserves the global dimension of all finite-dimensional $k$-algebras, which is to say that
    {the equality} $\gldim \Lambda = \gldim \Lambda_K$ {holds for all finite-dimensional $k$-algebras $\Lambda$}.
    \item\label{thm:JL2} Let $\Lambda$ be a finite-dimensional $k$-algebra. If $K:k$ is MacLane separable, then $\Lambda$ is representation finite if and only if $\Lambda_{K}$ is representation finite.
    \item\label{thm:JL3} In the setting of \eqref{thm:JL2} above, every indecomposable $\Lambda_K$-module is a direct summand of a $\Lambda_K$-module $M_K$ for some indecomposable $M \in \mods \Lambda$.
\end{enumerate}
\end{theorem}

In light of the assertions in \Cref{thm:JL}, one may argue that MacLane separability interacts well with homological algebra, as can be seen by many of the other results in \cite{JL82}.
In \Cref{sec:ttfin}, we explore the behaviour of $\tau$-tilting finiteness, where the analogous statement of \Cref{thm:JL}\eqref{thm:JL2} fails to hold. Further properties in relation to the radical and the socle of a module are summarised in the following lemma.

\begin{lemma}\cite[Lem. 3.3, 3.5]{Kas00}\label{lem:Kasradsoclift}
Let $\Lambda$ be a finite-dimensional $k$-algebra. Let $K:k$ be a field extension and $M \in \mods \Lambda$. 
\begin{enumerate}
	\item\label{lem:Kasradsoclift_rad} There is an inclusion $(\rad \Lambda)_K \subseteq \rad \Lambda_K$.
	\item There is an inclusion $(\rad M)_K \subseteq \rad M_K$.
	\item There is an inclusion $(\soc M)_K \subseteq \rad M_K$.
\end{enumerate}
If $K:k$ is MacLane separable, then all of the above are equalities.
\end{lemma}

\section{\texorpdfstring{$\tau$}{tau}-tilting theory}

The area of $\tau$-tilting theory \cite{AIR2014} can be seen as a generalisation of classical tilting theory to modules of arbitrary projective dimensions, alternatively a completion of the theory with respect to mutation. As the name suggests, the Auslander--Reiten translation $\tau$ plays a central role. In the work of \cite{Kas00}, the behaviour of almost split exact sequences of Auslander--Reiten theory under base field extension was studied. In particular, the Auslander--Reiten translations of $\mods \Lambda$ and $\mods \Lambda_K$, denoted by $\tau_\Lambda$ and $\tau_{\Lambda_K}$, respectively, are connected. We will make this statement precise in \Cref{lem:tautranslatelift}, where we also include a proof for the convenience of the reader. First, we establish a similar result for the Nakayama functors  $\nu_\Lambda: \mods \Lambda \to \mods \Lambda$ and $\nu_{\Lambda_K}: \mods \Lambda_K \to \mods \Lambda_K$, on which our proof of \Cref{lem:tautranslatelift} will depend.

\begin{lemma}\label{cor:Nakfunctlift}
    Let $K:k$ be a field extension and let $M \in \mods \Lambda$. There is an isomorphism of $\Lambda_K$-modules $(\nu_{\Lambda} M)_K \simeq \nu_{\Lambda_K} M_K$.
\end{lemma}
\begin{proof}
    The result follows from applying \cref{lem:Kas00.2.2} twice to obtain the isomorphisms and using the definition $\nu_?(-) = \Hom_k(\Hom_?(-,?),k)$. Indeed, the proof is then as follows:
    \begin{align*} 
    \nu_{\Lambda}(M) \otimes_k K &= \Hom_k(\Hom_\Lambda(M, \Lambda),k) \otimes_k K \\
    & \simeq \Hom_K(\Hom_{\Lambda} (M, \Lambda) \otimes_k K, K) \\
    & \simeq \Hom_K(\Hom_{\Lambda_K} (M_K, \Lambda_K),K) \\
    &= \nu_{\Lambda_K}(M_K).
    \end{align*}
    This completes the proof.
\end{proof}

Similarly, we establish a relationship between the Auslander--Reiten translations. 

\begin{lemma} \cite[Cor 3.6]{Kas00} \label{lem:tautranslatelift}
	Let $K:k$ be a field extension and let $M \in \mods \Lambda$. There is an isomorphism of $\Lambda_K$-modules $(\tau_{\Lambda} M)_K \simeq \tau_{\Lambda_K} M_K$.
\end{lemma}
\begin{proof}
	Let $P^1 \xrightarrow{p^1} P^0 \xrightarrow{p^0} M \to 0$ be a minimal projective presentation of $M$ in $\mods \Lambda$. By \cref{lem:Kasprojs}(4), $P_K^1$ and $P_K^0$ are projective $\Lambda_K$-modules{,} and by \cref{lem:Kasradsoclift}(2) the exact sequence $P_K^1 \xrightarrow{p_K^1} P_K^0 \xrightarrow{p_K^0} M_K \to 0$ is a minimal projective presentation of $M_K$ in $\mods \Lambda_K$. After applying the (exact) Nakayama functor to these sequences, consider the following diagram:
\[
\begin{tikzcd}
	0 \arrow[r] & \tau_{\Lambda_K} M_K \arrow[r]  \arrow[d, dashed] &\nu_{\Lambda_K} P_K^0 \arrow[r]  \arrow[d, "\simeq"] &\nu_{\Lambda_K} P_K^1 \arrow[r]  \arrow[d, "\simeq"] &\nu_{\Lambda_K} M_K \arrow[r]  \arrow[d, "\simeq"] & 0 \\
	0 \arrow[r] & (\tau_{\Lambda} M)_K \arrow[r] &(\nu_{\Lambda} P^0)_K \arrow[r] &(\nu_{\Lambda} P^1)_K \arrow[r] &(\nu_{\Lambda} M)_K \arrow[r] & 0 
\end{tikzcd}
\]
The three maps on the right are isomorphisms by \cref{cor:Nakfunctlift}, and since we have applied Nakayama functors to minimal projective presentations, the relevant squares commute. By exactness and the Five Lemma, there is an induced isomorphism $\tau_{\Lambda_K} M_K \simeq (\tau_\Lambda M)_K$, as required.
\end{proof}

We now introduce the central objects of $\tau$-tilting theory.

\begin{definition}
    Let $M \in \mods \Lambda$ and $P \in \proj \Lambda$.
    \begin{enumerate}
        \item We say $M$ is \textit{$\tau$-rigid} if $\Hom(M, \tau_{\Lambda} M)=0$. If additionally $|M|=|\Lambda|$ then we say $M$ is \textit{$\tau$-tilting}.
        \item We say that the pair $(M,P)$ is \textit{$\tau$-rigid} if {$M$ is a $\tau$-rigid $\Lambda$-module, $P$ is a projective $\Lambda$-module} and $\Hom(P,M)=0$. If additionally $|M|+|P|=|\Lambda|$ then we say $(M,P)$ is a \textit{support $\tau$-tilting pair}. {In this case, the $\Lambda$-module $M$ is called \textit{support $\tau$-tilting}.}
        \item We say that $\Lambda$ is \textit{$\tau$-tilting finite} if it only admits finitely many isomorphism classes of indecomposable $\tau$-rigid modules. Otherwise it is said to be \textit{$\tau$-tilting infinite}.
    \end{enumerate}
\end{definition}

We denote by $\taurigid \Lambda$ the collection of basic $\tau$-rigid $\Lambda$-modules and by $\staurigid \Lambda$ the collection of {basic} $\tau$-rigid pairs $(M,P)$ of $\mods \Lambda${, that is,} such that $M$ and $P$ are basic. We denote the set of basic support $\tau$-tilting pairs by $\stautilt \Lambda$.
The following result extends \cite[Prop. 6.6(d)]{DIJ2019}.
\begin{lemma}\label{lem:taurigidlift}
    Let $K:k$ be a field extension, let $M \in \mods \Lambda$ and let $P \in \proj \Lambda$. Then $M$ is $\tau$-rigid if and only if $M_K$ is $\tau$-rigid. Moreover, $(M,P)$ is a $\tau$-rigid pair if and only if $(M_K, P_K)$ is a $\tau$-rigid pair.
\end{lemma}
\begin{proof}
    The first statement follows directly from the sequence of isomorphisms
    \[ \Hom_\Lambda(M, \tau_\Lambda M) \otimes_k K \simeq \Hom_{\Lambda_K} (M_K, (\tau_\Lambda M)_K) \simeq \Hom_{\Lambda_K}(M_K, \tau_{\Lambda_K} M_K) \]
    where the first isomorphism is \cref{lem:Kas00.2.2} and the second follows from \cref{lem:tautranslatelift}. Moreover, \cref{lem:Kas00.2.2} similarly implies that $\Hom_\Lambda(P,M)=0$ if and only if $\Hom_{\Lambda_K}(P_K,M_K)${,} which gives the desired result.
\end{proof}

Since the number of indecomposable direct summands of modules may be difficult to control under field extension, it is necessary to take a more homological approach to understand (support) $\tau$-tilting modules and support $\tau$-tilting pairs. 

\begin{definition}
    Let $K^b(\proj \Lambda)$ denote the bounded homotopy category of $\proj \Lambda$, and for $i\in \Z$, let $[i]$ denote the $i$th power of the suspension functor in $K^b(\proj \Lambda)$.
    Let $P^\bullet \in K^b(\proj \Lambda)$.
    \begin{enumerate}
        \item We say that $P^\bullet$ is \textit{presilting} if $\Hom(P^\bullet,P^\bullet[i])=0$ for all $i>0$.
        \item We say that $P^\bullet$ is \textit{silting} if it is presilting and $\thick(P^\bullet) = K^b(\proj \Lambda)$.
    \end{enumerate}
    Finally, $P^\bullet$ is called \textit{2-term} if it is isomorphic to an object $(P^i,d^i)$ in $K^b(\proj \Lambda)$ such that $P^i = 0$ for $i \neq -1,0$.
\end{definition}

Denote the collection of basic 2-term presilting and basic silting objects of $K^b(\proj \Lambda)$ by $\twopresilt \Lambda$ and $\twosilt \Lambda$ respectively. The derived scalar extension functor $- \otimes_k^{\mathbf{L}} K: K^b(\proj \Lambda) \to K^b(\proj \Lambda_K)$ behaves nicely with respect to presilting and silting complexes.

\begin{lemma}(see also \cite[§2.4]{IK2024}) \label{lem:siltlift}
    Let $K:k$ be a field extension. If $P^\bullet \in K^b(\proj \Lambda)$ is a presilting object, then $P_K^\bullet \in K^b(\proj \Lambda_K)$ is a presilting object. Moreover, if $P^\bullet$ is a silting object, then $P_K^\bullet$ is a silting object.
\end{lemma}
\begin{proof}
    The first part of the statement follows directly from the adaptation of \cref{lem:Kas00.2.2} to this setting, where $\Ext^i(-,-)$ becomes $\Hom(-,-[i])$. If $P^\bullet \in K^b(\proj \Lambda)$ is silting, then $\Lambda \in \thick(P^\bullet) \subseteq K^b(\proj \Lambda)$. Since $- \otimes_k^{\mathbf{L}} K$ is a triangle functor, it follows that $\Lambda_K \in \thick(P_K^\bullet) \subseteq K^b(\proj \Lambda_K)$ which implies that $P_K^\bullet$ is silting in $K^b(\proj \Lambda_K)$. 
\end{proof}

The study of 2-term (pre)silting objects is the main focus of \cite{DerksenFei2015}. By \cite[Thm. 3.2]{AIR2014} the connection between 2-term (pre)silting objects and $\tau$-rigid pairs is given by the following mutually inverse maps:
\[
\begin{tikzcd}
    \stautilt \Lambda \arrow[r, "F", shift left] & \twosilt \Lambda{,} \arrow[l, "G", shift left]
\end{tikzcd}
\]
where $F(M,P) = (P^1 \oplus P \xrightarrow{(f\quad 0)} P^0)$ and $P_1 \xrightarrow{f} P_0 \to M$ is a minimal projective presentation of $M$, \\
and $G(P^\bullet) = (H^0(P),(P_1)'')$, where we decompose $P^\bullet =P^1 \xrightarrow{d} P^0$ as $P^\bullet =  (P^1)' \oplus (P^1)'' \xrightarrow{(d' \quad 0)} P^0$ with $d'$ right minimal, see \cite[Prop. 3.6(b)]{AIR2014}.
This relationship commutes with taking field extensions.

\begin{lemma}\label{cor:tautiltlift}
    Let $K:k$ be a field extension. If $(M,P)$ is a $\tau$-tilting pair in $\mods \Lambda$, then $(M_K, P_K)$ is a $\tau$-tilting pair in $\mods \Lambda_{K}$. {In particular}, if $M$ is a $\tau$-tilting module in $\mods \Lambda$, then $M_K$ is a $\tau$-tilting module in $\mods \Lambda_K$.
\end{lemma}
\begin{proof}
    Let $P^1 \xrightarrow{d} P^0 \to M \to 0$ be a minimal projective presentation of $M$. By \cref{lem:Kasradsoclift}(2), $P_K^1 \xrightarrow{d_K} P_K^0 \to M_K \to 0$ is a minimal projective presentation of $M_K$. Since $(M,P)$ is $\tau$-tilting, \cite[Thm. 3.2]{AIR2014} states that the corresponding 2-term complex $P^1 \oplus P \xrightarrow{(d \quad 0)} P^0$ is a silting object of $K^b(\proj \Lambda)$. By \cref{lem:siltlift}, the complex $P_K^1 \oplus P_K \xrightarrow{(d_K \quad 0)} P_K^0$ is a silting object in $K^b(\proj \Lambda_K)$. Moreover, its corresponding $\tau$-tilting pair is clearly $(M_K,P_K)$, completing the proof.
\end{proof}

As a corollary of the above, we obtain a similar result for classical tilting modules. Recall that $M \in \mods \Lambda$ is a \textit{tilting module} if $\projdim_\Lambda M \leq 1$, $\Ext_\Lambda^1(M,M)=0$ and $|M|=|\Lambda|$. 
\begin{corollary}\label{lem:tiltinglift}
	Let $K:k$ be a field extension and let $M \in \mods \Lambda$ be a tilting $\Lambda$-module. Then $M_K \in \mods \Lambda_K$ is a tilting $\Lambda_K$-module.
\end{corollary}
\begin{proof}
	By \cite[Cor. IV.4.7]{ARS1995}, if $\projdim_\Lambda M \leq 1$, then $\Ext_\Lambda^1(M,M)= 0$ if and only if $\Hom_\Lambda(M,\tau M)=0$. Consequently, tilting modules are exactly $\tau$-tilting modules of projective dimension at most 1. By \cite[Prop. 2.1(i)]{JL82} the $\Lambda_K$-module $M_K$ is such that
	\[ \projdim_{\Lambda_K} M_K = \projdim_\Lambda M \leq 1. \]
	Together with \cref{cor:tautiltlift}, this means that $M_K$ is a $\tau$-tilting $\Lambda_K$-module of projective dimension at most 1, and hence tilting as required. 
\end{proof}

Recall that a $k$-algebra $\Lambda$ is called \textit{tilted} provided that there is a hereditary $k$-algebra $\Gamma$ and a tilting $\Gamma$-module $T \in \mods \Gamma$ such that $\Lambda \simeq \End_\Gamma(T)^{\op}$. Much of the representation theory of tilted algebras can be understood by understanding the original hereditary algebra. We refer to \cite{Bon81, HappelRingelTilted1982,HandbookTilting} for thorough background on tilting theory.
\begin{lemma}\label{lem:tiltedlift}
	Let $K:k$ be a MacLane separable field extension. If $\Lambda$ is a tilted $k$-algebra, then $\Lambda_K$ is a tilted $K$-algebra.
\end{lemma}
\begin{proof}
	Assume that $\Lambda \simeq \End_{\Gamma}(T)^{\op}$ for some hereditary $k$-algebra $\Gamma$ and some tilting $\Gamma$-module $T \in \mods \Gamma$. Then 
	\begin{align*}
		\Lambda_K & \simeq (\End_{\Gamma}(T)^{\op})_K \\
		& \simeq ((\End_{\Gamma}(T))_K)^{\op} & \text{(by \cite[Lem. 2.3]{Kas00})} \\
		& \simeq \End_{\Gamma_K}(T_K)^{\op} & \text{(by \cref{lem:Kas00.2.2})}
	\end{align*}
	Since $\Gamma$ is hereditary, it follows that $\Gamma_K$ is hereditary since its global dimension is preserved by \Cref{thm:JL}\eqref{thm:JL1}. Moreover, by \cref{lem:tiltinglift}, $T_K$ is a tilting $\Gamma_K$-module, whence $\Lambda_K$ is a tilted $K$-algebra as required.
\end{proof}

Recall that a \textit{torsion class} $\calT \subseteq \mods \Lambda$ is a full subcategory which is closed under factor modules and extensions \cite{Dickson66}. Dually, a \textit{torsion-free class} $\calF \subseteq \mods \Lambda$ is a full subcategory which is closed under submodules and extensions. Denote the set of torsion classes (resp. torsion-free classes) of $\mods \Lambda$ by $\tors \Lambda$ (resp. $\torf \Lambda$). In fact, $\tors \Lambda$ and $\torf \Lambda$ form lattices under inclusion. For a full subcategory $\calX$ of $\mods \Lambda$, we define its \textit{right Hom-orthogonal} by
\[\calX^{\perp} \coloneqq \{Y\in \mods \Lambda \colon \Hom_{\Lambda}(X,Y)=0 \quad \forall X\in\calX \}.\]
Taking the Hom-orthogonal $(-)^{\perp}$ provides an anti-isomorphism from the lattice $\tors \Lambda$ to the lattice $\torf \Lambda$. We define the \textit{left Hom-orthogonal} ${}^\perp \calX$ dually. {For a module $M \in \mods \Lambda$, define the subcategories $M^\perp$ and ${}^\perp M$ of $\mods \Lambda$ as $(\add M)^\perp$ and ${}^\perp (\add M)$, respectively.}

There is an important subset of \textit{functorially finite} torsion classes (resp. torsion-free classes), denoted by $\ftors \Lambda \subseteq \tors \Lambda$ (resp. $\ftorf \Lambda \subseteq \torf \Lambda$). We recall the definition here. An arbitrary additive subcategory $\mathcal{X}$ of $\mods \Lambda$ is \textit{contravariantly finite} (resp \textit{covariantly finite}) if every $M \in \mods \Lambda$ admits a $\Lambda$-homomorphism $X \xrightarrow{\beta} M$ (resp. $M \xrightarrow{\alpha} X$) with the (nonuniversal) property that the induced map $\Hom_{\Lambda}(X{'},\beta)$ (resp. $\Hom_{\Lambda}(\alpha,X{'})$) is surjective for all $X'\in\mathcal{X}$. We call the $\Lambda$-homomorphism $\beta$ (resp. $\alpha$) a \textit{right $\mathcal{X}$-approximation} (resp. \textit{left $\mathcal{X}$-approximation}) of $M$. A \textit{functorially finite} subcategory of $\mods \Lambda$ is by definition a full subcategory that is both contravariantly and covariantly finite.

For a $\Lambda$-module $M$, let $\Fac M$ denote the smallest full subcategory of $\mods \Lambda$ containing $M$ which is closed under factor modules and extensions. Dually, let $\Sub M$ denote the smallest full subcategory of $\mods \Lambda$ containing $M$ which is closed under submodules and extensions.

\begin{theorem}\cite[Thm. 2.7]{AIR2014}\label{thm:AIRftorsbij}
	The map
	\[ \stautilt \Lambda \longrightarrow \ftors \Lambda \]
	which maps $(M,P) \in \stautilt \Lambda$ to $\Fac M$, is a bijection. Its inverse maps $\calT \in \ftors \Lambda$ to the unique support $\tau$-tilting pair $(\mathrm{P}(\calT), P)$, where $\mathrm{P}(\calT)$ is {the largest basic} $\Ext$-projective $\Lambda$-module in $\calT${, which is defined by \[\mathrm{P}(\calT)\coloneqq \bigoplus_{\substack{X\in \ind' \Lambda \\ \Ext^1_{\Lambda}({X},\calT)=0 }} X,\]}
    where $\ind' \Lambda$ denotes a set of indecomposable $\Lambda$-modules, in which exactly one representative of each isomorphism class occurs.
\end{theorem}

This relationship with torsion classes endows the set of support $\tau$-tilting pairs with a partial order given by $(M,P) \leq (N,Q)$ whenever $\Fac M \subseteq \Fac N$. In fact, this partial order can also be seen intrinsically via a process of mutation, see \cite[Sec. 2.3]{AIR2014}. 

\begin{lemma}\label{lem:posetlift}
    Let $K:k$ be a field extension and let $M$ and $N$ be $\tau$-rigid $\Lambda$-modules such that $\Fac M \subseteq \Fac N \subseteq \mods \Lambda$.  Then $\Fac M_K \subseteq \Fac N_K \subseteq \mods \Lambda_K$.
\end{lemma}
\begin{proof}
	Take $X \in \Fac M_K$. By definition, there is an epimorphism $M_K^{r} \twoheadrightarrow X \to 0$ in $\mods \Lambda_K$ for some $r \geq 1$. Since $\Fac M \subseteq \Fac N \subseteq \mods \Lambda$, it follows that in particular $M^r \in \Fac N$ which gives rise to an epimorphism $N^s \twoheadrightarrow M^r \to 0 $ in $\mods \Lambda$ for some $s \geq 1$. Applying the {(}exact{)} scalar extension functor then gives an epimorphism $N_K^s \twoheadrightarrow M_K^r \to 0$ in $\mods \Lambda_K$. In conclusion, there is a chain of epimorphisms $N_K^s \twoheadrightarrow M_K^r \twoheadrightarrow X${,} which yields $X \in \Fac N_K$ as required. 
\end{proof}

Consequently, we obtain the following relationship, which is to be compared with \cite[Thm. 2.14(a)]{IK2024}.
\begin{proposition}\label{cor:posetembed}
    Let $K:k$ be a field extension. For a $\Lambda_K$-module $N$, let $N^b$ denote a (choice of) basic direct summand of $N$ with $\add N = \add N^b$.
    We have an embedding of posets 
    \[ \stautilt \Lambda \to \stautilt \Lambda_K\]
    sending $(M,P)$ to $(M^b_K,P^b_K)$. {This map is well-defined up to isomorphism.}
\end{proposition}
\begin{proof}
    Since $-\otimes_k K$ is injective-on-objects by \cref{lem:adddeterminesbasic}(3), it follows from \cref{cor:tautiltlift} and \cref{lem:posetlift} that it can be regarded to send support $\tau$-tilting pairs to support $\tau$-tilting pairs. This is all we need to show.
\end{proof}

Similarly, it is possible to obtain the following result concerning functorially finite torsion classes.

\begin{corollary}\label{cor:posetembedftors}
    We have an embedding of posets
    \[ \ftors \Lambda \to \ftors \Lambda_K\]
    sending $\calT$ to $\calT_K \coloneqq \Fac M_K$, where $(M,P)$ is a support $\tau$-tilting pair such that $\calT = \Fac M$. Moreover, if $\calT=\Fac N$ for some $\tau$-rigid module $N$, then $\calT_K =  \Fac N_K$.
\end{corollary}
\begin{proof}
    The map $\ftors \Lambda \to \ftors \Lambda_K$ is defined as follows:
    \begin{equation}\label{eq:corposet}
    \begin{tikzcd}[column sep=65]
        \calT \arrow[r, "\textnormal{\cref{thm:AIRftorsbij}}","\textnormal{bijection}"',mapsto] & {(M,P)}  \arrow[r, "\textnormal{\cref{cor:posetembed}}", mapsto, "\textnormal{{injection}}"'] & {(M_K^b,P_K^b)} \arrow[r, "\textnormal{\cref{thm:AIRftorsbij}}","\textnormal{bijection}"',mapsto] & {\calT_K.}
    \end{tikzcd}
    \end{equation}
    All maps in \cref{eq:corposet} are order-preserving. By the results cited along the arrows in \eqref{eq:corposet}, it follows that this is indeed a well-defined embedding. {Let $M$ be $\tau$-rigid and assume $\calT = \Fac M$. Consider the image of $\calT$ under the inverse of the map in \cref{thm:AIRftorsbij}, and denote it by $(M',P')$. It is clear from the construction that $\Fac M = \Fac M'$.  By \cref{lem:posetlift}, we have $\Fac M_K = \Fac M_K'$, and therefore $\calT_K = \Fac M_K' = \Fac M_K$, as required.}
\end{proof}

We remark that $\calT_K$ is usually not equal to the full subcategory of $\mods \Lambda_K$ obtained by applying $- \otimes_k K$ to every object of $\calT \in \ftors \Lambda$. This fact illustrates why it is often preferable to study the behaviour of individual objects under base field extension, rather than the behaviour of subcategories.

\subsection{\texorpdfstring{$\tau^{-1}$}{tau-inverse}-tilting theory}

Dual to $\tau$-tilting theory, which uses the Auslander--Reiten translation $\tau$, there is $\tau^{-1}$-tilting theory using the inverse AR-translation $\tau^{-1}$. We include some explicit results in this subsection for later reference.

\begin{definition}
    Let $M \in \mods \Lambda$ and let $Q \in \inj \Lambda$. 
    \begin{enumerate}
        \item We say $M$ is \textit{$\tau^{-1}$-rigid} if $\Hom_{\Lambda}(\tau^{-1}_{\Lambda} M , M)=0$. If additionally $|M| = |\Lambda|$ then we say $M$ is \textit{$\tau^{-1}$-tilting}.
        \item We say that the pair $(M,Q)$ is \textit{$\tau^{-1}$-rigid} if $M$ is a $\tau^{-1}$-rigid $\Lambda$-module, $Q$ is an injective $\Lambda$-module and $\Hom_{\Lambda}(M,Q)=0$. If additionally $|M|+|Q| = |\Lambda|$ then we say that $(M,Q)$ is \textit{support $\tau^{-1}$-tilting}. The $\Lambda$-module $M$ is a \textit{support $\tau^{-1}$-tilting $\Lambda$-module}.
    \end{enumerate}
\end{definition}

Denote by $\stauinvtilt \Lambda$ the collection of basic support $\tau^{-1}$-tilting pairs for $\Lambda$. It is clear that the dual statements of \cref{lem:taurigidlift} and \cref{cor:tautiltlift} hold for $\tau^{-1}$-rigid and support $\tau^{-1}$-tilting pairs.

\begin{proposition} \cite[Discussion following Thm. 2.15]{AIR2014}\label{prop:AIRp2.15}
We have a bijection
\begin{equation}\label{eq:tauinversebij} H: \stautilt \Lambda \to \stauinvtilt \Lambda \end{equation}
given by $H(M,P) = (\tau M \oplus \nu P, \nu M_{\mathrm{pr}})${,} where $M_{\mathrm{pr}}$ is the largest projective direct summand of $M$. It fits into a commutative square
\begin{equation}\label{eq:tauinversebijsq}
\begin{tikzcd}
\stautilt \Lambda \arrow[r, "H"] \arrow[d, "\Fac"] & \stauinvtilt \Lambda \arrow[d, "\Sub"] \\
\ftors \Lambda \arrow[r, "(-)^{\perp}"]             & \ftorf \Lambda.   
\end{tikzcd}
\end{equation}
\end{proposition}

The interaction between the bijection in \cref{prop:AIRp2.15} and the extension of scalars is described in the following.

\begin{lemma}\label{lem:dualtaucommutes}
   Let $K:k$ be a field extension. The following square is commutative:
    \[ \begin{tikzcd}
        \stautilt \Lambda \arrow[r, "H"] \arrow[d, "- \otimes_k K", swap]\ &\stauinvtilt \Lambda  \arrow[d, "- \otimes_k K"]\\
        \stautilt \Lambda_K \arrow[r, "H"] &\stauinvtilt \Lambda_K.
    \end{tikzcd}\]
\end{lemma}
\begin{proof}
    Let $(M,P) \in \stautilt \Lambda$. By \cref{lem:tautranslatelift}, there is an isomorphism $(\tau_{\Lambda} M)_K \simeq \tau_{\Lambda_K} M_K$ and by \cref{cor:Nakfunctlift}{,} there is an isomorphism $(\nu_{\Lambda} P)_K \simeq \nu_{\Lambda_K} P_K$. Combining these two observations, there is an isomorphism $(\tau_\Lambda M \oplus \nu_\Lambda P)_K \simeq \tau_{\Lambda_K} M_K \oplus \nu_{\Lambda_K} P_K$ as the tensor functor is additive. Finally, \cref{lem:Kasprojs}(4) implies that $(M_{\mathrm{pr}})_K \simeq (M_K)_{\mathrm{pr}}$ and using \cref{cor:Nakfunctlift} again implies the desired result.
\end{proof}

\subsection{\texorpdfstring{$g$}{g}-vector fans and wall-and-chamber structures}
Many aspects of $\tau$-tilting theory of an algebra $\Lambda$ are encoded geometrically in its $g$-vector fan $\Sigma(\Lambda)$ \cite{DIJ2019}. Consider the Grothendieck group $K_0(\proj \Lambda)$ of the full subcategory $\proj \Lambda \subseteq \mods \Lambda$, whose basis is given by the isomorphism classes of indecomposable projective $\Lambda$-modules. The $g$-vector fan then lives the \textit{real Grothendieck group} $K_0(\proj \Lambda)_{\R} = K_0(\proj \Lambda) \otimes_{\Z} \R \simeq \R^{|\Lambda|}$, which is a finite-dimensional $\R$-vector space.

\begin{definition}
    Let $(M,P)$ {be a $\tau$-rigid pair in $\mods \Lambda$}. Let $P^{-1} \to P^0 \to M$ be the minimal projective presentation of $M$. The \textit{$g$-vector} of $M$ is defined to be 
    \[ g^M = [P^0] - [P^{-1}] \in K_0(\proj \Lambda)_\R.\]
    The \textit{$g$-vector cone} of $(M,P)$ is the nonnegative linear span
    \[ \calC_{(M,P)} = \spann_{\geq 0} \{ g^{M_1}, \dots, g^{M_t}, -g^{P_{t+1}}, \dots, -g^{P_{r}}\} \subseteq K_0(\proj \Lambda)_\R,\]
    where $M \simeq \bigoplus_{i=1}^t M_i$ and $P \simeq \bigoplus_{i=t+1}^r P_i$ are decompositions of $M$ and $P$ into indecomposable direct summands, respectively. Finally, the \textit{$g$-vector fan} $\Sigma(\Lambda)$ is given by the collection of $g$-vector cones
    \[ \Sigma(\Lambda) = \{ \calC_{(M,P)} : (M,P) \in \staurigid \Lambda \}. \]
\end{definition}

By considering containment of cones, we may regard the $g$-vector fan as a poset. As we have demonstrated in the previous sections, the $\tau$-tilting theory of $\Lambda$ maps
{nicely} into the $\tau$-tilting theory of $\Lambda_K$ for any field extension $K:k$. The following result establishes such a relationship for the respective $g$-vector fans. In \Cref{lemma:gvectorfan}, we suppress the subscript $\R$ and denote the real Grothendieck group simply by $K_0(\proj \Lambda)$.

\begin{theorem}\label{lemma:gvectorfan}
    Let $K:k$ be a field extension. Denote by $K_0(\proj \Lambda)_K \subseteq K_0(\proj \Lambda_K)$ the image of the embedding $- \otimes_k K: K_0(\proj \Lambda) \hookrightarrow K_0(\proj \Lambda_K)$ given by $[P] \mapsto [P_K]$. 
    There is an injective map of fans (as posets)
    \begin{equation}\label{eq:injectivemaponcones}
    \begin{aligned}
    \Sigma(\Lambda) &\hookrightarrow \Sigma(\Lambda_K)  \\
    \calC_{(M,P)} &\mapsto \calC_{(M_K,P_K)}.
    \end{aligned}
    \end{equation}
    which restricts to a (geometric) embedding of fans
    \begin{equation}\label{eq:gfanembedding}
    \begin{aligned} 
    \Sigma(\Lambda) &\hookrightarrow \Sigma(\Lambda_K) \cap K_0(\proj \Lambda)_K \\
    \calC_{(M,P)} \ni v&\mapsto v \otimes_k K \in \calC_{(M_K,P_K)}.
    \end{aligned}
    \end{equation}
\end{theorem}
\begin{proof}
    Since the $g$-vector fan $\Sigma(\Lambda)$ is completely determined by basic $\tau$-rigid pairs, we can restrict to those, see \cite[Lem. 6.12]{Kai23}.
    
    Let $(M_1,P_1)$ and $(M_2,P_2)$ be basic $\tau$-rigid pairs of $\mods \Lambda$. From \cref{lem:taurigidlift}{,} it follows that $((M_1)_K,(P_1)_K)$ and $((M_2)_K,(P_2)_K)$ are $\tau$-rigid pairs of $\mods \Lambda_K$. Therefore $\calC_{((M_1)_K, (P_1)_K)}, \calC_{(M_2)_K, (P_2)_K)} \in \Sigma(\Lambda_K)$, rendering our map of fans well-defined. Assume for a contradiction that $\calC_{((M_1)_K, (P_1)_K)} = \calC_{((M_2)_K, (P_2)_K)}$. By \cite[Cor. 6.7(2)]{DIJ2019}, see also \cite[Lem. 6.12, Thm. 6.13]{Kai23}, we have that $\add((M_1)_K) = \add((M_2)_K)$ and $\add((P_1)_K) = \add((P_2)_K)$. By \cref{lem:adddeterminesbasic}(5) $M_1 \simeq M_2$ and $P_1 \simeq P_2$ as required, hence the map in \cref{eq:injectivemaponcones} is well-defined and injective.

    Since $K_0(\proj \Lambda)_K$ is a subspace of $K_0(\proj \Lambda_K)$ it is clear that $\Sigma(\Lambda_K) \cap K_0(\proj \Lambda)_K$ is again a fan. Now, let $(M,P)$ be a basic $\tau$-rigid pair in $\mods \Lambda$. Take $v \in \calC_{(M,P)}$, then by definition we may write 
    \[ v= \sum_{i=1}^{l} \alpha_i g^{M_i} - \sum_{i=l}^{r} \alpha_i g^{P_i}\]
    for $|M|=l$ and $|M|+|P|=r$ and some $0 \leq \alpha_i \in \R$. From \cite[Prop. 6.6(c)]{DIJ2019} we know that $g^{(M_K)} = (g^M) \otimes_k K \in K_0(\proj \Lambda)_K$, from which we obtain
    \[ v \otimes_k K = \sum_{i=1}^l \alpha_i g^{(M_i)_K} - \sum_{i=l}^{r} \alpha_i g^{(P_i)_K}.\]
    It follows immediately that 
    \[ v \otimes_k K \in \calC_{(M_K,P_K)} \cap K_0(\proj \Lambda)_K. \]
    From \cref{lem:intersectlift} and \cite[Cor. 6.7(2)]{DIJ2019}, see also \cite[ Thm. 6.13]{Kai23}, it follows that the embedding is compatible with the fan structure, that is, intersections of cones are preserved.
\end{proof}

One important feature of the $g$-vector fan is that it embeds into the stability scattering diagram of \cite{Bri17}, and into its support, the so-called wall-and-chamber structure \cite{Asai2020, BST19}. The wall-and-chamber structure itself captures stability phenomena of modules as described shortly.
Through the Euler form $\langle -, ? \rangle : K_0(\proj \Lambda) \times K_0(\mods \Lambda) \to \Z$ we may view $\theta \in K_0(\proj \Lambda)$ as an element of $K_0(\mods \Lambda)^*$ via $\theta([M]) = \langle \theta, [M] \rangle$, see \cite{Asa21}. This establishes the connection between the $g$-vector fan and the following stability conditions. 

\begin{definition}\cite{King1994}
    Let $K_0(\mods \Lambda)_\R^*$ denote the dual $\R$-vector space of $K_0(\mods \Lambda)_\R$ and let $\theta \in K_0(\mods \Lambda)_\R^*$. We say that $M \in \mods \Lambda$ is \textit{$\theta$-semistable} (resp. \textit{$\theta$-stable}) if $\theta([M])=0$ and $\theta([L]) \leq 0$ (resp. $\theta([L])<0$) for all nonzero proper submodules $L$ of $M$. 
\end{definition}

Thus, each $g$-vector $\theta$ gives rise to a stability condition whose $\theta$-semistable objects are studied in \cite{Asa21,BST19, Yur}. Since $-\otimes_k K$ and $(-)|_\Lambda$ are both exact, they induce maps between Grothendieck groups and their duals
\begin{equation}\label{eq:Grothendieckmaps}
\begin{tikzcd}
    K_0(\mods \Lambda_K) \arrow[d, "{(-)|_\Lambda}", shift left] & K_0(\mods \Lambda_K)^* \arrow[d, "(- \otimes_k K)^*", shift left]\\
    K_0(\mods \Lambda) \arrow[u, shift left, "- \otimes_k K"]& K_0(\mods \Lambda)^* \arrow[u, shift left, "{(-)|_\Lambda^*}"]
\end{tikzcd}
\end{equation}
where for $\theta \in K_0(\mods \Lambda_K)^*$ we define $\theta_K^* \coloneqq (\theta \otimes_k K)^*$ by $\theta_K^*([M]) = \theta([M_K])$, and similarly we set $\theta|_\Lambda^*([M]) = \theta([M|_\Lambda])$. This alone is not enough to establish a well behaved relationship. However, recall from \cref{lem:nicerestriction} that when $K:k$ is a finite Galois extension, we have $(M_K)|_\Lambda \simeq M^{[K:k]}$. Using this observation we are able to prove the following. 

\begin{lemma}\label{prop:semistablelift}
    Let $K:k$ be a finite and separable. Let $M \in \mods \Lambda$ and let $\theta \in K_0(\mods \Lambda)_\R^*$. Then $M$ is $\theta$-semistable if and only if $M_K$ is $\theta|_\Lambda^*$-semistable. 
\end{lemma}
\begin{proof}
    $(\Rightarrow)$ Assume $M$ is $\theta$-semistable, then we obtain
    \[ \theta|_\Lambda^*([M_K]) = \theta([(M_K)|_\Lambda]) = \theta([M^{[K:k]}]) = [K:k] \cdot \theta([M])=0.\]
    Now let $L'$ be a submodule of $M_K$. 
    Since $K:k$ is separable and algebraic, it follows from \cite[Prop. 4.13]{Kas00} that $L'$ is a direct summand of $L_K$ for some $L \in \mods \Lambda$. Moreover, as $(-)|_\Lambda$ is exact, we have that $L$ is a submodule of $M$. We then get
    \[ \theta|_\Lambda^*([L']) = \theta([(L')|_\Lambda]) = \theta( [L^r]) = r \theta([L]) \leq 0,\]
    where $1 \leq r \leq [K:k]$.

    $(\Leftarrow)$ Assume $M_K$ is $\theta|_\Lambda^*$-semistable. Then we have
    \[ \theta([M]) = [K:k]^{-1} \theta([M^{[K:k]}]) = [K:k]^{-1} \theta( [(M_K)|_\Lambda]) = [K:k]^{-1} \theta|_\Lambda^* \theta([M_K]) = 0.\]
    Entirely analogously we have for a submodule $L$ of $M$ that 
    \[ \theta([L]) = [K:k]^{-1} \theta([L^{[K:k]}]) = [K:k]^{-1} \theta( [(L_K)|_\Lambda]) = [K:k]^{-1} \theta|_\Lambda^* \theta([L_K]) \leq 0,\]
    since $L_K$ is a submodule of $M_K$ as $- \otimes_k K$ is exact. 
\end{proof}

Now the stability spaces fit together nicely to form the \textit{wall-and-chamber structure}. More precisely, given a nonzero module $M \in \mods \Lambda$, define its stability space of $M$ to be
\[ \Theta_M = \{ \theta \in K_0(\mods \Lambda)_\R^*: M \text{ is $\theta$-semistable}\},\]
which we call the \textit{wall} associated to $M$. The collection of all walls defines the \textit{wall-and-chamber structure}, where a chamber is a part of $K_0(\mods \Lambda)_\R^*$ which does not lie in the closure of any wall.

\begin{theorem}\label{thm:WACembed}
    Let $K:k$ be a finite and separable extension. Then the wall-and-chamber structure of $\mods \Lambda$ embeds into that of $\mods \Lambda_K$ via the map $(-)|_\Lambda^*$ of \cref{eq:Grothendieckmaps}. 
\end{theorem}
\begin{proof}
    \cref{prop:semistablelift} states that the image of every wall in $K_0(\mods \Lambda)^*$ under $(-)|_\Lambda^*$ lies inside a wall of $K_0(\mods \Lambda_K)^*$. To see that there are no other walls in the image of $(-)|_\Lambda^*$, consider $M \in \mods \Lambda_K$ such that $M$ is $\theta|_\Lambda^*$-semistable for some $\theta \in K_0(\mods \Lambda)_\R^*$. Since $K:k$ is separable and algebraic{,} we know that $M \in \add N_K$ for some $N \in \mods \Lambda$ by \cite[Prop. 4.13]{Kas00}. We conclude the proof by showing that $N_K$ is also $\theta|_\Lambda^*$-semistable, hence any wall in the image of $(-)|_\Lambda^*$ lies inside the image of a wall of $K_0(\mods \Lambda)$ under $(-)|_\Lambda^*$. We have 
    \begin{align*}
         \theta|_\Lambda^*([N_K]) &= \theta([(N_K)|_\Lambda]) \\
         &= \theta([N^{[K:k]}]) \\
         &= [K:k] \theta([N]) \\
         &= \frac{[K:k]}{r} \theta([N^r]) \\
         &= \frac{[K:k]}{r}\theta([M|\Lambda]) \\
         &= \frac{[K:k]}{r}\theta|_\Lambda^*([M]) \\
         &= 0,
    \end{align*}
    where $M|_\Lambda \simeq N^r$ for some $1 \leq r \leq [K:k]$.
\end{proof}

Comparing \cref{lemma:gvectorfan} with \cref{thm:WACembed} shows that the $g$-vector fan, so in some sense the $\tau$-tilting theory, behaves well for all field extensions. Outside of the $g$-vector fan, the stability spaces may behave worse. Only with the additional assumption of having a finite {and separable} extension is it possible to gain control outside the $g$-vector fan. 

\subsection{\texorpdfstring{$\tau$}{tau}-perpendicular intervals}\label{sec:tauperp}
An interval $[\calU, \calT] \subseteq \tors \Lambda$ is called \textit{$\tau$-perpendicular} if $\calU = \Fac M$ and $\calT = {}^\perp \tau M \cap P^\perp$ for some $\tau$-rigid pair $(M,P)$. In this case{,} we also denote this interval by $[\calU_{(M,P)}, \calT_{(M,P)}]$. Via the bijection \cref{thm:AIRftorsbij}, the torsion class $\calU_{(M,P)}$ corresponds with the so-called \textit{co-Bongartz completion} $(M^-,P^-) \in \stautilt \Lambda$ and the torsion class $\calT_{(M,P)}$ corresponds with the so-called \textit{Bongartz completion} $(M^+, P) \in \stautilt \Lambda$, see also \cite[Thm. 4.4]{DIRRT17} (the second component of the Bongartz completion stays invariant when passing from $(M,P)$ to its Bongartz completion $(M^+, P)$). By analogy{,} an interval of  $\stautilt \Lambda$ is called $\tau$-perpendicular if it is of the form $[(M^-,P^-), (M^+,P)] \subseteq \stautilt \Lambda$ for some $\tau$-rigid pair $(M,P)$. We begin with the following observation, which is similar to \cref{lem:posetlift}.

\begin{lemma}\label{lem:cobongartzcommute}
    Taking co-Bongartz completions commutes with base field extension. More precisely, given
     a field extension {$K:k$} and $(M,P) \in \staurigid \Lambda$, then 
    \[((M_K)^-,(P_K)^-) = ((M^-)_K, (P^-)_K). \]
  \end{lemma}
\begin{proof}
    {By definition, the co-Bongartz completion of $(M,P)$ is the uniquely determined basic support $\tau$-tilting pair $(M^-,P^-)$ with the property that $\Fac(M^-) = \Fac(M)$. }
    {By \cref{lem:taurigidlift}, $(M_K,P_K)$ is a $\tau$-rigid module in $\mods \Lambda_K$ and thus $((M_K)^-,(P_K)^-)$ is similarly defined as the unique basic $\tau$-tilting pair in $\mods \Lambda_K$ such that $\Fac((M_K)^-) = \Fac(M_K)$. However, from \Cref{cor:posetembedftors} we know that $\Fac(M^-) = \Fac(M)$ implies $\Fac((M^-)_K) = \Fac(M_K)$. This implies that $\Fac((M_K)^-) = \Fac(M_K) = \Fac((M^-)_K)$. Finally, since $(M^-,P^-)$ is a $\tau$-tilting pair in $\mods \Lambda$, it follows that $((M^-)_K,(P^-)_K)$ is a $\tau$-tilting pair in $\mods \Lambda_K$ by \cref{cor:tautiltlift}. Thus, \cref{thm:AIRftorsbij} implies that $((M_K)^-,(P_K)^-) = ((M^-)_K,(P^-)_K)$.}
\end{proof}

With a little more work, we are also able to obtain a dual statement.

\begin{lemma}\label{lem:bongartzcommute}
    Taking Bongartz completions commutes with base field extension. More precisely, let $K:k$ be a field extension and $(M,P) \in \staurigid \Lambda$. Then 
    \[((M_K)^+,P_K) = ((M^+)_K, P_K). \]
  \end{lemma}
  \begin{proof}
    {Recall from \cref{prop:AIRp2.15} that $H(M,P)$ is defined by $ (\tau M \oplus \nu P, M_\textnormal{pr})$.} {By definition, the Bongartz completion of $(M,P)$ is the uniquely determined basic support $\tau$-tilting pair $(M^+,P)$ in $\mods \Lambda$ with the property that $\Fac(M^+) = {^{\perp}\tau_\Lambda M} \cap P^{\perp}$.}
    Since ${^{\perp}\tau_\Lambda M} \cap P^{\perp} = {^{\perp}(\tau_{\Lambda} M \oplus \nu_{\Lambda} P)}$, this is equivalent to saying $\Fac(M^+) = {}^\perp (\tau_\Lambda M \oplus \nu_\Lambda P)$, and thus equivalent to $\Fac(M^+)^\perp = \Sub(\tau_\Lambda M \oplus \nu_\Lambda P)$ since $\tau_\Lambda M \oplus \nu_\Lambda P$ is $\tau^{-1}$-rigid in $\mods \Lambda$. Because torsion classes uniquely determine their corresponding torsion-free classes, we equivalently have that $(M^+,P)$ is the unique $\tau$-tilting pair in $\mods \Lambda$ such that $\Fac(M^+)^\perp = \Sub(\tau_\Lambda M \oplus \nu_\Lambda P)$. It clearly follows that $\Sub(\tau_\Lambda M \oplus \nu_\Lambda P) = \Sub(\tau_\Lambda (M^+) \oplus \nu_\Lambda P)$, so \cref{lem:dualtaucommutes} and the dual of \cref{cor:posetembedftors} imply that 
    \begin{equation} \label{eq:bongartzcommute} \Sub(\tau_{\Lambda_K} M_K \oplus \nu_{\Lambda_K} P_K) = \Sub(\tau_{\Lambda_K} (M^+)_K \oplus \nu_{\Lambda_K} P_K).\end{equation} 
    By \cref{lem:taurigidlift}, $(M_K,P_K)$ is a $\tau$-rigid pair in $\mods \Lambda_K$, whence $((M_K)^+,P_K)$ is the unique basic $\tau$-tilting pair in $\mods \Lambda_K$ such that $\Fac((M_K)^+)^\perp = \Sub(\tau_{\Lambda_K} M_K \oplus \nu_{\Lambda_K} P_K)$. 
    However, since $H(M^+,P)$ is a $\tau^{-1}$-tilting pair in $\mods \Lambda$, the dual of \cref{cor:tautiltlift} implies that $(H(M^+,P))_K$ is a $\tau^{-1}$-tilting pair in $\mods \Lambda$. By \cref{eq:bongartzcommute}, it also corresponds to the torsion-free class $\Sub(\tau_{\Lambda_K} M_K \oplus \nu_{\Lambda_K} P_K)$ under the dual of \cref{thm:AIRftorsbij}. This implies $(H(M^+,P))_K = H((M_K)^+,P_K)$. By \cref{lem:dualtaucommutes} we thus have
    \[ H((M^+)_K, P_K) = (H(M^+,P))_K = H((M_K)^+,P_K),\]
    so that applying $H^{-1}$ to both sides yields $((M^+)_K,P_K) = ((M_K)^+,P_K)$ as required.
  \end{proof}

It is important to remark that the intervals $[\calU_{(M,P)}, \calT_{(M,P)}] \subseteq \tors \Lambda$ and $[(M^-,P^-), (M^+,P)]$ are not generally isomorphic as lattices because there may exist torsion classes in the interval which are not functorially finite. Nonetheless, there is a bijection between $\tau$-perpendicular intervals of $\stautilt \Lambda$ and $\tau$-perpendicular intervals of $\tors \Lambda$. Denote by $\tauint(\tors \Lambda)$ the collection of $\tau$-perpendicular intervals of $\tors A$ and view it as a poset such that $[\calU, \calT] \leq [\calX, \calY ]$ whenever $[\calX, \calY] \subseteq [\calU, \calT]$.

\begin{proposition}\label{cor:tauperpitvlift}
	Let $K:k$ be a field extension. There is a well-defined map 
	\begin{equation}\label{eq:welldefineditvmap} 
- \otimes_k K: \tauint( \tors \Lambda) \to \tauint(\tors \Lambda_K)
\end{equation}
given component-wise by $- \otimes_k K$. More precisely, let $(M,P)$ be a $\tau$-tilting pair in $\mods \Lambda$. Then applying the scalar extension functor $- \otimes_k K$ component-wise to the two boundary $\tau$-tilting pairs in $\Lambda$ of the $\tau$-perpendicular interval $[(M^-,P^-), (M^+, P)] \subseteq \stautilt \Lambda$ gives rise to the $\tau$-perpendicular interval $[(M_K^-, P_K^-), (M_K^+,P_K)] \subseteq \tors \Lambda_K$.
\end{proposition}
\begin{proof}
	This follows from \cref{lem:cobongartzcommute}, \cref{lem:bongartzcommute} and the fact that the order of support $\tau$-tilting pairs are preserved by \cref{lem:posetlift}.
\end{proof}

\subsection{Ordered modules and \texorpdfstring{$\tau$}{tau}-exceptional sequences}
Core techniques of $\tau$-tilting theory were applied in \cite{BM18t} to introduce $\tau$-exceptional sequences. These sequences generalise exceptional sequences of bricks over hereditary algebras \cite{cbw92,ringel_exceptional}. Among other things, they are motivated by the fact that complete sequences of maximal length always exist. Since their introduction, $\tau$-exceptional sequences have received much attention \cite{BarnardHanson2022exc,BHM2024,BHM2025,BKT2025,BM2023,HansonThomas2024,Msa2022exc,MT2020, Non25, Nonis2025mut, Terland2025}. 

\begin{definition}\cite[Def. 1.3]{BM18t}\label{def:tauexcep}
    Let $\ind \Lambda$ denote the subcategory of $\Lambda$ spanned by indecomposable $\Lambda$-modules.
    Denote by $C(\Lambda)$ the disjoint union of two copies of $\ind \Lambda$. To distinguish the two copies, we consider objects in the first category to be ``\textit{unsigned},'' or just $\Lambda$-modules in their own right, and the objects of the second category to be (formally) \textit{signed}. We denote a signed $\Lambda$-module by $\overline{M}$, where $M \in \mods \Lambda$. A sequence of objects in $C(\Lambda)$ given by
    \[(M_1 ,M_2,\dots,M_t)\] 
    is called a \textit{signed $\tau$-exceptional sequence} of length $t$ in $\mods \Lambda$ if:
    \begin{enumerate}
        \item $M_t$ is an unsigned $\Lambda$-module and $\tau$-rigid {in $\mods \Lambda$}, or $M_t$ is a signed $\Lambda$-module whose underlying $\Lambda$-module is {a} projective {object in $\mods \Lambda$}; and
        \item if $t > 1$, then $(M_1 ,\dots,M_{t-1})$ is a $\tau$-exceptional sequence of length $t-1$ in {the abelian subcategory} $M_{t}^\perp \cap {}^\perp \tau M_{t} {\subseteq \mods \Lambda}$, see also \cref{thm:widesubcatlatticeiso}.
    \end{enumerate}
    Such a sequence is called a\textit{(n unsigned) $\tau$-exceptional sequence} if each $M_i$ is an unsigned $\Lambda$-module. It is {called} a \textit{signed $\tau$-exceptional sequence} otherwise.
    A signed or unsigned $\tau$-exceptional sequences is \textit{complete} if $t=|\Lambda|$.
\end{definition}

{In view of (2) above, let $(M,P)$ be a $\tau$-rigid pair in $\mods \Lambda$.} {R}ecall that $M^\perp \cap {}^\perp \tau M \cap P^\perp$ is equivalent to $\mods \Gamma$ for some finite-dimensional algebra $\Gamma$ so that it has its own (relative) Auslander--Reiten translation and projective {objects} \cite{Jas15,DIRRT17}, see also \cref{thm:widesubcatlatticeiso}. As a consequence, the unsigned $\Lambda$-modules $M_1, \dots, M_{t-1}$ in \cref{def:tauexcep} are not necessarily $\tau$-rigid in $\mods \Lambda$ or signed projective $\Lambda$-modules. 

In some cases it is preferable to not use the recursive definition above and instead apply the following theorem which helps to index the signed $\tau$-exceptional sequences. {Let $(M,P)$ be a $\tau$-rigid pair. An \emph{ordered decomposition} of $(M,P)$ is given by an ordered {direct} sum {$X_1 \oplus \dots \oplus X_{r}$, where each $X_i$ is a direct summand of $M$ or a signed direct summand of $P$, such that the underlying module is isomorphic to $M \oplus P$.} 
We say that the ordered decomposition is an \textit{ordered decomposition into indecomposables} if each {$X_i$} is either an indecomposable $\Lambda$-module or an signed direct summand $\overline{P}_i$ of $\overline{P}$, where $P_i$ is indecomposable.} 

\begin{theorem}\cite[Thm. 1.4]{BM18t}\label{thm:tauexcpbij}
     There is a bijection
    \[ \left\{ \begin{varwidth}{20em} \begin{center} ordered decompositions of basic $\tau$-rigid pairs into indecomposables \end{center} \end{varwidth} \right\} \longleftrightarrow \left\{ \begin{varwidth}{15em} \begin{center} signed $\tau$-exceptional sequences \end{center} \end{varwidth} \right\}\]
    which restricts to support $\tau$-tilting pairs and complete signed $\tau$-exceptional sequences. 
\end{theorem}

The idea of working with ordered decompositions instead of (signed) $\tau$-exceptional sequences was applied in \cite{BKT2025} to gain new insights into the mutation of signed $\tau$-exceptional sequences as introduced in \cite{BHM2024}. Moreover, it allows us to lift signed $\tau$-exceptional sequences under base field extension.

\begin{lemma}\label{lem:signedexcplift}
    Let $K:k$ be a field extension. Then there is an injective map
    \[ \left\{ \begin{varwidth}{15em} \begin{center} complete signed $\tau$-exceptional sequences {in $\mods \Lambda$} \end{center} \end{varwidth} \right\} \hookrightarrow \left\{ \begin{varwidth}{15em} \begin{center} complete signed $\tau$-exceptional sequences {in $\mods \Lambda_K$} \end{center} \end{varwidth} \right\}\]
    which factors through ordered decompositions of basic support $\tau$-tilting pairs.
\end{lemma}
\begin{proof}
    {By \cref{cor:posetembed}}, there is an injective map from basic $\tau$-tilting pairs to basic $\tau$-tilting pairs {in $\mods \Lambda$}. Since $- \otimes_k K$ is additive, any ordered decomposition of a $\tau$-tilting pair {in $\mods \Lambda$} gives rise to an ordered decomposition of the resulting $\tau$-tilting pair {in $\mods \Lambda_K$} into direct summands (which are not necessarily indecomposable). Fixing any choice of ordered decomposition for each such decomposable direct summand and combining this choice with \cref{thm:tauexcpbij} yields the desired injective map. 
\end{proof}

In fact, it is possible to restrict the bijection in \cref{thm:tauexcpbij} to (unsigned) $\tau$-exceptional sequences. This was shown in the work of \cite{MT2020}, where a close relationship between $\tau$-exceptional sequences and stratifying systems \cite{ErdmannSaenz2003} was established. We generalise their ideas, following \cite{Hanson2025}.

\begin{definition}
    Let $M \in \mods \Lambda$ be a basic $\Lambda$-module and consider an ordered decomposition $M \simeq M_1 \oplus \dots \oplus M_r$. 
    \begin{enumerate}
        \item We call $M_1 \oplus \dots \oplus M_t$ a \textit{weakly TF-preordered decomposition} if for every $1 \leq i \leq t$ and every indecomposable direct summand $M'$ of $M_i$ we have $M' \not \in \Fac \left( \bigoplus_{j = i+1}^t M_j\right)$.
        \item A weakly TF-preordered decomposition in which all direct summands are indecomposable is called \textit{TF-ordered}. 
    \end{enumerate}
\end{definition}

The importance of TF-ordered decompositions becomes apparent in the following result.

\begin{theorem}\cite[Thm. 5.1]{MT2020}\label{thm:TForderedbij}
    The bijection in \cref{thm:tauexcpbij} restricts to a bijection between TF-ordered $\tau$-rigid modules and (unsigned) $\tau$-exceptional sequences.
\end{theorem}

In contrast to signed $\tau$-exceptional sequences, the additional (TF-ordered) structure that (unsigned) $\tau$-exceptional sequences require makes it more difficult to understand their behaviour under base field extension. In particular, it seems necessary to impose stronger assumptions on the field extension.

\begin{proposition}\label{lem:TFordertoTFpreorder}
    Let $K:k$ be a finite extension and let $M \simeq M_1 \oplus \dots \oplus M_r$ be a TF-ordered decomposition of a basic module $M \in \mods \Lambda$. {Moreover, let $(M_i)_K^b$ denote a (choice of) basic direct summand of $(M_i)_K$ such that $\add ((M_i)_K) = \add((M_i)_K^b)$ for all $1 \leq i \leq r$.} Then: 
    \begin{enumerate}
        \item $(M_1)_K \oplus \dots \oplus (M_r)_K$ is a weakly TF-preordered decomposition;
        \item {There exists exists an ordered decomposition $(M_i)_K^b \cong \widehat{M}_{i,1} \oplus \dots \oplus \widehat{M}_{i,n_i}$ into indecomposables for each $1 \leq i \leq r$ such that
    \[\widehat{M}_{1,1} \oplus \dots \oplus \widehat{M}_{1,n_1} \oplus \widehat{M}_{2, 1} \oplus \dots \oplus \widehat{M}_{r-1, n_{r-1}} \oplus \widehat{M}_{r,1} \oplus \dots \oplus \widehat{M}_{r,n_r} \]
    is a TF-ordered decomposition of $M_K$.}
    \end{enumerate}
\end{proposition}
\begin{proof}
    It is sufficient to consider the case where $M \oplus N$ is a TF-ordered decomposition of a basic $\Lambda$-module. The general case follows in the same way since the scalar extension and scalar restriction functors are additive. Assume $M \oplus N$ is TF-ordered, that is, $M \not \in \Fac N$.
    \begin{enumerate}
        \item Assume that an indecomposable $\Lambda_K$-module $M' \in \add(M_K)$ is generated by $N_K$. By definition this means that there is an exact sequence $N_K^{r} \to M' \to 0$
    in $\mods \Lambda_K$ which gets mapped to the exact sequence $ N^{[K:k] \cdot r} \to M^s \to 0$
    by the exact restriction functor, as a result of \cref{lem:nicerestriction}. This implies that $M \in \Fac N$, a contradiction. 
    \item {Let $M'$ be an indecomposable direct summand of $M_K^b$. By (1), we have that $M' \not \in \Fac(N_K)$. Consider the torsion pair $(\Fac(N_K), (N_K)^\perp)$ of $\mods \Lambda_K$. Recall that it induces the \textit{torsion-free functor} $f_{N_K}$ from $\mods \Lambda_K$ into the torsion-free class $(N_K)^\perp$. Since $M_K^b \not \in \Fac(N_K)$ the module $f_{N_K} (M')$ is a non-zero $\Lambda_K$-module. By \cite[Prop. 5.6]{BM18t}, the module $f_{N_K}(M')$ is $\tau$-rigid in $(N_K)^\perp \cap {}^\perp (\tau_\Lambda N_K) \subseteq \mods \Lambda_K$.

    Let $M_K^b \cong \bigoplus_{j=1}^{|M_K^b|} (M_K)_j$ be a decomposition of $M_K^b$ into pairwise nonisomorphic indecomposable direct summands. By \cite[Prop. 3.2]{MT2020}, the module
    \[ \bigoplus_{j=1}^{|M_K^b|} f_{N_K}((M_K)_j), \]
    which is $\tau$-rigid in the abelian subcategory $(N_K)^\perp \cap {}^\perp (\tau_\Lambda N_K)$ of $\mods \Lambda_K$ but may fail to be $\tau$-rigid in $\mods \Lambda$, admits a TF-ordered decomposition}
    \begin{equation}\label{eq:TFordered1} f_{N_K}(\widehat{M}_1) \oplus \dots \oplus f_{N_K}(\widehat{M}_{|M_K^b|}),\end{equation}
    {in the subcategory $(N_K)^\perp \cap {}^\perp (\tau_\Lambda N_K)$ of $\mods \Lambda_K$. That is, there exists a decomposition such that}
    \[ f_{N_K}(\widehat{M}_\ell) \not \in \Fac_{(N_K)^\perp \cap {}^\perp (\tau_\Lambda N_K)} \left( \bigoplus_{j > \ell}^{|M_K^b|} f_{N_K}(\widehat{M}_j) \right) \]
    {for each $1 \leq \ell < |M_K^b|$. It is important to highlight that $\Fac_{(N_K)^\perp \cap {}^\perp (\tau_\Lambda N_K)} \left( \bigoplus_{j > \ell}^{|M_K^b|} f_{N_K}(\widehat{M}_j) \right)$ is a torsion class of the abelian subcategory $(N_K)^\perp \cap {}^\perp (\tau_\Lambda N_K)$ of $\mods \Lambda$ but may fail to be a torsion class of $\mods \Lambda$.

    By \cite[Prop. 3.2]{MT2020}, the $\tau$-rigid module $N_K^b$ of $\mods \Lambda_K$ similarly admits a TF-ordered decomposition $N_K^b \cong \widehat{N}_1 \oplus \dots \oplus \widehat{N}_{|N_K^b|}$ in $\mods \Lambda$. We now claim that
    \begin{equation}\label{eq:TFordered2} \widehat{M}_1 \oplus \dots \widehat{M}_{|M_K^b|} \oplus \widehat{N}_1 \oplus \dots \oplus \widehat{N}_{|N_K^b|}\end{equation}
    is a TF-ordered decomposition of $M_K^b \oplus N_K^b$ in $\mods \Lambda$. Indeed, $\widehat{N}_i \not \in \Fac\left( \bigoplus_{j > i}^{|M_K^b|} \widehat{N}_j \right)$ for all $1 \leq i \leq |M_K^b|$ by construction. Moreover, if 
    \[ \widehat{M}_i \in \Fac \left( \widehat{M}_{i+1} \oplus \dots \widehat{M}_{|M_K^b|} \oplus N_K \right)\]
    held, then there exists an epimorphism 
    \[ \left( \widehat{M}_{i+1} \oplus \dots \widehat{M}_{|M_K^b|} \oplus N_K \right)^s \to \widehat{M}_i \to 0.\]
    However, applying the additive 
    torsion-free functor $f_{N_K}$ gives the following commutative diagram}
    \[
    \begin{tikzcd}
        \left( \widehat{M}_{i+1} \oplus \dots \widehat{M}_{|M_K^b|} \oplus N_K \right)^s \arrow[d, "g_1", twoheadrightarrow] \arrow[r, "h_1", twoheadrightarrow] & \widehat{M}_i \arrow[r] \arrow[d, "g_2", twoheadrightarrow] & 0 \\
        \left(\bigoplus_{j=i+1}^{|M_K^b|} f_{N_K}(\widehat{M}_{j})^s \right) \arrow[r, dashed, twoheadrightarrow, "f_{N_K}(h_1)"]\arrow[d] & f_{N_K}(\widehat{M}_i) \arrow[r] \arrow[d]& 0 \\
        0 & 0 
    \end{tikzcd}
    \]
    Note that $f_{N_K}(h_1)$ is a nonzero epimorphism, since $f_{N_K}(h_1) \circ g_1 = g_2 \circ h_1$ and $g_2 \circ h_1$ is a nonzero epimorphism. This yields a contradiction to \cref{eq:TFordered1} being TF-ordered. Therefore, the decomposition of \cref{eq:TFordered2} is a TF-ordered decomposition of $M_K^b \oplus N_K^b$ in $\mods \Lambda$. 
    \end{enumerate}
\end{proof}

{
\begin{corollary}\label{cor:tauexceptionallift}
    Let $K:k$ be a finite field extension. Then there is an injective map
    \[ \left\{ \begin{varwidth}{15em} \begin{center} complete $\tau$-exceptional sequences in $\mods \Lambda$ \end{center} \end{varwidth} \right\} \hookrightarrow \left\{ \begin{varwidth}{15em} \begin{center} complete $\tau$-exceptional sequences in $\mods \Lambda_K$ \end{center} \end{varwidth} \right\}\]
    which factors through TF-ordered decompositions of basic $\tau$-tilting modules.
\end{corollary}
\begin{proof}
    {By \cref{cor:posetembed}}, there is an injective map from basic $\tau$-tilting modules in $\mods \Lambda$ to basic $\tau$-tilting modules in $\mods \Lambda_K$. By \cref{lem:TFordertoTFpreorder}, any TF-ordered decomposition of a $\tau$-tilting module in $\mods \Lambda$ gives rise to an weakly TF-preordered decomposition of the resulting $\tau$-tilting module in $\mods \Lambda_K$ into direct summands (which are not necessarily indecomposable). By \cref{lem:TFordertoTFpreorder}, there exists at least one possible choice of refining the TF-preordered decomposition. 
\end{proof}
}

\section{Bricks and semibricks}

In this section we are interested in a certain class of indecomposable modules called bricks. Formally, a module $M \in \mods \Lambda$ is called a \textit{brick} if its endomorphism $k$-algebra $\End_\Lambda(M)$ is a division $k$-algebra. A collection of pairwise $\Hom$-orthogonal bricks is called a \textit{semibrick}. In particular, (semi)simple modules are particular examples of (semi)bricks.
Denote by $\brick \Lambda$ (resp. $\sbrick \Lambda$) the set of bricks (resp. semibricks) in $\mods \Lambda$.

{The following lemma gives a sufficient condition for a brick to remain indecomposable after applying the scalar extension functor.}

\begin{lemma}\label{lem:indlifting}
    {Let} {$\Gamma$} be a (not necessarily finite-dimensional) $k$-algebra and let $M$ be a finite-dimensional ${\Gamma}$-module. If $M$ is such that $\End_{{\Gamma}}(M) \simeq k$, then for any field extension $K:k$, we have $\End_{{\Gamma}_K}(M_K) \simeq K$ and hence $M_K$ is indecomposable.
\end{lemma}
\begin{proof}
    The result essentially follows from the well-known fact that if $\dim_k \End_{{\Gamma}}(M) = 1$ then $M$ is indecomposable. This is easy to see using the contrapositive: Given a nontrivial decomposition $M \simeq M_1 \oplus M_2$ of $M$, then 
    \[ \dim_k \End_{{\Gamma}}(M) \geq \dim_k \End_{{\Gamma}}(M_1) + \dim_k \End_\Lambda(M_2) \geq 1+1 = 2.\]
    Hence if $\End_{{\Gamma}}(M) \simeq k$, then it follows from \cref{lem:Kas00.2.2} that $\End_{{{\Gamma}}_K}(M_K) \simeq K$ (see \cite[Lem. 7.4]{Lam2001} for infinite-dimensional algebras). In other words, $\dim_K \End_{{{\Gamma}}_K}(M_K) = 1$, hence $M_K$ is indecomposable.
\end{proof}

An example of a field extension $K:k$ such that $- \otimes_k K$ preserves indecomposability is also considered in \cite[Thm. 3.5]{JL82}. As an almost trivial first property, we show that it is sometimes possible to lift semisimple modules.

\begin{lemma}\label{lem:semisimplelift}
    Let $K:k$ be a MacLane separable field extension and let $\Lambda$ be a finite-dimensional $k$-algebra. Consider a (semi)simple $\Lambda$-module $S$. Then $S_K$ is a semisimple $\Lambda_K$-module.
\end{lemma}
\begin{proof}
    We recall that a module $M$ is semisimple if and only if its radical $\rad M$ is zero. By \cref{lem:Kasradsoclift}(2), we have $0 = (\rad S) \otimes_k K \simeq \rad S_K$, whence $S_K$ is also semisimple.
\end{proof}

The following example occurs frequently in the literature, and it shows that the assumption of MacLane separability is necessary in \Cref{lem:semisimplelift}.

\begin{example}[see {\cite{Jac53},\cite[Rmk. 3.4]{Kas00}}]\label{exmp:nonsemibrick}
    Let $k$ denote the field $\mathbb{F}_2(t)$ of rational functions over the Galois field with two elements. Consider the non(Maclane )separable finite field extension $K:k$, where $K=k(\sqrt[2]{t})$, and let $\Lambda$ denote $K$ regarded as a $k$-algebra. Since $\Lambda$ is also a field, it is a simple module over itself. Consequently, 
    \[ \Lambda\otimes_k K = k[X]/(X-\sqrt[2]{t})^2 \otimes_k K = K[X]/(X-\sqrt[2]{t})^2.\] 
    As a result, the simple $\Lambda$-module $\Lambda$ does not remain semisimple when extending the scalars to $K$; the endomorphism algebra of $\Lambda\otimes_k K$ as a $\Lambda\otimes_k K$-module contains a nilpotent element, namely $X-\sqrt[2]{t}$.
\end{example}

{Given a full subcategory $\calC \subseteq \mods \Lambda$, denote by $\ind(\calC)$ the set of isomorphism classes of indecomposable direct summands of modules $M$ in $\calC$.} Now, consider the following generalisation of \cref{lem:semisimplelift} to all bricks. Due to its increased generality it requires a stronger assumption on the base field. 

\begin{proposition}\label{lem:semibrickslift}
    Let $k$ be a perfect field and let $K:k$ be a field extension.
    Then the extension of scalars functor induces an injective map from bricks in $\mods \Lambda$ to semibricks in $\mods \Lambda_K$ which gives an injective map:
    \[ \ind(- \otimes_k K): \sbrick \Lambda \to \sbrick \Lambda_K\]
    Moreover, if $B \in \brick \Lambda$ satisfies $\End_\Lambda(B) \simeq k$, then $B_K \in \brick \Lambda_K$.
\end{proposition}
\begin{proof}
    The moreover part follows immediately from \cref{lem:indlifting}.

    It is sufficient to show that $\ind(- \otimes_k K)$ defines an injective map $\brick \Lambda \hookrightarrow \sbrick \Lambda_K$. Indeed, $\Hom$-orthogonality between different isomorphism classes of modules in a semibrick in $\mods \Lambda$ is preserved by \cref{lem:Kas00.2.2}. Therefore, take $B \in \brick \Lambda$. By definition, we have that $\End_\Lambda(B)$ is a division $k$-algebra. In particular, it is reduced, meaning that there are no nonzero nilpotent elements. Since $k$ is perfect, it follows from \cite[Thm. V.15.5.3(d)]{BourbakiAlgII} that $\End_\Lambda(B) \otimes_k K \simeq \End_{\Lambda_K}(B_K)$ is a reduced $K$-algebra. 

    By \Cref{defprop:perfect}(3), the field extension $K:k$ is separable, and thus also MacLane separable. By \Cref{thm:JL}\eqref{thm:JL1}, the global dimension of the $K$-algebra $\End_\Lambda(B) \otimes_k K$ is thus the same as that of the $k$-algebra $\End_\Lambda(B)$. Since the latter is $0$, we deduce that $\End_\Lambda(B) \otimes_k K$ is also semisimple.
    Then, the Wedderburn--Artin Theorem implies that $\End_\Lambda(B) \otimes_k K$ is isomorphic to a product of matrix rings of division $k$-algebras. If it were not the case that all of these matrix rings are division $k$-algebras, then $\End_\Lambda(B) \otimes_k K$ would not be reduced, contradicting the previous paragraph. Hence, we have that $\End_\Lambda(B) \otimes_k K$ is a product of division $k$-algebras. 

    Consider a decomposition $B_K \simeq B_1 \oplus \dots B_m$ into indecomposable $\Lambda_K$-modules, and assume $f: B_i \to B_j$ is a nonzero homomorphism for some $1 \leq i,j \leq m$. Let $\widetilde{f}$ denote the endomorphism of $B_K$ induced by $f$.
    Then, $\widetilde{f}$ must be nilpotent, contradicting the fact that $\End_\Lambda(B) \otimes_k K$ is reduced. It follows that $\ind(\add(B_K))$ is a semibrick. We conclude that the map 
    \[ \ind(- \otimes_k K): \sbrick \Lambda \to \sbrick \Lambda_K\]
    is well-defined. It is injective since it is the restriction of the map
    \[ \ind(- \otimes_k K): \mods \Lambda \to \mods \Lambda_K,\]
    which was shown to be injective in \cref{lem:Kasprojs}(2) and \cref{lem:Kasprojs}(3).
\end{proof}

\Cref{exmp:nonsemibrick} illustrates that the assumption of $K:k$ to be perfect is necessary, since the nilpotent endomorphism cannot be an isomorphism, whence $\Lambda \otimes_k K$ is not a brick. We conclude with showcasing the importance of semibricks via their connection to wide subcategories.

Recall that a full subcategory $\calW \subseteq \mods \Lambda$ is called \textit{wide} if it is closed under extensions, kernels and cokernels. Denote the collection of all wide subcategories of $\mods \Lambda$ by $\wide \Lambda$.

\begin{theorem}\cite[Thm. 1.2]{Rin76}\label{eq:Ringelbij}
We have mutually-inverse maps defined as follows:
\begin{align*}
    \sbrick \Lambda &\longleftrightarrow \wide \Lambda \\
    \calS & \mapsto \Filt_\Lambda(\calS) \\
    \{ S: S \textnormal{ is simple in $\calW$} \} & \mapsfrom \calW
\end{align*}
\end{theorem}

As a consequence, the bijection with semibricks allows us to lift wide subcategories under base field extension. 

\begin{corollary}\label{cor:liftingwidesubcats}
    Let $k$ be a perfect field and let $K:k$ be a field extension. Then there exists an injective map
    \[ \wide \Lambda \hookrightarrow \wide \Lambda_K\]
    given by
    \[ 
    \begin{tikzcd}[column sep=65]
        \calW \arrow[r, "\textnormal{\cref{eq:Ringelbij}}","\textnormal{bijective}"',mapsto] & \calS  \arrow[r, "\textnormal{\cref{lem:semibrickslift}}", mapsto, "\textnormal{injective}"'] & \calS_K \arrow[r, "\textnormal{\cref{eq:Ringelbij}}","\textnormal{bijective}"',mapsto] & \calW_K.
    \end{tikzcd}
    \]
\end{corollary}

It should be noted that $\calW_K$ in the statement of \cref{cor:liftingwidesubcats} is not equal to the subcategory of $\mods \Lambda$ obtained by applying $- \otimes_k K$ to every object in $\calW$, which is usually not a wide subcategory of $\mods \Lambda_K$. Indeed, this na\"{i}ve procedure would not yield a subcategory which is closed under direct summands. This motivates the investigation of individual objects under base field extension rather than of subcategories, as previously explored in \cref{cor:posetembedftors}.

\subsection{Left-finite and right-finite semibricks}\label{subsec:lrfsbrick}
Let $(P,\leq)$ be a poset. Recall that a \textit{cover relation} in $P$ is given a pair of elements $p_1,p_2\in P$ such that $p_1 \leq p_2$, and if $p_1 \leq p_{1.5} \leq p_2$ then $p_{1.5}$ is either $p_1$ or $p_2$. We denote this cover relation by {$p_1 \lessdot p_2$}.
The Hasse quiver $\Hasse(P)$ of $P$ has as its vertices the elements of $P$ and arrows $p_2 \to p_1$ for every cover relation $p_2 \lessdot p_1$.
The Hasse quiver of $\tors \Lambda$ is particularly  well studied, see for example \cite{DIRRT17}. An important feature of $\Hasse(\tors \Lambda)$ is the brick labelling of its arrows. By \cite[Thm. 3.3(b)]{DIRRT17}, for every arrow $\calT \to \calU$ in $\Hasse(\tors \Lambda)$, there exists a unique brick $B$ such that $\calU^\perp \cap \calT = \Filt \{ B \}$. Thus, we may label the arrow with this brick and denote it by $\calT \xrightarrow{B} \calU$, see also \cite{BarnardCarrolZhu19}. For the purpose of $\tau$-perpendicular intervals, the following description is particularly important.

\begin{proposition} \cite[Thm. 1.3]{Asai2020} \cite[Prop. 4.9]{DIRRT17}\label{prop:sbrickslabels}
	Let $\calT \in \ftors \Lambda$ correspond to $(M,P) \in \stautilt \Lambda$ under the bijection of \cref{thm:AIRftorsbij}. Then the cover relations $\calT \to \calU_i$ in $\Hasse(\tors \Lambda)$ are labelled by distinct isomorphism classes of bricks in the (left-finite) semibrick
	\begin{equation}\label{eq:defnLFsbrick} \calS = \ind ( M / \rad_{\End_{\Lambda}(M)} M). \end{equation}
	Dually, let $(N,Q) = H(M,P) \in \stauinvtilt \Lambda$ be the corresponding support $\tau^{-1}$-tilting pair. Then the cover relations $\calS_j \to \calT$ in $\Hasse(\tors \Lambda)$ are labelled by the distinct isomorphism classes of brick in the (right-finite) semibrick 
	\begin{equation}\label{eq:defnRFsbrick} \calS' = \ind (\soc_{\End_{\Lambda} (N)} N). \end{equation}
\end{proposition}

Denote the collection of \textit{left-finite} semibricks, that is, semibricks arising as in \cref{eq:defnLFsbrick} for some support $\tau$-tilting pair $(M,P)$, by $\flsbrick \Lambda$ and the collection of right-finite semibricks, that is, semibricks arising as in \cref{eq:defnRFsbrick} for some support $\tau^{-1}$-tilting pair $(N,Q)$, by $\frsbrick \Lambda$. The following theorem illustrates the importance of the brick labelling for our purpose.

\begin{theorem}\cite[Thm. 4.16]{DIRRT17}\label{thm:DIRRT174.16}
    Let $(M,P)$ be a $\tau$-rigid pair. Define the torsion classes $\calT_{(M,P)}$ and $\calU_{(M,P)}$ as in \Cref{sec:tauperp}.
    Then $\calU_{(M,P)}^{\perp} \cap \calT_{(M,P)} = \Filt_\Lambda \{ \calS\}$, where $\calS$ consists of the labels of arrows incident to $\calU_{(M,P)}$ in $\Hasse[\calU_{(M,P)}, \calT_{(M,P)}]$. Equivalently, $\calS$ consists of the labels of arrows incident to $\calT_{(M,P)}$ in $\Hasse[\calU_{(M,P)}, \calT_{(M,P)}]$.
\end{theorem}

Therefore{,} we are especially interested in investigating the left-finite and right-finite semibricks under base field extension. 

\begin{lemma}\label{lem:flbrickcommute}
    Let $K:k$ be a MacLane separable field extension and let $\Lambda$ be a finite-dimensional $k$-algebra. Then the following square is commutative:
    \begin{equation}\label{eq:commLFsbrick}
    \begin{tikzcd}
        \staurigid \Lambda \arrow[r] \arrow[d, "- \otimes_k K", swap]  & \flsbrick \Lambda \arrow[d, "\ind(- \otimes_k K)"] \\
        \staurigid \Lambda_K \arrow[r] & \flsbrick \Lambda_K
    \end{tikzcd}
    \end{equation}
    where the horizontal map at the top is given by $(M,P) \mapsto \ind(M/\rad_{\End_\Lambda(M)} M)$ and the horizontal map at the bottom is given by $(M',P') \mapsto \ind(M'/\rad_{\End_{\Lambda_K}(M')} M')$.
\end{lemma}
\begin{proof}
    Let $M \in \staurigid \Lambda$. We simply need to show that 
    \[ \ind(\ind(M / \rad_{\End_\Lambda(M)} M) \otimes_k K) \simeq \ind(M_K/ \rad_{\End_{\Lambda_K}(M_K)}M_K).\]
    Since $-\otimes_k K$ is additive, we observe that 
    \begin{equation}\label{eq:indindsbrick} \ind(\ind(M/\rad_{\End_\Lambda(M)} M) \otimes_k K) \simeq \ind({(}M/\rad_{\End_\Lambda(M)}M{)} \otimes_k K ).\end{equation}
    From \cref{lem:Kas00.2.2} we know that $\End_\Lambda(M) \otimes_k K \simeq \End_{\Lambda_K}(M_K)$, and from \cref{lem:Kasradsoclift}(2) we have $(\rad M)_K \simeq \rad M_K$ since $K:k$ is MacLane separable. Moreover, $- \otimes_k K$ is exact, yielding $(M/N)_K \simeq M_K / N_K$. Putting these observations together, we obtain
    \begin{align*}
        \ind(\ind(M/\rad_{\End_\Lambda(M)} M) \otimes_k K) &= \ind((M/\rad_{\End_\Lambda(M)}M) \otimes_k K ) &(\text{by \cref{eq:indindsbrick}})\\
        & = \ind(M_K / (\rad_{\End_\Lambda(M)} M)_K) &(\text{since $- \otimes_k K$ is exact}) \\
        & = \ind(M_K / \rad_{\End_\Lambda(M) \otimes_k K} M_K) &(\text{by \cref{lem:Kasradsoclift}(2)}) \\
        & = \ind(M_K / \rad_{\End_{\Lambda_K}(M_K)} M_K) &(\text{by \cref{lem:Kas00.2.2}}),
    \end{align*}
    as required.
\end{proof}

The dual statement follows from a similar proof. It is included for the sake of completeness and for later reference.
\begin{lemma}\label{lem:frbrickcommute}
    Let $K:k$ be a MacLane separable field extension and let $\Lambda$ be a finite-dimensional $k$-algebra. Then the following square is commutative:
    \begin{equation}
    \begin{tikzcd}\label{eq:commRFsbrick}
        \stauinvrigid \Lambda \arrow[r] \arrow[d, "- \otimes_k K", swap]  & \frsbrick \Lambda \arrow[d, "\ind(- \otimes_k K)"] \\
        \stauinvrigid \Lambda_K \arrow[r] & \frsbrick \Lambda_K
    \end{tikzcd}
    \end{equation}
    where the horizontal map at the top is given by $(M,P) \mapsto \ind(\soc_{\End_\Lambda(M)} M)$ and the horizontal map at the bottom is given by $(M',P') \mapsto \ind(\soc_{\End_{\Lambda_K}(M')} M')$.
\end{lemma}
\begin{proof}
   Since $K:k$ is MacLane separable, it follows that $(\soc_\Lambda M)_K \simeq \soc_{\Lambda_K} M_K$ by \cref{lem:Kasradsoclift}(3). Similarly to the proof of \Cref{lem:flbrickcommute}, it follows that
    \begin{align*}
        \ind(\ind(\soc_{\End_\Lambda(M)} M) \otimes_k K) &\simeq \ind((\soc_{\End_\Lambda(M)}M) \otimes_k K ) & \text{(since $- \otimes_k K$ is additive)}\\
        & \simeq \ind(\soc_{\End_\Lambda(M) \otimes_k K} M_K) &(\text{by \cref{lem:Kasradsoclift}(3)}) \\
        & \simeq \ind(\soc_{\End_{\Lambda_K}(M_K)} M_K) &(\text{by \cref{lem:Kas00.2.2}}).
    \end{align*}
    This completes the proof.
\end{proof}

\begin{remark}
	Let $K:k$ be any field extension. Then there exist two injective maps
	\[ \mathcal{L}: \flsbrick(\Lambda) \to \flsbrick(\Lambda_K), \quad \mathcal{R}: \frsbrick(\Lambda) \to \frsbrick(\Lambda_K).\]
	This observation follows from the following two facts:
    \begin{enumerate}
        \item The horizontal maps in the commutative diagrams \cref{eq:commLFsbrick} and \cref{eq:commRFsbrick} are bijections by \cite[Thm. 1.3]{Asai2020}.
        \item The scalar extension functor $- \otimes_k K$ is injective-on-objects and lifts support $\tau$-tilting pairs and support $\tau^{-1}$-tilting pairs by \cref{cor:tautiltlift} and its dual.
    \end{enumerate}
    However, unlike in \cref{lem:flbrickcommute} and \cref{lem:frbrickcommute}, it is not clear whether $\mathcal{L}$ and $\mathcal{R}$ are given by $\ind(- \otimes_k K)$ when $K:k$ is not MacLane separable.
\end{remark}

\section{The \texorpdfstring{$\tau$}{tau}-cluster morphism category}\label{sec:tcmc}

The $\tau$-cluster morphism category $\Wfrak(\Lambda)$ of a finite-dimensional algebra $\Lambda$ was sequentially introduced in a series of papers \cite{IT17,BM18w,BH21}. It captures the structure of $\tau$-tilting reductions as introduced in \cite{Jas15,DIRRT17}. The classifying space of $\Wfrak(\Lambda)$ is particularly interesting since its fundamental group, known as the picture group \cite{ITW16,HI21} when $\Lambda$ is $\tau$-tilting finite, is closely related to maximal green sequences \cite{IgusaTodorovMGS2021,BorMot}. In  certain cases, these classifying spaces are actually $K(\pi,1)$ spaces for the picture group \cite{BarnardHanson2022,HI21,HI21p,IT17,IgusaTodorov22,Kai23,Kai24}. Furthermore, generalisations of the $\tau$-cluster morphisms and different approaches to it have provided important insights to its structure, see \cite{Bor21,STTW23,Kai23,Bor24,Kai24}. The main result of this section provides a close connection between the $\tau$-cluster morphism category of $\Lambda$ and that of the $\tau$-cluster morphism category of $\Lambda_K$, where $K:k$ is a MacLane separable field extension. 

\subsection{Background on the \texorpdfstring{$\tau$}{tau}-cluster morphism category}
We begin by introducing some notation. Let $(M,P)$ be a $\tau$-rigid pair and recall that the interval $[\calU_{(M,P)}, \calT_{(M,P)}] \coloneqq [\Fac M, {}^\perp \tau M \cap P^\perp] \subseteq \tors {\Lambda}$ is called \textit{$\tau$-perpendicular}. One defines the corresponding \textit{$\tau$-perpendicular subcategory} by
\begin{equation}\label{eq:defntauperpwide} \calW_{(M,P)} = M^\perp \cap {}^\perp \tau M \cap P^\perp \subseteq \mods {\Lambda}. \end{equation}

The importance of these subcategories comes from the following results. 
\begin{theorem}\cite[Thm. 3.8]{Jas15} \cite[Thm. 4.12]{DIRRT17}\cite[Thm. 3.14]{BST19}\label{thm:widesubcatlatticeiso}
	Let $(M,P)$ be a $\tau$-rigid pair {in $\mods \Lambda$}. 
    \begin{enumerate}
        \item The $\tau$-perpendicular subcategory $\calW_{(M,P)}$ is a \textit{wide subcategory} of $\mods \Lambda$, which is to say that it is closed under kernels, cokernels and extensions.
        \item The category $\calW_{(M,P)}$ is equivalent to the module category $\mods {\Gamma}_{(M,P)}$, for a finite-dimensional algebra ${\Gamma}_{(M,P)}$ with $|{\Lambda}|- |M|-|P|$ isomorphism classes of simple modules.
        \item There exists an isomorphism of complete lattices 
	\[ - \cap \calW_{(M,P)} : [\Fac M, {}^\perp \tau M \cap P^\perp]  \to \tors \calW_{(M,P)} \]
	which restricts to an order-preserving bijection of functorially finite torsion classes.
    \end{enumerate}
\end{theorem}

As a direct consequence of \cref{thm:widesubcatlatticeiso}(3), we obtain that distinct $\tau$-perpendicular intervals of $\tors A$ whose corresponding $\tau$-perpendicular subcategories coincide are related by a lattice isomorphism. In fact, this result holds in much greater generality. Let $[\calU, \calT]$ be an interval in $\tors \Lambda$. By \cite[Thm. 3.2]{Tattar}, the so-called \textit{heart} $\calU^\perp \cap \calT$ of $[\calU, \calT]$ is a quasi-abelian category. Such quasi-abelian categories admit torsion classes, see \cite[Def. 2.4]{Tattar}, which are controller by those of $\tors \Lambda$ in our setting.

\begin{theorem}\cite[Thm. A]{Tattar}\label{thm:wideintiso}
    Let $[\calU_1, \calT_1]$ and $[\calU_2, \calT_2]$ be intervals in $\tors \Lambda$ such that $\calU_1^\perp \cap \calT_1 = \calU_2^\perp \cap \calT_2$. Then there are three lattice isomorphisms, as shown in the diagram below:
    \[
    \begin{tikzcd}
        {[}\calU_1, \calT_1{]} \arrow[rr,"\simeq"] \arrow[rd, "- \cap \calU_1^\perp",swap] && {[}\calU_2, \calT_2{]} \arrow[ld, "- \cap \calU_2^\perp"] \\
        &\tors(\calU_1^\perp \cap \calT_1) = \tors (\calU_2^\perp \cap \calT_2). 
    \end{tikzcd}
    \]
    The inverses of the downward maps are given by
    \begin{align*}
    \tors (\calU_i^\perp \cap \calT_i) & \to [\calU_i, \calT_i] \\
    \V &\mapsto \calU_i * \V \coloneqq \{ X \in \mods \Lambda: \exists 0 \to U \to X \to V \to 0, U \in \calU_i, V \in \V\}.
    \end{align*}
\end{theorem}

The horizontal map in \cite[Thm. A]{Tattar} factors through the downward ones. It is important to understand how the brick labellings of the intervals relate.

\begin{proposition}\label{prop:brickpreserve}
    The horizontal lattice isomorphism between $[\calU_1, \calT_1]$ and $[\calU_2, \calT_2]$ in \cref{thm:wideintiso} preserves the brick labelling of Hasse quivers. Moreover, if $\calU_i^\perp \cap \calT_i$ is a wide subcategory, then the downward facing maps $- \cap \calU_i^\perp$ also preserve the brick labelling.
\end{proposition}
\begin{proof}
    The first statement follows immediately from \cite[Lem. 4.7]{Tattar}. The second statement is \cite[Thm. 4.2(2)]{AsaiPfeifer2019}.
\end{proof}

We now present the definition of the $\tau$-cluster morphism category from the lattice of torsion classes.

\begin{definition}\cite[Def. 3.3]{Kai24}\label{defn:latticedef}
    The $\tau$-cluster morphism category $\Wfrak(\Lambda)$ has as its objects equivalence classes $[\calU_{(M,P)}, \calT_{(M,P)}]_\sim$ of $\tau$-perpendicular intervals of $\tors \Lambda$ under the equivalence relation 
    \[[\calU_{(M_1,P_1)}, \calT_{(M_1,P_1)}] \sim [\calU_{(M_2,P_2)}, \calT_{(M_2,P_2)}] \]
    whenever $\calW_{(M_1,P_1)} = \calW_{(M_2,P_2)}$. The morphisms of $\Wfrak(\Lambda)$ are given by equivalence classes of morphisms in the poset category $\tauint (\tors \Lambda)$. More precisely,
    \begin{align*}
            &\Hom_{\Wfrak(\Lambda)}([\calU_{(M,P)}, \calT_{(M,P)}]_\sim, [\calU_{(N,Q)}, \calT_{(N,Q)}]_\sim) \\
            &\coloneqq \bigcup_{\substack{[\calU_{(M',P')}, \calT_{(M',P')}] \in [\calU_{(M,P)}, \calT_{(M,P)}]_\sim \\ [\calU_{(N',Q')}, \calT_{(N',Q')}] \in [\calU_{(N,Q)}, \calT_{(N,Q)}]_\sim}} \Hom_{\tauint(\tors \Lambda)} ( [\calU_{(M',P')}, \calT_{(M',P')}], [\calU_{(N',Q')}, \calT_{(N',Q')}])
        \end{align*}
    under the equivalence relation
    \[f_{[\calU_{(M_1,P_1)}, \calT_{(M_1,P_1)}][\calU_{(N_1,Q_1)}, \calT_{(N_1,Q_1)}]} \sim f_{[\calU_{(M_2,P_2)}, \calT_{(M_2,P_2)}][\calU_{(N_2,Q_2)}, \calT_{(N_2,Q_2)}]}\]
    whenever $[\calU_{(N_1, Q_1)}, \calT_{(N_1,Q_1)}] \cap \calW_{(M_1, P_1)} = [\calU_{(N_2, Q_2)}, \calT_{(N_2,Q_2)}] \cap \calW_{(M_2, P_2)}$.
\end{definition}

Instead of recalling the definition of composition in $\Wfrak(A)$, we recall the following lemma.
\begin{lemma}\cite[Lem. 3.6]{Kai24}\label{lem:composition}
	Let $\rho_1: \W_1 \rightarrow \W_2$ and $\rho_2: \W_2 \rightarrow \W_3$ be composable morphisms in $\mathfrak{W}(\Lambda)$. Then there exist $\tau$-perpendicular intervals $[\U_1,\T_1] \leq [\U_2,\T_2] \leq [\U_3,\T_3]$ such that $[\U_1,\T_1] \leq [\U_2,\T_2]$ is a representative of $\rho_1$, $[\U_2,\T_2] \leq [\U_3,\T_3]$ is a representative of $\rho_2$, and $[\U_1,\T_1] \leq [\U_3,\T_3]$ is a representative of $\rho_2 \circ \rho_1$.
\end{lemma}

This category may equivalently be defined using $\tau$-perpendicular subcategories and $\tau$-rigid pairs \cite{BM18w,BH21}, 2-term silting objects \cite{Bor21} or the {$g$-vector fan} \cite{STTW23}. In its different guises, various aspects of this category have been investigated. The starting point for all endeavours regarding this category is the fact that the classifying space $\calB \Wfrak(\Lambda)$ is a cube complex whose fundamental group is the picture group $G(\Lambda)$ of \cite{ITW16, HI21}, see \cite{HI21,IgusaTodorov22,Kai23}. The picture group is an important object with close connections to, for example, maximal green sequences \cite{IgusaTodorovMGS2021,BorMot}. As such, if $\calB \Wfrak(\Lambda)$ is a $K(\pi,1)$ space the relationship between the $\tau$-cluster morphism category and the picture group is particularly intimate. The main tool for determining whether $\calB \Wfrak(\Lambda)$ is a $K(\pi,1)$ space are the sufficient conditions developed in \cite{Gro87} for cube complexes to be (locally) CAT(0), which have been translated into a categorical setting in \cite{Igu14}.

\begin{proposition}\cite[Prop 3.5, 3.8]{Igu14}\label{Igu14.3.5}
    {I}f the following three conditions are satisfied, then $\calB \Wfrak(\Lambda)$ is locally CAT(0). In particular, it is then a $K(\pi,1)$ space:
    \begin{enumerate}
		\item There is a faithful functor $\Psi: \Wfrak(\Lambda) \to G$ for some group $G$ considered as a groupoid with one object;
		\item The category $\Wfrak(\Lambda)$ satisfies the pairwise-compatibility condition of first factors;
		\item The category $\Wfrak(\Lambda)$ satisfies the pairwise-compatibility condition of last factors.
	\end{enumerate}
\end{proposition}

If Condition (1) holds, we say that $\Lambda$ \textit{admits a faithful group functor}.
The focus of the remaining section lies on the existence of faithful group functors, which has been conjectured to hold for all finite-dimensional $k$-algebras \cite{HI21p}\footnote{The cited reference assumes $\tau$-tilting finiteness when making this conjecture, but this seems unnecessary in hindsight. At the time \cite{HI21p} written, the definition of the $\tau$-cluster morphism category was only available for $\tau$-tilting finite $k$-algebras \cite{BM18w}. It has since been extended to all finite-dimensional $k$-algebras \cite{BH21,Bor21}}. We refer to the Appendix, where this conjecture has been restated in \cref{conj:faithful}. Various examples of algebras satisfying Condition (1) are known \cite{HI21,HI21p,IgusaTodorov22,Kai23,Kai24,BorMot}. {In particular,} {if} $k$ is algebraically closed and of characteristic 0, this question is addressed in \cite{BorMot} using Ringel--Hall algebras of locally constructible functions. {Condition (2) of \cref{Igu14.3.5} is always satisfied due to the structure of $\tau$-rigid pairs, whereas some} examples of algebras are known not to satisfy Condition (3), see \cite{BarnardHanson2022,HI21p,IgusaTodorov22}. Before considering Condition (1) in more detail, we need to gather more preliminary results.

\subsection{Further results about field extensions}

In this section we collect some further results about the objects involved in \cref{defn:latticedef} under base field extension. The first one is a simple observation.

\begin{lemma}\label{lem:objintauperppreserve}
Let $K:k$ be a field extension, let $(M,P)$ be a support $\tau$-tilting pair {in $\mods \Lambda$} and let $N \in \mods \Lambda$. Then, $N \in \calW_{(M,P)}$ if and only if $N_K \in \calW_{(M_K,P_K)}$.
\end{lemma}
\begin{proof}
	The notation $\calW_{(M_K,P_K)} \subseteq \mods \Lambda_K$ is well-defined since $(M_K,P_K)$ is a $\tau$-rigid pair in $\mods \Lambda$ by \cref{lem:taurigidlift}. Then, using \cref{lem:Kas00.2.2} the following hold:
	\begin{enumerate}
		\item $\Hom_{\Lambda} (M,N)= 0$ if and only if $\Hom_{\Lambda_K}(M_K,N_K)=0$;
		\item $\Hom_{\Lambda} (P,N)=0$ if and only if $\Hom_{\Lambda_K}(P_K,N_K)=0$;
		\item $\Hom_{\Lambda} (N, \tau_{\Lambda} M) = 0$ if and only if $\Hom_{\Lambda_K}(N_K, (\tau_{\Lambda} M)_K)=0$ if and only if $\Hom_{\Lambda_K}(N_K, \tau_{\Lambda_K} M_K)=0$ by \cref{lem:tautranslatelift}.
	\end{enumerate}
	Therefore the result follows from the definitions of $\calW_{(M,P)}$ and $\calW_{(M_K,P_K)}$. 
\end{proof}

The study of left and right-finite semibricks in \cref{subsec:lrfsbrick} is motivated by the powerful description of such semibricks as labels of $\Hasse(\tors \Lambda)$ in \cref{prop:sbrickslabels}. This allows us to describe the relative simple modules of the $\tau$-perpendicular subcategory in terms of left-finite and right-finite semibricks. Our next result is has ties to \cite{AsaiPfeifer2019,BH21}.

\begin{lemma}\label{lem:tauperpbysbricks}
	Let $(M,P)$ be a $\tau$-rigid pair in $\mods \Lambda$. Write $\calB_\calL^{(M,P)}$ for the left-finite semibrick corresponding to $(M^+,P)$ as in \cref{eq:defnLFsbrick}, see \cref{prop:sbrickslabels}, and $\calB_\calR^{(M,P)}$ for the right-finite semibrick corresponding to the support $\tau^{-1}$-tilting pair $H(M^-,P^-)$ as in \cref{eq:defnRFsbrick}, with $H$ as in \cref{eq:tauinversebij}. Then the $\tau$-perpendicular subcategory $\calW_{(M,P)} \subseteq \mods \Lambda$ as defined in \cref{eq:defntauperpwide} is given by
	\[ \calW_{(M,P)} = \Filt_{\Lambda} \{ \calB_\calL^{(M,P)} \cap \calB_\calR^{(M,P)} \}, \]
\end{lemma}
\begin{proof}
	Write $[\calU, \calT] = [\Fac M, {}^\perp \tau M \cap P^\perp] \subseteq \tors \Lambda$ for the corresponding $\tau$-perpendicular interval. By \cref{eq:Ringelbij}, it is possible to write $\calW_{(M,P)} = \Filt_{\Lambda} \{ \calS\}$ for some semibrick $\calS \in \sbrick \Lambda$ and by \cite[Thm. 4.16(a)(c)]{DIRRT17}, $\calS$ consists of labels of arrows incident to $\calU$ in $\Hasse[\calU, \calT]$. However, by \cref{prop:sbrickslabels}, the right-finite semibrick $\calB_\calR^{(M,P)}$ consists of labels of all arrows going into $\calU$ in $\Hasse(\tors \Lambda)$, therefore $\calS \subseteq \calB_\calR^{(M,P)}$. Dually by \cite[Thm. 4.16(a)(d)]{DIRRT17}, $\calS$ consists of labels of arrows incident to $\calT$ in $\Hasse[\calU, \calT]$. By \cref{prop:sbrickslabels} again, the left-finite semibrick $\calB_\calL^{(M,P)}$ consists of labels of all arrows going out of $\calU$ in $\Hasse(\tors \Lambda)$. It follows that $\calS \subseteq \calB_\calL^{(M,P)} \cap \calB_\calR^{(M,P)}$.
	
	Conversely, assume that there is a brick $S \in \calB_\calL^{(M,P)} \cap \calB_\calR^{(M,P)}$. By \cref{prop:sbrickslabels}, $S$ labels an arrow $\calU' \to \calU$ in $\Hasse(\tors \Lambda)$ and an arrow $\calT \to \calT'$ in $\Hasse(\tors \Lambda)$. However, by definition of the brick labelling this means that $S \in \calU^\perp$ and $S \in \calT$, so that $S \in \calU^\perp \cap \calT = \calW_{(M,P)}$. In particular, if $ \mathrm{T}(S)$ denotes the smallest torsion class of containing $S$, we have that $\calU' = \calU \lor \mathrm{T}(S) \subseteq \calT$, so $S$ is labels an arrow adjacent to $\calU$ in $\Hasse[\calU, \calT]$, and thus $S \in \calS$ by \cite[Thm. 4.16]{DIRRT17}.
\end{proof}

In particular, it allows us to apply the result of \cref{subsec:lrfsbrick} to study wide subcategories under base field extension.

\begin{proposition}\label{prop:objectslift}
Let $K:k$ be a MacLane separable field extension, and $(M,P)$ be a $\tau_{\Lambda} $-rigid pair whose $\tau$-perpendicular subcategory is given by $\calW_{(M,P)} = \Filt_{\Lambda} \{ \calS\}$ for some semibrick $\calS$. Then 
\[ \calW_{(M_K,P_K)} = \Filt_{\Lambda_K} ( \calS_K). \]
\end{proposition}
\begin{proof}
Using \cref{lem:tauperpbysbricks}, write $\calS = \calB_\calL^{(M,P)} \cap \calB_\calR^{(M,P)}$. By \cref{lem:intersectlift}(6), the equality $\calS_K = (\calB_\calL^{(M,P)} \cap \calB_\calR^{(M,P)})_K = (\calB_\calL^{(M,P)})_K \cap (\calB_\calR^{(M,P)})_K$ holds. Combining \cref{lem:bongartzcommute} and \cref{lem:flbrickcommute} gives the equality
\[ \ind( (\calB_\calL^{(M,P)})_K) = \calB_{\calL}^{(M_K, P_K)} \in \flsbrick \Lambda_K, \]
and similarly combining \cref{lem:cobongartzcommute}, \cref{lem:dualtaucommutes} and \cref{lem:frbrickcommute} gives the equality
\[ \ind ((\calB_\calR^{(M,P)})_K) = \calB_{\calL}^{(M_K,P_K)} \in \frsbrick \Lambda_K. \]
Since $\Filt_{\Lambda} (\calS_K) = \Filt_{\Lambda} (\add \calS_K)$ by definition and $\add \calS_K$ is closed under direct summands, the result follows from \cref{lem:tauperpbysbricks}.
\end{proof}

It is not difficult to see that \cref{prop:objectslift} concerns objects of the $\tau$-cluster morphism category. To investigate the morphisms of $\Wfrak(\Lambda)$ the following equivalent condition for the identification of morphisms in \cref{defn:latticedef} is necessary.

\begin{lemma}\label{lem:morphismlifts}
	Let $(M,P)$ and $(M',P')$ be $\tau$-rigid pairs in $\mods \Lambda$ such that $\calW \coloneqq \calW_{(M,P)} = \calW_{(M',P')}$. Moreover, let $(N,Q)$ and $(N',Q')$ be $\tau$-rigid pairs in $\mods \Lambda$ such that $M \in\add(N)$, $P\in\add(Q)$, $M'\in\add(N')$ and $P'\in\add(Q')$. Assume also that $\calW_{(N,Q)} = \calW_{(N',Q')}$. Then, the following are equivalent:
	\begin{enumerate}
		\item $[\calU_{(N,Q)}, \calT_{(N,Q)}] \cap \calW = [\calU_{(N',Q')}, \calT_{(N',Q')}] \cap \calW$;
		\item $\calU_{(N,Q)} \cap \calW = \calU_{(N',Q')} \cap \calW$ and $\calT_{(N,Q)} \cap \calW = \calT_{(N',Q')} \cap \calW$; 
        \item $\calU_{(N,Q)} \cap \calW = \calU_{(N',Q')} \cap \calW$;
        \item $\calT_{(N,Q)} \cap \calW = \calT_{(N',Q')} \cap \calW$
	\item $\calB_{\calR}^{(N,Q)} \cap \calW = \calB_{\calR}^{(N',Q')} \cap \calW$;
    \item $\calB_{\calL}^{(N,Q)} \cap \calW = \calB_{\calL}^{(N',Q')} \cap \calW$. 
	\end{enumerate}
\end{lemma}
\begin{proof}
	The implications $(1) \iff (2)$, $(2) \implies (3)$, $(2) \implies (4)$ and $(3)+(4) \implies (2)$ are immediate.
    
    $(3) \implies (4)$ Since $\calW_{(N,Q)} = \calW_{(N',Q')}$, we compute
        \[ \calT_{(N,Q)} \cap \calW = (\calU_{(N,Q)} \cap \calW) * \calW_{(N,Q)}= (\calU_{(N',Q'))} \cap \calW) * \calW_{(N',Q')} = \calT_{(N,Q)} \cap \calW. \]
        Here the first and last equalities follow from a standard argument.
    
	$(4) \implies (3)$ Since $\calW_{(N,Q)} = \calW_{(N',Q')}$, we compute
    \[ \calU_{(N,Q)} \cap \calW = ( \calT_{(N,Q)} \cap \calW) \cap {}^\perp \calW_{(N,Q)} = (\calT_{(N',Q')} \cap \calW) \cap {}^\perp \calW_{(N',Q')} = \calU_{(N',Q')} \cap \calW. \]
    Here the first and last equalities follow directly from the dual of \cite[Lem. 2.8]{ES22}.
    
	$(4) \implies (6)$. Let $\calX \in \ftors \calW$, then there exists a unique (relative) left-finite semibrick $\widetilde{\calB}_{\calL}^{\calX} \in \flsbrick \calW$ corresponding to it by \cite[Thm. 1.3]{Asai2020} and \cref{thm:AIRftorsbij}. By \cref{prop:sbrickslabels}{,} it labels the arrows going out of $\calX$ in $\Hasse(\tors \calW)${,} and by \cref{thm:widesubcatlatticeiso}, there exists a unique $\calY \in [\calU_{(N,Q)}, \calT_{(N,Q)}] \subseteq \tors \Lambda$ such that  now $\calX = \calY \cap \calW$. It follows that $\widetilde{\calB}_{\calL}^{\calX} = \calB_{\calL}^{\calY} \cap \calW$, since the intersection with $\calW$ preserves the brick labels by \cref{prop:brickpreserve}. Consequently, since $\calT_{(N,Q)} \cap \calW = \calT_{(N',Q')} \cap \calW \in \tors \calW${,} it follows that
	\[ \calB_{\calL}^{(N,Q)} \cap \calW = \widetilde{\calB}_{\calL}^{\calT_{(N,Q)} \cap \calW} =  \widetilde{\calB}_{\calL}^{\calT_{(N',Q')} \cap \calW} = \calB_{\calL}^{(N',Q')} \cap \calW. \]
	
    $(3) \implies (5)$ follows from an entirely analogous argument.
	
	$(6) \implies (4)$. Using the same notation and reasoning as in $(4) \implies (6)$, the assumption implies 
	\[  \widetilde{\calB}_{\calL}^{\calT_{(N,Q)} \cap \calW}  = \calB_{\calL}^{(N,Q)} \cap \calW = \calB_{\calL}^{(N',Q')} \cap \calW =  \widetilde{\calB}_{\calL}^{\calT_{(N',Q')} \cap \calW} . \]
	Then it follows directly from \cite[Thm. 1.3]{Asai2020} that $\calT_{(N,Q)} \cap \calW = \calT_{(N',Q')} \cap \calW${,} as required. 
    
    $(5) \implies (2)$ follows from an entirely analogous argument.
\end{proof}

Notice that \cref{lem:morphismlifts}(1) states the defintion of when morphisms $f_{[\calU_{(M,P)}, \calT_{(M,P)}][\calU_{(N,Q)}, \calT_{(N,Q)}]}$ and $f_{[\calU_{(M',P')}, \calT_{(M',P')}][\calU_{(N',Q')}, \calT_{(N',Q')}]}$ from the $\tau$-perpendicular subcategory $\calW_{(M,P)} = \calW_{(M',P')}$ to the $\tau$-perpendicular subcategory $\calW_{(N,Q)} = \calW_{(N',Q')}$ are identified in $\Wfrak(A)$. Thus, the equivalent conditions give various simpler ways of identifying morphisms in $\Wfrak(A)$.

The following lemma, in combination with the previous, is important to control the identification of morphisms when studying the $\tau$-cluster morphism category under scalar extension.

\begin{lemma}\label{lem:perpflsbricklift}
Let $K:k$ be a MacLane separable field extension and let $(M,P)$ be a $\tau$-rigid pair in $\mods \Lambda$. Then the following square is commutative
\[
\begin{tikzcd}[column sep = 70]
	\left\{ \substack{\calB_{\calL}^{(N,Q)} \in \flsbrick \Lambda \\ \text{s.t. } (N,Q) \in {[}(M^-,P^-), (M^+, P){]} \subseteq \stautilt \Lambda }\right\} \arrow[r, "-\cap \calW_{(M,P)}"] \arrow[d, "\ind(- \otimes_k K)"] & \flsbrick \calW_{(M,P)} \arrow[d, "\ind(- \otimes_k K)"]  \\
	\left\{ \substack{\calB_{\calL}^{(N',Q')} \in \flsbrick \Lambda_K \\ \text{s.t. } (N',Q') \in {[}(M_K^-,P_K^-), (M_K^+, P_K){]} \subseteq \stautilt \Lambda_K }\right\}  \arrow[r, "-\cap \calW_{(M_K,P_K)}"] & \flsbrick \calW_{(M_K,P_K)}
\end{tikzcd}
\]
and the horizontal maps are bijections. 
\end{lemma}
\begin{proof}
	The vertical maps are well-defined by \cref{cor:tauperpitvlift} and \cref{lem:flbrickcommute}. Moreover, the horizontal map given by intersection with $\calW_{(M,P)}$ defines a brick label preserving isomorphism from $\Hasse[(M^-,P^-), (M^+,P)]$ to $\Hasse(\stautilt \calW_{(M,P)})$ by \cref{prop:brickpreserve}. As discussed in the proof of \cref{lem:morphismlifts}{,} the horizontal map is therefore well-defined, since the brick labelling is preserved and left-finite semibricks can be read of as the labels of arrows going out of a {support} $\tau$-tilting pair in the Hasse quiver by \cref{prop:sbrickslabels}. The same holds for the bottom horizontal map. The horizontal maps are bijections because the intervals of {support} $\tau$-tilting pairs are in bijection by \cref{thm:widesubcatlatticeiso} and {support} $\tau$-tilting pairs are in bijection with left-finite semibricks by \cite[Thm. 1.3]{Asai2020}. 
	
	Let $S \in \calB_{\calL}^{(N,Q)}$. By \cref{lem:objintauperppreserve}, $S \in \calW_{(M,P)}$ if and only if $S_K \in \calW_{(M_K,P_K)}$. In particular, every indecomposable direct summand $S' \in \add (S_K)$ is contained in $\calW_{(M_K,P_K)}$. Therefore, the square indeed commutes.
\end{proof}

We have collected all the preliminary results to establish the main theorem of this section.

\subsection{Faithful functors under base field extension}
As described earlier, the existence of a faithful functor from $\Wfrak(\Lambda)$ to a groupoid with one object has important consequences and is a sufficient condition for the classifying space to be a $K(\pi,1)$ space. The main result relating to the $\tau$-cluster morphism category is the following result extending the existence of faithful group functors algebras over reasonably nice subfields. All preliminary results have been collected in the previous subsection.

\begin{theorem}\label{thm:basefieldmain}
	Let $K:k$ be a MacLane separable field extension. There exists a well-defined faithful functor $\calF: \Wfrak(\Lambda) \to \Wfrak(\Lambda_K)$ given by
	\begin{equation}
	\begin{aligned}
	 \calF: \Wfrak(\Lambda) &\to \Wfrak(\Lambda_K) \\
	 [\calU_{(M,P)}, \calT_{(M,P)}]_\sim & \mapsto [\calU_{(M_K,P_K)}, \calT_{(M_K,P_K)}]_\sim \\
	 [f_{[\calU_{(N,Q)}, \calT_{(N,Q)}][\calU_{(M,P)}, \calT_{(M,P)}]}] & \mapsto [f_{[\calU_{(N_K,Q_K)}, \calT_{(N_K,Q_K)}][\calU_{(M_K,P_K)}, \calT_{(M_K,P_K)}]}].
	\end{aligned}
	\end{equation}
Consequently, if $\Wfrak(\Lambda_K)$ admits a faithful group functor, so does $\Wfrak(\Lambda)$.
\end{theorem}
\begin{proof}
    We first show that $\calF$ is well-defined on objects.
	It is shown in \cref{cor:tauperpitvlift} that if $[\calU_{(M,P)}, \calT_{(M,P)}]$ is a $\tau_{\Lambda} $-perpendicular interval in $\tors \Lambda$, then $[\calU_{(M_K,P_K)}, \calT_{(M_K,P_K)}]$ is a $\tau_{\Lambda_K}$-perpendicular interval in $\tors \Lambda_K$. Consider two $\tau_{\Lambda} $-perpendicular intervals $[\calU_{(M,P)}, \calT_{(M,P)}]$ and $[\calU_{(M',P')}, \calT_{(M',P')}]$ in $\tors \Lambda$ such that $\calW_{(M,P)} = \calW_{(M',P')} = \Filt_{\Lambda} \{ \calS\}$ for some $\calS \in \sbrick \Lambda$. It follows from \cref{prop:objectslift} that
	\[ \calW_{(M_K,P_K)} = \Filt_{\Lambda_K} \{S_K\} = \calW_{((M')_K, (P')_K)}.\]
	 Therefore, if $[\calU_{(M,P)}, \calT_{(M,P)}]_\sim = [\calU_{(M',P')}, \calT_{(M',P')}]_\sim$ in $\Wfrak(\Lambda)$, then 
    \[ [\calU_{(M_K,P_K)}, \calT_{(M_K,P_K)}]_\sim = [\calU_{((M')_K,(P')_K)}, \calT_{((M')_K,(P')_K)}]_\sim \]
    in $\Wfrak(\Lambda_K)$. This shows that $\calF$ is well-defined on objects. 
    
    To investigate the morphisms, consider two $\tau_{\Lambda} $-perpendicular intervals $[\calU_{(M,P)}, \calT_{(M,P)}]$ and $[\calU_{(N,Q)}, \calT_{(N,Q)}]$ such that $[\calU_{(N,Q)}, \calT_{(N,Q)}] \subseteq [\calU_{(M,P)}, \calT_{(M,P)}] \subseteq \tors \Lambda$. Then the inclusion
	\[ [\calU_{(N_K,Q_K)}, \calT_{(N_K,Q_K)}] \subseteq [\calU_{(M_K,P_K)}, \calT_{(M_K,P_K)}] \subseteq \tors \Lambda_K,\]
	follows from the fact that $- \otimes_k K: \stautilt \Lambda \to \stautilt \Lambda_K$-preserves the partial order, as shown in \cref{lem:posetlift}. 
	Now, consider four $\tau_{\Lambda} $-perpendicular intervals{:} 
    \[ [\calU_{(N,Q)}, \calT_{(N,Q)}] \subseteq [\calU_{(M,P)}, \calT_{(M,P)}] \text{ and } [\calU_{(N',Q')}, \calT_{(N',Q')}] \subseteq [\calU_{(M',P')}, \calT_{(M',P')}]\]
    in $\tors \Lambda$ such that $\Filt_{A} \{\calS\} = \calW_{(M,P)} = \calW_{(M',P')}$ for some $\calS \in \sbrick \Lambda$, $\calW_{(N,Q)} = \calW_{(N',Q')}$ and moreover 
	\[ [\calU_{(N,Q)}, \calT_{(N,Q)}] \cap \calW_{(M,P)} = [\calU_{(N',Q')}, \calT_{(N',Q')}] \cap \calW_{(M',P')} \subseteq \tors \Filt_{\Lambda} \{\calS\}. \]
	Then there is a chain of implications
	\begin{align*}
	&\calT_{(N,Q)} \cap \calW_{(M,P)} = \calT_{(N',Q')} \cap \calW_{(M',P')} \\
		& = \calT_{(N,Q)} \cap \Filt_{\Lambda} \{\calS \} = \calT_{(N',Q')} \cap \Filt_{\Lambda} \{\calS \} \\
		& \Rightarrow \calB_{\calL}^{(N,Q)} \cap \Filt_{\Lambda} \{\calS \} = \calB_{\calL}^{(N,Q)} \cap \Filt_{\Lambda} \{\calS \} & \text{(by \cref{lem:morphismlifts})} \\
		& \Rightarrow \ind((\calB_{\calL}^{(N,Q)} \cap \Filt_{\Lambda} \{\calS \} )_K)  = \ind((\calB_{\calL}^{(N',Q')} \cap \Filt_{\Lambda} \{\calS \} )_K)  \\
		& \Rightarrow \calB_{\calL}^{(N_K,Q_K)} \cap \Filt_{\Lambda_K} \{\calS_K\} = \calB_{\calL}^{((N')_K,(Q')_K)} \cap \Filt_{\Lambda_K} \{\calS_K\} & \text{(by \cref{lem:perpflsbricklift})} \\
		& \Rightarrow \calT_{(N_K,Q_K)} \cap \Filt_{\Lambda_K} \{\calS_K\} = \calT_{((N')_K,(Q')_K)} \cap \Filt_{\Lambda_K} \{ \calS_K\} & \text{(by \cref{lem:morphismlifts})} \\
		& \Rightarrow \calT_{(N_K,Q_K)} \cap \calW_{(M_K,P_K)} = \calT_{((N')_K,(Q')_K)}\cap \calW_{((M')_K,(P')_K)} & \text{(by \cref{prop:objectslift})}
	\end{align*}
	Therefore, \cref{lem:morphismlifts} yields that
	\[ [\calU_{(N_K,Q_K)}, \calT_{(N_K,Q_K)}] \cap \calW_{(M_K,P_K)} = [\calU_{(N',Q')}, \calT_{((N')_K,(Q')_K)}] \cap \calW_{((M')_K,(P')_K)} \]
	in $\tors \Filt_{\Lambda_K}\{\calS_K\}$. In other words, identification of morphisms is preserved, whence $\calF$ determines a well-defined map on the morphism sets. To investigate the composition of morphisms, let 
	\[ [f_{[\calU_{(N',Q')}, \calT_{(N',Q')}]  [\calU_{(N'',Q'')}], \calT_{(N'',Q'')}]} \circ  [f_{[\calU_{(M,P)}, \calT_{(M,P)}][\calU_{(M',P')}, \calT_{(M',P')}]}] \] 
	be two composable morphisms in $\Wfrak(\Lambda)$. Then from \cref{lem:composition} we know that there exists an interval $ [\calU_{(M'',P'')}, \calT_{(M'',P'')}] \subseteq [\calU_{(M',P')}, \calT_{(M',P')}]$ such that 
    \[ [f_{[\calU_{(N',Q')}, \calT_{(N',Q')}][\calU_{(N'',Q'')}, \calT_{(N'',Q'')}]}] = [f_{[\calU_{(M',P')}, \calT_{(M',P')}] [\calU_{(M'',P'')}, \calT_{(M'',P'')}]}] \]
    so that their composition is given by 
	\[ [f_{[\calU_{(M,P)}, \calT_{(M,P)}][\calU_{(M'',P'')}, \calT_{(M'',P'')}]}].\]
	 It is immediate from the definitions that
	\begin{align*}
		&\calF([f_{[\calU_{(M',P')}, \calT_{(M',P')}][\calU_{(M'',P'')}, \calT_{(M'',P'')}]}]) \circ \calF([f_{[\calU_{(M,P)}, \calT_{(M,P)}][\calU_{(M',P')}, \calT_{(M',P')}]}]) \\
		& =  [f_{[\calU_{((M')_K,(P')_K)}, \calT_{((M')_K,(P')_K)}][\calU_{((M'')_K,(P'')_K)}, \calT_{((M'')_K,(P'')_K)}]}] \\
		& \qquad \circ  [f_{[\calU_{(M_K,P_K)}, \calT_{(M_K,P_K)}][\calU_{((M')_K,(P')_K)}, \calT_{((M')_K,(P')_K)}]}] \\
		& =[f_{[\calU_{(M_K,P_K)}, \calT_{(M_K,P_K)}][\calU_{((M'')_K,(P'')_K)}, \calT_{((M'')_K,(P'')_K)}]}] \\
		& = \calF([f_{[\calU_{(M,P)}, \calT_{(M,P)}][\calU_{(M'',P'')}, \calT_{(M'',P'')}]}]).
	\end{align*}
	Thus $\calF$ preserves composition of morphisms. It is clear that $\calF$ preserves identity morphisms. Therefore $\calF: \Wfrak(\Lambda) \to \Wfrak(\Lambda_K)$ is a well-defined functor. 
    
    Finally, to see that $\calF$ is faithful, take two distinct morphisms
	\begin{align} [f_{[\calU_{(M,P)}, \calT_{(M,P)}][\calU_{(N,Q)}, \calT_{(N,Q)}]}], &[f_{[\calU_{(M,P)}, \calT_{(M,P)}][\calU_{(N',Q')}, \calT_{(N',Q')}]}]\label{eq:chosen_mor} \\
	&\in \Hom_{\Wfrak(\Lambda)}([\calU_{(M,P)}, \calT_{(M,P)}]_\sim, [\calU_{(N,Q)}, \calT_{(N,Q)}]_\sim). \nonumber\end{align}
	which may be taken to be represented two morphisms in $\tauint(\tors A)$ with the same domain by aspecial case of \cref{lem:composition} lemma, see also \cite[Cor. 3.6]{Kai23}. Assume that
	\[  [f_{[\calU_{(M_K,P_K)}, \calT_{(M_K,P_K)}][\calU_{(N_K,Q_K)}, \calT_{(N_K,Q_K)}]}] = [f_{[\calU_{(M_K,P_K)}, \calT_{(M_K,P_K)}][\calU_{((N')_K,(Q')_K)}, \calT_{((N')_K,(Q')_K)}]}], \]
	in $\Wfrak(\Lambda_K)$, which is to lead to a contradiction. This would mean that the intervals
	\[ [\calU_{(N_K,Q_K)}, \calT_{(N_K,Q_K)}]\quad \text{and} \quad[\calU_{((N')_K,(Q')_K)}, \calT_{((N')_K,(Q')_K)}] \]
	coincide as a result of \cref{thm:wideintiso}. Indeed, their intersections with $\calW_{(M_K,P_K)}$ coincide {by assumption} and the lattice $\tors \calW_{(M_K,P_K)}$ is isomorphic to the interval $[\calU_{(M_K,P_K)}, \calT_{(M_K,P_K)}] \subseteq \tors \Lambda_K$ by \cref{thm:wideintiso}. Using \cref{lem:cobongartzcommute} and \cref{lem:bongartzcommute}, this implies that 
	\begin{align*}
	&\calU_{(N_K,Q_K)} = \calU_{((N')_K, (Q')_K)} \quad \text{and} \quad \calT_{(N_K,Q_K)} = \calT_{((N')_K, (Q')_K)} \\
	& \Rightarrow ((N_K)^-, (Q_K)^-) = ((((N')_K)^-, (Q')_K)^-) \quad \text{and} \quad ((N_K)^+, Q_K) = ((N')_K^+, (Q')_K) \\
	& \Rightarrow ((N^-)_K, (Q^-)_K) = (((N')^-)_K, ((Q')^-)_K) \quad \text{and} \quad ((N^+)_K, Q_K) = (((N')^+)_K, (Q')_K).
\end{align*}
Since $-\otimes_k K$ defines an injective map of support $\tau$-tilting pairs by \cref{cor:posetembed}, it follows that $(N^-, Q^-) = ((N')^-,(Q')^-)$ and $(N^+, Q)=((N')^+, Q')$. Since two distinct $\tau$-rigid pairs cannot both have the same Bongartz completion and the same co-Bongartz completion{,} this yields $(N,Q) = (N',Q')$. However, this is a contradicts the assumption that the  morphisms chosen in \eqref{eq:chosen_mor} are distinct. We conclude that the functor $\calF$ is faithful.
\end{proof}

\begin{corollary}\label{cor:mainbasefield+BorMot}
    Let $k$ be a field of characteristic 0 and let $\Lambda$ be a finite-dimensional $k$-algebra. Then $\Wfrak(\Lambda)$ admits a faithful group functor. 
\end{corollary}
\begin{proof}
    Let $\overline{k}$ denote {an} algebraic closure of $k$.  Since $k$ is of characteristic 0, the field extension $\overline{k}:k$ is MacLane separable by definition. The $\overline{k}$-algebra $\Lambda_{\overline{k}}$ is a finite-dimensional algebra over an algebraically closed field of characteristic zero. It is shown in \cite{BorMot} that $\Wfrak(\Lambda_{\overline{k}})$ admits a faithful group functor. Thus, we have that $\Wfrak(\Lambda)$ admits a faithful group functor, obtained by composition with the faithful functor $\calF\colon \Wfrak(\Lambda) \to \Wfrak(\Lambda_K)$ in \Cref{thm:basefieldmain}.
\end{proof}

\begin{remark} Suppose that $k$ is algebraically closed. In \cite{Non25}, the author considers the tensor product of a hereditary finite-dimensional $k$-algebra $\Lambda$ with a finite-dimensional local commutative $k$-algebra $R$. It is shown that the functor $-\otimes_k R: \mods \Lambda \to \mods \Lambda \otimes_k R$ commutes with the Auslander--Reiten translation and the Nakayama functor \cite[Prop. 1.8]{Non25}. By \cref{cor:Nakfunctlift} and \cref{lem:tautranslatelift}, the scalar extension functor $- \otimes_k K$ of a field extension shares these properties. In \cite{Non25} it is shown that this leads to a bijection of $\stautilt \Lambda \leftrightarrow \stautilt (\Lambda \otimes_k R)$ \cite[Cor. 1.9]{Non25} and an equivalence of categories $\Wfrak(\Lambda) \simeq \Wfrak(\Lambda \otimes_k R)$ \cite[Thm. 6.17]{Non25}. Given a field extension $K:k$, there is an injection $\stautilt \Lambda \hookrightarrow \stautilt \Lambda_K$ by \cref{cor:tautiltlift} and, if $K:k$ is MacLane separable, a faithful functor $\Wfrak(\Lambda) \to \Wfrak(\Lambda_K)$ by \cref{thm:basefieldmain}. It would be interesting to investigate if a simultaneous generalisation is possible. In this spirit, we refer to \cite{IK2024} and \cite{LZ23}, where some questions regarding silting theory have been studied in a general setting. 
\end{remark}

\subsection{Worked examples}\label{subsec:worked}

{As an interlude, we spend a subsection spelling out our results so far for a few simple examples. It is necessary to recall the notion of species, a generalisation of quivers.

\begin{definition} {\cite{Gab73}}\label{def:species}
	Let $k$ be any field. A \textit{$k$-species} $\mathbb{S} = (Q^{\mathbb{S}}, \{D_a\}_{a\in Q^{\mathbb{S}}_0}, \{ X_\alpha \}_{\alpha \in Q^{\mathbb{S}}_1})$ consists of the data of:
	\begin{itemize}
		\item A quiver $Q^{\mathbb{S}} = (Q_0^{\mathbb{S}}, Q_1^{\mathbb{S}})$;
		\item A division $k$-algebra $D_a$ for each $a \in Q_0^{\mathbb{S}}$;
		\item A $D_a\mhyphen D_b$-bimodule $X_\alpha$ for each $a \xrightarrow{\alpha} b$ in $Q_1^{\mathbb{S}}$. 
	\end{itemize}
\end{definition}

We remark that if $k$ algebraically closed then every divison $k$-algebra is isomorphic to $k$ itself. 
More generally, if every division $k$-algebra $D_a$ is $k$ and every $D_a \mhyphen D_b$-bimodule (or just $k \mhyphen k$-bimodule) $X_{\alpha}$ is $k$, then the $k$-species {encodes the same information as} a quiver. 

Given a $k$-species $\mathbb{S} = (Q^{\mathbb{S}}, \{D_a\}_{a\in Q^{\mathbb{S}}_0}, \{ X_\alpha \}_{\alpha \in Q^{\mathbb{S}}_1})$, let $D$ denote the semisimple $k$-algebra $\prod_{a\in Q_0}D_a$ and let $X$ denote the abelian group $\bigoplus_{\alpha\in Q_1}X_{\alpha}$. Extending the bimodule structure {of each $X_{\alpha}$ in the standard way}, one equips $X$ with the structure of $D$-$D$-bimodule. The \textit{tensor path algebra} of $\mathbb{S}$ is then defined as the tensor $k$-algebra 
\[k\mathbb{S} \coloneqq T_{D}({X}) = \bigoplus_{\ell\geq 0} X^{\otimes_{D} \ell},\]
{where $X^{\otimes_D 0} = D$, and} in which the multiplication is given by the linear extension of the following rule:
\[(x_1 \otimes \cdots \otimes x_{p}) \cdot (y_1 \otimes \cdots \otimes y_{q}) = x_1 \otimes \cdots \otimes x_{p}\otimes y_1 \otimes \cdots \otimes y_{q},\] where the first factor on the left hand side is in $X^{\otimes_{D} p}$ and the second is in $X^{\otimes_{D} q}$. If every division $k$-algebra $D_a$ is $k$ and every $k \mhyphen k$-bimodule $X_{\alpha}$ is $k$, the tensor path algebra of the species $\mathbb{S}$ becomes the ordinary path $k$-algebra of the underlying quiver $Q^{\mathbb{S}}$.

{In the following, most of our examples will be constructed using species over the real numbers $\R$. For an $\R$-species $\mathbb{S}$ and an ideal $I$ of the $\R$-algebra $\R\mathbb{S}$, we will consider the $\R$-algebra $\Lambda = \R\mathbb{S}/I$. It will be useful to know how to compute the $\C$-algebra $\Lambda_{\C}$, obtained by scalar extension along $\C:\R$.\footnote{{Recall that a field $k$ is \textit{real closed} if it is not algebraically closed, and an algebraic closure $\overline{k}$ is a finite extension of $k$. By the Artin--Schreier Theorem, the splitting field of the polynomial $x^2+1$ is then an algebraic closure of ${k}$ \cite{AS27}. The assertions in \Cref{prop:LiCompl} remain true if we replace $\R$ with any other real closed field $k$, and $\C$ by $\overline{k}$, \textit{mutatis mutandis}.}}}

{
\begin{proposition}\label{prop:LiCompl}
    Let $\mathbb{S}$ be an $\R$-species. Suppose that the division $\R$-algebra $D_a$ is either $\R$ or $\C$ for all vertices $a \in Q^{\mathbb{S}}_0$. Assume also that each bimodule $X_{\alpha}$ is simple for each arrow $\alpha \in Q^{\mathbb{S}}_1$.
    \begin{enumerate}
        \item\label{prop:LiCompl1} There is a $\C$-algebra isomorphism $\Psi\colon \R\mathbb{S}\otimes_{\R} \C \xrightarrow{} \C Q$, where the quiver $Q$ is constructed as follows:
        \begin{align*}
             Q_1 \coloneqq & \{a \xrightarrow{\alpha} b \, : \, a \xrightarrow{\alpha} b \in Q_1 ,\, D_a=D_b=\R\} \\
             \sqcup &\{a \xrightarrow{\plus{\alpha}} \plus{b}, \, a \xrightarrow{\overline{\alpha}} \minus{b} \, : \, a \xrightarrow{\alpha} b \in Q_1 ,\, D_a=\R,\, D_b=\C \}\\
             \sqcup &\{\plus{a} \xrightarrow{\plus{\alpha}} b, \, \minus{a} \xrightarrow{\overline{\alpha}} b \, : \, a \xrightarrow{\alpha} b \in Q_1 ,\, D_a=\C,\, D_b=\R \} \\
             \sqcup &\{\plus{a} \xrightarrow{\plus{\alpha}} \plus{b}, \, \minus{a} \xrightarrow{\overline{\alpha}} 
             \minus{b} \, : \, a \xrightarrow{\alpha} b \in Q_1 ,\, D_a=D_b=\C, X_\alpha = {}_{\C} \C_{\C} \} \\ 
             \sqcup &\{\minus{a} \xrightarrow{\plus{\alpha}} \plus{b}, \, \plus{a} \xrightarrow{\overline{\alpha}} \minus{b} \, : \, a \xrightarrow{\alpha} b \in Q_1,\, D_a=D_b=\C,\, X_\alpha = {}_{\C}\overline{\C}_{\C} \} ,
        \end{align*}
    where ${}_{\C}\C_{\C}$ denotes the natural $\C$-$\C$-bimodule $\C$, and ${}_{\C}\overline{\C}_{\C}$ denotes the $\C$-$\C$-bimodule $\C$ equipped with the bimodule structure given by $x \curvearrowright z \curvearrowleft y = xz{y}^\ast$, for $x,y\in \C$, where ${y}^\ast$ denotes the complex conjugate of $y$. We remark that ${}_{\C}\C_{\C}$ and ${}_{\C}\overline{\C}_{\C}$ are the only simple $\C$-$\C$-bimodules up to isomorphism, so that this description covers all cases, given our assumptions that the division $\R$-algebras $D_a$ can only be $\R$ or $\C$ and that each bimodule $X_{\alpha}$ is simple.
    \item\label{prop:LiCompl2} 
        We have that $(\R\mathbb{S} / \langle X \otimes_D X\rangle) \otimes_{\R} \C = \C Q /\langle \text{paths of length 2} \rangle $. 
    \end{enumerate}
\end{proposition}
\begin{proof}
    For \eqref{prop:LiCompl1}, we refer to \cite[§3.2]{Li2023derivediscrete}.
    We prove \eqref{prop:LiCompl2} through the following steps
    \[(\R\mathbb{S} / \langle X \otimes_D X \rangle ) \otimes_{\R} \C \simeq {(\R\mathbb{S})_{\C} \over {\langle {X_\C} \otimes_{D_\C} X_\C \rangle}} \simeq {\C Q \over \langle \text{paths of length 2} \rangle}, \]
    where the first step uses the fact that $(\R \mathbb{S} /\langle X \rangle )\otimes_\R \C \cong D_\C$ is semisimple by \cref{thm:JL}\eqref{thm:JL1}, and the last step follows since $\langle X \otimes_\R \C\rangle_{\C Q} = \langle \text{paths of length 1} \rangle_{\C Q} \subseteq \C Q$ by \eqref{prop:LiCompl1}. This concludes the proof.
\end{proof}
}

Using \Cref{prop:LiCompl}, we give two worked examples.

\begin{example}\label{exmp:tauONEloop} 
    {
    Consider the $\R$-species $\mathbb{S}$ defined by the quiver
    $ Q= \begin{tikzcd} 1, \arrow[loop left, "\alpha", looseness=3, in=130, out = 230] \end{tikzcd}$
    with the division $\R$-algebra $\C$ associated to the vertex $1$ and the $\C$-$\C$-bimodule ${}_{\C}\overline{\C}_{\C}$ associated to the loop $\alpha$. To summarise, let $\mathbb{S} = (Q, \{ D_1 = \C \}, \{ X_\alpha = {}_{\C}\overline{\C}_{\C} \})$. Consider the $\R$-algebra
    $
    \Lambda \simeq \R \mathbb{S} / \langle X \otimes_D X \rangle.
    $
    Since $\Lambda$ is a local $\R$-algebra, it is in particular $\tau$-tilting finite. Using the description given in \Cref{prop:LiCompl}\eqref{prop:LiCompl2}, we find that $\Lambda_\C = \Lambda \otimes_\R \C$ is isomorphic to the path $\C$-algebra of the quiver} 
    \begin{equation*}
        \begin{tikzcd}
\plus{1} \arrow[r, "\minus{\alpha}", bend left] & {\minus{1}} \arrow[l, "\plus{\alpha}", bend left]
\end{tikzcd}
    \end{equation*}
    {modulo the square of the arrow ideal. The complex algebra $\Lambda_\C$ is thus isomorphic to the preprojective $\C$-algebra of type $A_2$. The Auslander--Reiten quiver of $\Lambda$ and $\Lambda_\C$ are given, respectively, by} 
\[
\begin{tikzcd}
S(1) \arrow[r, "(2:1)", bend left] \arrow["{\small\tau}" description, dotted, loop, distance=2em, in=215, out=145] & P({1}) \arrow[l, "(1:2)", bend left]
\end{tikzcd}
\qquad \qquad \qquad  
\begin{tikzcd}
S(\plus{1}) \arrow[d, "\small\tau" description, dotted, bend right=20] \arrow[rd]  & P(\plus{1}) \arrow[l]  \\
S(\minus{1}) \arrow[ru] \arrow[u, "\small\tau" description, dotted, bend right=20] & P(\minus{1}) \arrow[l]
\end{tikzcd}
\]
{The Hasse quiver of the poset of $\tau$-tilting pairs in $\mods \Lambda_\C$ is given below. The images of the support $\tau$-tilting pairs of $\mods \Lambda$ under the embedding in \Cref{cor:posetembed}, in this case induced by the functor $-\otimes_{\R}\C$, are highlighted in \textcolor{orange}{orange} boldface.}
\begin{equation*}
    \begin{tikzcd}
                                    & {\color{orange}{\boldsymbol{(P(\plus{1})\oplus P(\minus{1}),0)}}} \arrow[ld] \arrow[rd] &                                     \\
{(P(\plus{1})\oplus S(\plus{1}),0)} \arrow[d] &                                                                 & {(P(\minus{1})\oplus S(\minus{1}),0)} \arrow[d] \\
{(S(\plus{1}),P(\minus{1}))} \arrow[rd]        &                                                                 & {(S(\minus{1}),P(\plus{1}))} \arrow[ld]        \\
                                    & {\color{orange}{\boldsymbol{(0,P(\plus{1})\oplus P(\minus{1}))}}}                       &                                    
\end{tikzcd}
\end{equation*}
{The $\tau$-cluster morphism category of $\Lambda_{\C}$ is shown in \Cref{fig:tcmc_oneloop}. We refer to the figure caption for further information.}
\begin{figure}[h!]
    \centering
        \begin{tikzcd}[column sep=0.5em]
                                 &                                    & {\tiny\sqbinom{P(\minus{1})\oplus S(\minus{1})}{P(\minus{1})\oplus S(\minus{1})}}                       &                                                                                                                                                    & {\tiny\sqbinom{P(\plus{1})\oplus P(\minus{1})}{P(\minus{1})\oplus S(\minus{1})}} \arrow[rr,"c"] \arrow[ll,"d"'] &                                    & {\tiny\color{orange}{\sqbinom{P(\plus{1})\oplus P(\minus{1})}{P(\plus{1})\oplus P(\minus{1})}}}     \\
                                 & {\tiny\sqbinom{P(\minus{1})\oplus S(\minus{1})}{S(\minus{1})}} \arrow[ru] \arrow[ld] &                                    &                                                                                                                                                    &                                    &                                    & {\tiny\sqbinom{P(\plus{1})\oplus P(\minus{1})}{P(\plus{1})\oplus S(\plus{1})}} \arrow[u,"a"] \arrow[d,"b"'] \\
{\tiny\sqbinom{S(\minus{1})}{S(\minus{1})}}                     &                                    &                                    & {\tiny\color{orange}{\sqbinom{P(\plus{1})\oplus P(\minus{1})}{0}}} \arrow[ruu] \arrow[rrru] \arrow[llld] \arrow[llu] \arrow[rrd] \arrow[ldd] \arrow[rrruu, color=orange] \arrow[llldd,  color=orange] &                                    &                                    & {\tiny\sqbinom{P(\plus{1})\oplus S(\plus{1})}{P(\plus{1})\oplus S(\plus{1})}}                     \\
{\tiny\sqbinom{S(\minus{1})}{0}} \arrow[u,"a"] \arrow[d,"b"'] &                                    &                                    &                                                                                                                                                    &                                    & {\tiny\sqbinom{P(\plus{1})\oplus S(\plus{1})}{S(\plus{1})}} \arrow[ld] \arrow[ru] &                                  \\
{\tiny\color{orange}{\sqbinom{0}{0}}}     &                                    & {\tiny\sqbinom{S(\plus{1})}{0}} \arrow[ll,"d"'] \arrow[rr,"c"] &                                                                                                                                                    & {\tiny\sqbinom{S(\plus{1})}{S(\plus{1})}}                       &                                    &                                 
\end{tikzcd}
    \caption{{The $\tau$-cluster morphism category of $\Lambda_{\C}$, as defined in \Cref{exmp:tauONEloop}. To save space, we are \textit{not} employing the same notation as in \Cref{defn:latticedef}; we write $\sqbinom{X}{Y}$, we mean ${[\Fac Y,\Fac X]_{\sim}}$. Note that, among other identifications, the corners of the outer hexagonal shape are identified to be the same object in $\Wfrak(\Lambda_{\C})$. If $[\Fac Y',\Fac X']$ is a subinterval of $[\Fac Y,\Fac X]$ in $\ftors \Lambda_{\C}$, there is a morphism $\sqbinom{X}{Y} \to \sqbinom{X'}{Y'}$ in $\Wfrak(\Lambda_{\C})$. With the exception of the \textcolor{orange}{orange} morphisms, we only display the \textit{irreducible morphisms} $\Wfrak(\Lambda_{\C})$, i.e. they cannot be factored nontrivially. The \textcolor{orange}{orange} morphisms are images of irreducible morphisms in $\Wfrak(\Lambda)$ under the faithful functor $\calF\colon \Wfrak(\Lambda)\to \Wfrak(\Lambda_{\C})$, see \Cref{thm:basefieldmain}. Morphisms with the same label ($a$, $b$, $c$ or $d$) are to be identified.}}
    \label{fig:tcmc_oneloop}
\end{figure}
\end{example}

\begin{example}\label{eg:large}
Consider the $\R$-species $\mathbb{S}$ defined by the quiver
\[ 1 \xrightarrow{\alpha} 2\]
with division $\R$-algebras $D_1 = \R$ and $D_2 = \C$ assigned to the vertices and the natural $\R$-$\C$-bimodule $X_\alpha = {}_\R \C_\C$ assigned to the arrow. {Consider the hereditary $\R$-algebra} $\Lambda = \R \mathbb{S}$. By \Cref{prop:LiCompl}\eqref{prop:LiCompl1}, scalar extension along the field extension $\C:\R$ gives rise to a $\C$-algebra isomorphism: 
\[ \Lambda_{\C} \simeq \C \left( \begin{tikzcd}[row sep=5] & \plus{2} \\ 1 \arrow[ru, "\plus{\alpha}",pos=0.7] \arrow[rd,"\minus{\alpha}", pos=0.7,swap] \\ & \minus{2} \end{tikzcd} \right).\]
The Auslander--Reiten quiver of $\Lambda$ and $\Lambda_\C$ are given, respectively, by
\[
\begin{tikzcd}
    P(2) \arrow[rd, "(1:2)", swap] && I(2) \arrow[rd, "(1:2)"] \arrow[ll,dotted,"{\small \tau}" description] \\
    & P(1) \arrow[ru, "(2:1)"] & & I(1)\arrow[ll,dotted,"{\small \tau}" description] 
\end{tikzcd}
\qquad 
\begin{tikzcd}
    P(\plus{2}) \arrow[rd] && I(\minus{2}) \arrow[rd] \arrow[ll,dotted,"{\small \tau}" description] \\
    & P(1) \arrow[ru] \arrow[rd] && I(1) \arrow[ll,dotted,"{\small \tau}" description] \\
    P(\minus{2}) \arrow[ru] && I(\plus{2}) \arrow[ru] \arrow[ll,dotted,"{\small \tau}" description] 
\end{tikzcd}
\]

{The Hasse quiver of the poset of $\tau$-tilting pairs in $\mods \Lambda_\C$ is given below. {To save space, $\tau$-tilting pairs with trivial support are written as $\tau$-tilting modules.} Again, the images of the support $\tau$-tilting pairs of $\mods \Lambda$ under the embedding in \Cref{cor:posetembed}, in this case induced by the functor $-\otimes_{\R}\C$, are highlighted in \textcolor{orange}{orange} boldface.}
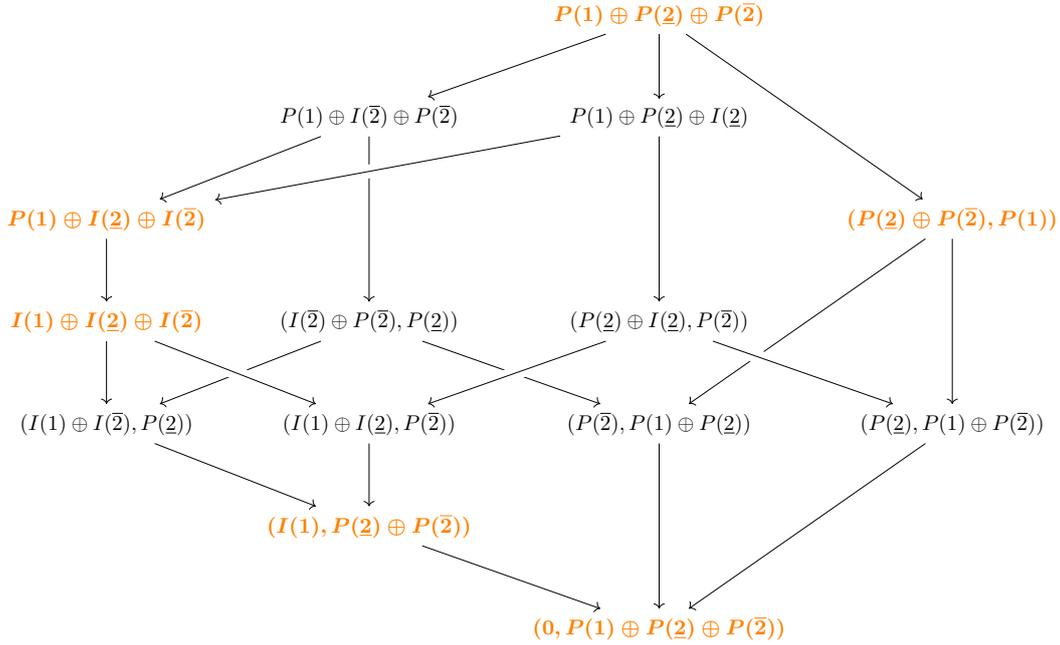
\[
\adjustbox{scale=0.8,center}{
\begin{tikzcd}[column sep=20, row sep =30]
    && {\color{orange}\boldsymbol{P(1) \oplus P(\plus{2}) \oplus P(\minus{2})}} \arrow[ld] \arrow[d] \arrow[rdd] \\
    &P(1) \oplus I(\minus{2}) \oplus P(\minus{2}) \arrow[ld] \arrow[dd] & P(1) \oplus P(\plus{2}) \oplus I(\plus{2}) \arrow[lld,crossing over] \arrow[dd] \\
    {\color{orange}\boldsymbol{P(1) \oplus I(\plus{2}) \oplus I(\minus{2})}} \arrow[d] & & & {\color{orange}\boldsymbol{(P(\plus{2}) \oplus P(\minus{2}), P(1))}} \arrow[ldd] \arrow[dd] \\
    {\color{orange}\boldsymbol{I(1) \oplus I(\plus{2}) \oplus I(\minus{2})}} \arrow[d] & (I(\minus{2}) \oplus P(\minus{2}), P(\plus{2})) \arrow[ld] \arrow[rd] & (P(\plus{2}) \oplus I(\plus{2}),P(\minus{2}))\arrow[ld,crossing over]  & \\
    (I(1) \oplus I(\minus{2}), P(\plus{2})) \arrow[rd] & (I(1) \oplus I(\plus{2}), P(\minus{2})) \arrow[d] \arrow[ul,crossing over,<-] & (P(\minus{2}), P(1) \oplus P(\plus{2})) \arrow[dd] & (P(\plus{2}), P(1) \oplus P(\minus{2})) \arrow[ldd]\arrow[lu,crossing over,<-] \\
    &{\color{orange}\boldsymbol{(I(1), P(\plus{2}) \oplus P(\minus{2}))}} \arrow[rd] \\
    && {\color{orange}\boldsymbol{(0, P(1) \oplus P(\plus{2}) \oplus P(\minus{2}))}}
\end{tikzcd}}
\]

For this example, we also get a nontrivial illustration of \Cref{lem:semibrickslift};
every indecomposable $\Lambda$-module is a brick and hence lifts to a semibrick in the following way:
\[ P(2) \mapsto P(\plus{2}) \oplus P(\minus{2}), \quad P(1) \mapsto P(1), \quad I(2) \mapsto I(\plus{2}) \oplus I(\minus{2}), \quad I(1) \mapsto I(1).\]
In \Cref{fig:smallINlargeTCMC} on p. \pageref{fig:smallINlargeTCMC}, we illustrate how the $\tau$-cluster morphism category of $\Lambda$ embeds into that of $\Lambda_K$. The (nonfull) subcategory $\Wfrak(\Lambda)$ of $\Wfrak(\Lambda_K)$ is marked in \textcolor{orange}{orange}. For further information, we refer to the caption of the figure.

\begin{landscape}
\begin{figure}
\begin{tikzcd}[column sep=2.5em,row sep=2em,execute at end picture={\begin{scope}[on background layer] \draw[->,color=orange] (M) -- (H);\end{scope}}]
                      &                                                                                                                   & {\tiny\sqbinom{P(1)\oplus P(\plus{2})\oplus I(\plus{2})}{P(1)\oplus I(\plus{2})\oplus I(\minus{2})}} \arrow[r,"c"] \arrow[ld,"e"'] & \color{black}{\tiny\sqbinom{P(1)\oplus P(\plus{2})\oplus I(\plus{2})}{P(1)\oplus P(\plus{2})\oplus I(\plus{2})}}                                                                                                                                                       & {\tiny\sqbinom{P(1)\oplus P(\plus{2})\oplus P(\minus{2})}{P(1)\oplus P(\plus{2})\oplus I(\plus{2})}} \arrow[l,"d"'] \arrow[rd,"f"] &                                                                      \\
                      & \color{orange}{\tiny\sqbinom{P(1)\oplus I(\plus{2})\oplus I(\minus{2})}{P(1)\oplus I(\plus{2})\oplus I(\minus{2})}}                                                                                                &                        & |[alias=H]|\color{orange}{\tiny\sqbinom{P(1)\oplus P(\plus{2})\oplus P(\minus{2})}{P(1)\oplus I(\plus{2})\oplus I(\minus{2})}} \arrow[ru,"g"] \arrow[ld,"h"] \arrow[lu,"i"'] \arrow[rd,"j"'] \arrow[ll, color=orange,"\textcolor{blue}{\mathfrak{q}}"'] \arrow[rr, color=orange,"\textcolor{blue}{\mathfrak{p}}"]                                                     &                        & \color{orange}{\tiny\sqbinom{P(1)\oplus P(\plus{2})\oplus P(\minus{2})}{P(1)\oplus P(\plus{2})\oplus P(\minus{2})}}                                                     \\
                      & \color{orange}{\tiny \sqbinom{P(1)\oplus I(\plus{2})\oplus I(\minus{2})}{I(1)\oplus I(\plus{2})\oplus I(\minus{2})}} \arrow[u, color=orange] \arrow[d, color=orange]                                                             & {\tiny\sqbinom{P(1)\oplus I(\minus{2})\oplus P(\minus{2})}{P(1)\oplus I(\plus{2})\oplus I(\minus{2})}} \arrow[lu,"d"] \arrow[r,"f"'] & \color{black}{\tiny\sqbinom{P(1)\oplus I(\minus{2})\oplus P(\minus{2})}{P(1)\oplus I(\minus{2})\oplus P(\minus{2})}}                                                                                                                                                         & {\tiny\sqbinom{P(1)\oplus P(\plus{2})\oplus P(\minus{2})}{P(1)\oplus I(\minus{2})\oplus P(\minus{2})}} \arrow[ru,"c"'] \arrow[l,"e"] & \color{orange}{\tiny\sqbinom{P(1)\oplus P(\plus{2})\oplus {P}({\minus{2}})}{P(\plus{2})\oplus P(\minus{2})}} \arrow[u, "a"', color=orange] \arrow[ddddd, "b", color=orange] \\
{\tiny \sqbinom{I(1)\oplus I(\plus{2})\oplus I(\minus{2})}{I(1)\oplus I(\minus{2})}} \arrow[r,"\ell"'] \arrow[d,"n"] & \color{orange}{\tiny\sqbinom{I(1)\oplus I(\plus{2})\oplus I(\minus{2})}{I(1)\oplus I(\plus{2})\oplus I(\minus{2})}}                                                                                                   & {\tiny \sqbinom{I(1)\oplus I(\plus{2})\oplus I(\minus{2})}{I(1)\oplus I(\plus{2})}} \arrow[l,"m"] \arrow[d,"o"',near start]  &                                                                                                                                                                         &                        &                                                                      \\
\color{black}{\tiny \sqbinom{I(1)\oplus I(\minus{2})}{I(1)\oplus I(\minus{2})}}       & \color{orange}{\tiny \sqbinom{I(1)\oplus I(\plus{2})\oplus I(\minus{2})}{I(1)}} \arrow[lu] \arrow[rd] \arrow[ru] \arrow[ld] \arrow[u, color=orange, "\textcolor{blue}{\mathfrak{r}}"] \arrow[d, color=orange, "\textcolor{blue}{\mathfrak{s}}"'] & \color{black} {\tiny \sqbinom{I(1)\oplus I(\plus{2})}{I(1)\oplus I(\plus{2})}}        & |[alias=M]|\color{orange}{\tiny\sqbinom{P(1)\oplus P(\plus{2})\oplus P(\minus{2})}{0}} \arrow[ddd, color=orange] \arrow[rruu, color=orange,crossing over] \arrow[lldd, color=orange,crossing over,bend left=10] \arrow[lluu, color=orange, bend right=15,crossing over] \arrow[ll, color=orange, bend right=15,crossing over] &                        &                                                                      \\
{\tiny \sqbinom{I(1)\oplus I(\minus{2})}{I(1)}} \arrow[r,"o"] \arrow[u,"m"'] & \color{orange}{{\tiny \sqbinom{I(1)}{I(1)}}}                                                                                                   & {\tiny \sqbinom{I(1)\oplus I(\plus{2})}{I(1)}} \arrow[l,"n"'] \arrow[u,"\ell"]  &                                                                                                                                                                         &                        &                                                                      \\
                      & \color{orange}{\tiny \sqbinom{I(1)}{0}} \arrow[u, "a", color=orange] \arrow[d, "b"', color=orange]                                                  & {\tiny \sqbinom{P(\plus{2})}{0}} \arrow[r,"c"] \arrow[ld,"e"'] & \color{black}{\tiny \sqbinom{P(\plus{2})}{P(\plus{2})}}                                                                                                                                                    & {\tiny \sqbinom{P(\plus{2})\oplus P(\minus{2})}{P(\plus{2})}} \arrow[l,"d"'] \arrow[rd,"f"] &                                                                      \\
                      & \color{orange}{\tiny \sqbinom{0}{0}}                                                                                                  &                        & \color{orange}{\tiny \sqbinom{P(\plus{2})\oplus P(\minus{2})}{0}} \arrow[ru,"g"] \arrow[ld,"h"] \arrow[lu,"i"'] \arrow[rd,"j"'] \arrow[ll, color=orange ,"\textcolor{blue}{\mathfrak{q}}"'] \arrow[rr, color=orange,"\textcolor{blue}{\mathfrak{p}}"]                                                     &                        & \color{orange}{\tiny \sqbinom{P(\plus{2})\oplus P(\minus{2})}{P(\plus{2})\oplus P(\minus{2})}}                                                      \\
                      &                                                                                                                   & {\tiny \sqbinom{P(\minus{2})}{0}} \arrow[lu,"d"] \arrow[r,"f"'] & \color{black}{\tiny \sqbinom{P(\minus{2})}{P(\minus{2})}}                                                                                                                                                         & {\tiny \sqbinom{P(\plus{2})\oplus P(\minus{2})}{P(\minus{2})}} \arrow[l,"e"] \arrow[ru,"c"'] &                                                                     
\end{tikzcd}
\caption{{The diagram displays a subcategory of the $\tau$-cluster morphism category $\Wfrak(\Lambda_{\C})$, where the $\R$-algebra $\Lambda$ is as in \Cref{eg:large}. 
Alike {to} \Cref{fig:tcmc_oneloop}, we save space by \textit{not} employing the same notation as in \Cref{defn:latticedef}; we write $\sqbinom{X}{Y}$, we mean ${[\Fac Y,\Fac X]_{\sim}}$. If $[\Fac Y',\Fac X']$ is a subinterval of $[\Fac Y,\Fac X]$ in $\ftors \Lambda_{\C}$, there is a morphism $\sqbinom{X}{Y} \to \sqbinom{X'}{Y'}$ in $\Wfrak(\Lambda_{\C})$. Some of the morphisms have been labelled, and morphisms with the same label are identified. In particular, the upper ``hexagon'' is to be identified with the lower ``hexagon'' (and under further identifications, these classifying spaces of these ``hexagons'' are seen to become 2-tori). The \textcolor{orange}{orange} subdiagram is precisely the (essential) image of the faithful functor $\calF\colon \Wfrak(\Lambda) \xrightarrow{} \Wfrak(\Lambda_{\C})$, see \Cref{thm:basefieldmain}. Morphisms labelled by \textcolor{blue}{blue} fractured letters, namely $\textcolor{blue}{\mathfrak{p}}$--$\textcolor{blue}{\mathfrak{s}}$, are images of irreducible morphisms in $\Wfrak(\Lambda)$, yet they are not irreducible in $\Wfrak(\Lambda_{\C})$.}} 
\label{fig:smallINlargeTCMC}
\end{figure}
\end{landscape}
\end{example}

\newpage

\section{\texorpdfstring{$\tau$}{tau}-tilting finiteness and $\bfg$-tameness under base field extension}\label{sec:ttfin}

Drozd's celebrated trichotomy groups all finite-dimensional algebras over algebraically closed fields into one of three types: finite, tame and wild. A generalisation to hereditary finite-dimensional algebras over arbitrary fields can be made; one determines whether the associated symmetric bilinear form on the Grothendieck group is positive definite, positive semidefinite or indefinite. 
{We remind the reader that representation finiteness of an algebra is preserved by MacLane separable field extensions, see \Cref{thm:JL}\eqref{thm:JL2}.}

\begin{reptheorem}{thm:JL}
\cite[Thm. 3.3]{JL82}\label{thm:JLmainthm}
    {If $K:k$ is MacLane separable, then $\Lambda$ is representation finite if and only if $\Lambda_K$ is representation finite. In this case, every indecomposable $\Lambda_K$-module is a direct summand of a $\Lambda_K$-module $M_K$ for some indecomposable $M \in \mods \Lambda$.} 
\end{reptheorem}

The class of $\tau$-tilting finite algebras is strictly larger than that of representation finite ones. Since $\tau$-rigid modules are preserved under base field extension by \cref{lem:taurigidlift}, it is a natural question to ask whether $\tau$-tilting finiteness is preserved under base field extensions. 

{In contrast, tame} representation type is only a well-defined notion over algebraically closed base fields. Over an arbitrary base field, one can regard generic tameness as a generalisation, since the definitions are equivalent in the algebraically closed case \cite[Thm. 4.4]{CB1991}. Kasjan shows that generic tameness is preserved under separable algebraic scalar extension \cite[Thm. 4.3]{Kas2001}.

{One possible} ``$\tau$-tilted'' version of tameness is $g$-tameness, a notion we recall presently. Let $\Lambda$ be a finite-dimensional $k$-algebra admitting $n$ isomorphism classes of simple modules.
We then say that $\Lambda$ is \textit{$g$-tame} if the union of the cones in its $g$-vector fan is dense in the ambient Euclidean space $\R^n$ \cite{AokiYurikusa2023}. If $k$ is algebraically closed, then any tame $k$-algebra is $g$-tame \cite{PY23}. {For other generalisations of the notion of tameness using $\tau$-tilting theory see \cite{Pfeifer2025} and the references therein.}

\begin{example}\label{exmp:Acounterexample}
Consider the $\R$-species $\mathbb{S}$ with underlying quiver
    \[ Q= \begin{tikzcd}
1 \arrow["\gamma"', loop, distance=2em, in=305, out=235] \arrow["\alpha", loop, distance=2em, in=145, out=215] \arrow["\beta"', loop, distance=2em, in=35, out=325]
\end{tikzcd} \]
    with the division $\R$-algebra $\C$ associated to the vertex $1$ and the $\C$-$\C$-bimodule ${}_{\C}\overline{\C}_{\C}$ associated to each loop $\alpha$, $\beta$ and $\gamma$.  
    To summarise, let $\mathbb{S} = (Q, \{D_1 = \C\}, \{X_\alpha = {}_{\C}\overline{\C}_{\C}, X_\beta = {}_{\C}\overline{\C}_{\C}, X_\gamma = {}_{\C}\overline{\C}_{\C}\})$. Referring to the discussion following \Cref{{def:species}}, we may define the tensor path algebra $\R \mathbb{S}$. Now, consider the $\R$-algebra
    $
    \Lambda \simeq \R \mathbb{S} / \langle X \otimes_D X \rangle.$ {It  is clear from the explicit description that the only nonzero idempotent of $\R \mathbb{S}/ \langle X \otimes_D X \rangle$ corresponds to the simple $\R$-algebra $D = \C$.} {Thus, $\Lambda$} is a local $\R$-algebra, so in particular, it is $\tau$-tilting finite. Since the cones of the $g$-vector fan of a $\tau$-tilting finite $k$-algebra cover the ambient Euclidean space \cite[Thm. 4.7]{Asa21}, we have that $\Lambda$ is $g$-tame. 
    Following the description given in \Cref{prop:LiCompl}, we find that $\Lambda_\C = \Lambda \otimes_\R \C$ is isomorphic to the path $\C$-algebra of the quiver
    \begin{equation*}
        \begin{tikzcd}[column sep=4em]
\plus{1} \arrow[r, "\minus{\alpha}"', bend left] \arrow[r, "\minus{\beta}" description, bend left=55]  \arrow[r, "\minus{\gamma}", bend left=100] & \minus{1} \arrow[l, "\plus{\alpha}"', bend left] \arrow[l, "\plus{\beta}" description, bend left=55] \arrow[l, "\plus{\gamma}", bend left=100]
\end{tikzcd}
    \end{equation*}
    modulo the square of the arrow ideal. {The algebra $\Lambda_\C$ may be obtained by gluing together, in the sense of \cite[Thm. 4.2]{AHIKM2023}, two 3-Kronecker algebras }
    \[ \Lambda_{\plus{1}} \cong \C \left( \begin{tikzcd} \plus{1} \arrow[r] \arrow[r, shift left, bend left] \arrow[r, shift right, bend right] & \minus{1} \end{tikzcd}\right) \cong \begin{pmatrix} \C & \C^3 \\ 0 & \C \end{pmatrix}, \quad \text{and} \quad \Lambda_{\minus{1}} \cong \C \left( \begin{tikzcd} \plus{1} & \minus{1} \arrow[l] \arrow[l, shift left, bend left] \arrow[l, shift right, bend right] \end{tikzcd} \right) \cong \begin{pmatrix} \C & 0 \\ \C^3 & \C \end{pmatrix}.\]
    {Each algebra $\Lambda_{\plus{1}}$ and $\Lambda_{\minus{1}}$ are well-known to not be $g$-tame, and the description of the $g$-vector fan of algebra $\Lambda_\C$ obtained by gluing immediately yields that $\Lambda_\C$ is not $g$-tame.} {Thus,} we have produced a counter-example to {both} the assertion that {$\tau$-tilting finiteness is preserved by scalar extension along $\C:\R$ as well as the assertation that} $g$-tameness is preserved by scalar extension along $\C:\R$.
\end{example}

\subsection{Cases where $\tau$-tilting finiteness is preserved}

Nonetheless, it seems that a generalisation {of \cref{thm:JLmainthm}} does hold in many cases. We will not give precise conditions for this on this occasion, but rather use the following family of algebras as an illustrating example.

\begin{definition}\cite{BGL1987}
    Let $k$ be any field, and let $\Lambda$ be a finite-dimensional hereditary $k$-algebra.\footnote{In \cite{BGL1987}, it is also assumed that $\Lambda$ is a basic $k$-algebra. However, this is not a necessary assumption in the following \cite[Thm. 2.3, Cor. 5.5]{CB99}.} Consider a projective $\Lambda$-module $P$. The \textit{preprojective algebra} for $P$ is defined as the $\Z$-graded $k$-algebra
\[ \Pi^{{\Lambda}}(P) = \bigoplus_{i=0}^\infty \Pi_i^{{\Lambda}}(P),\]
where $\Pi_i^{{\Lambda}}(P) = \Hom_\Lambda(P,(\tau_{\Lambda}^{-1})^i P)$ and multiplication $\Pi_r^{{\Lambda}}(P) \times \Pi_s^{{\Lambda}}(P) \to \Pi_{s+r}^{{\Lambda}}(P)$ is given by
\begin{equation}\label{eq:preprojmult}
     u_r u_s = ((\tau_{\Lambda}^{-1})^r (u_r)) \circ u_s, \quad u_r \in \Pi_r^{{\Lambda}}(P), u_s \in \Pi_s^{{\Lambda}}(P).
\end{equation}
For $P = \Lambda$, the $k$-algebra $\Pi(P)$ is called the \textit{preprojective algebra} of $\Lambda$.
\end{definition}

Other definitions of preprojective algebras occur in different settings. The most common definitions are shown to be equivalent in \cite[Prop. 3.1]{BGL1987} and \cite{Ringel1998}. As a result, the construction giving rise to preprojective algebras will now be shown to commute with taking MacLane separable base field extensions.

\begin{proposition}\label{preprojcommute}
    Let $K:k$ be a MacLane separable field extension and let $\Lambda$ be a finite-dimensional hereditary $k$-algebra. For every projective module $P \in \proj \Lambda$, there {is an isomorphism $(\Pi^{\Lambda}(P))_K \simeq \Pi^{\Lambda_K}(P_K^{})$ of $K$-algebras.} 
    In other words, taking preprojective algebras commutes with field extension.
\end{proposition}
\begin{proof}
    By \Cref{thm:JL}\eqref{thm:JL1}, we have that $\Lambda_K$ is a hereditary $K$-algebra. We use \cref{lem:Kasprojs}(4) to argue that $P_K$ is a projective $\Lambda$-module.
    For $r\geq 0$, consider the $k$-linear map $\varphi_r\colon \Pi^\Lambda_{r}(P) \to \Pi^{\Lambda_K}_{r}(P_K^{})$ sending $P\xrightarrow{u_r}(\tau^{-1}_{\Lambda})^rP$ to $P_K\xrightarrow{u_r\otimes 1}(\tau^{-1}_{\Lambda_K})^rP_K$. Let $\varphi$ denote the linear extension $ \Pi^\Lambda(P) \to \Pi^{\Lambda_K}(P_K^{})$. 
    Since we can decompose $\varphi \otimes_k K$ into the following sequence of isomorphisms, we have that $\varphi \otimes_k K$ is an isomorphism of $k$-vector spaces.
    \begin{align*}
        \Pi^\Lambda(P) \otimes_k K &\simeq \left( \bigoplus_{i=0}^\infty \Pi_i^\Lambda (P) \right) \otimes_k K \\
        &\simeq \bigoplus_{i=0}^\infty \left(\Pi_i^\Lambda(P) \otimes_k K\right)\\
        & \simeq \bigoplus_{i=0}^\infty \Hom_{\Lambda}(P, (\tau_{\Lambda}^{-1})^i P) \otimes_k K \\
        & \simeq \bigoplus_{i=0}^\infty \Hom_{\Lambda_K} (P_K, ((\tau_{\Lambda}^{-1})^iP)_K) & \text{(by \cref{lem:Kas00.2.2})}\\
        & \simeq \bigoplus_{i=0}^\infty \Hom_{\Lambda_K} (P_K, (\tau^{{-1}}_{\Lambda_K})^i P_K) & \text{(by the dual of \cref{lem:tautranslatelift})} \\
        & \simeq \bigoplus_{i=0}^\infty \Pi_i^{\Lambda_K}(P_K) \\
        & \simeq \Pi^{\Lambda_K}(P_K^{})
    \end{align*}
    It remains to check that $\varphi$ is compatible with multiplication. It suffices to show that $\varphi(u_ru_s) = \varphi_r(u_r)\varphi_s(u_s)$ for all $r,s\geq 0$. We find that
    \begin{align*}
        \varphi(u_ru_s) &= u_r u_s \otimes 1 \\
        &= \big(((\tau_{{\Lambda}}^{-1})^r (u_r)) \circ u_s \big) \otimes 1 & \text{(by \cref{eq:preprojmult})} \\
        &= ((\tau_{{\Lambda}}^{-1})^r (u_r)\otimes 1) \circ (u_s\otimes 1) & \text{(functoriality of $-\otimes_k K$)} \\
        &= ((\tau_{{\Lambda_{{K}}}}^{-1})^r (u_r {\otimes 1})) \circ (u_s\otimes 1) & \text{(by {the dual of} \cref{{lem:tautranslatelift}})} \\
        &= \varphi_r(u_r)\varphi_s(u_s),
    \end{align*}
    which is what we wanted to show, concluding the proof.
\end{proof}

This observation is particularly interesting because preprojective algebras of representation finite hereditary algebras {with sufficiently many isomorphism classes of simple modules} are representation infinite. Such examples provide nontrivial (meaning representation infinite) examples such that $\tau$-tilting finiteness is preserved by \cite[Thm. 7.9]{AHIKM2022}, see also \cref{thm:introtaufinite}.

\begin{example}\label{exmp:preprojexmp}
    Consider the following $\R$-species $\mathbb{S}$ of type $C_3$ given by the quiver
    \[ \begin{tikzcd}
        1 \arrow[r, "\alpha"] & 2  \arrow[r, "\beta"] &3
    \end{tikzcd}, \]
    with assigned division $\R$-algebras $D_1 = \C = D_2$ and $D_3 = \R${,} and bimodules $X_\alpha = {}_\C \C_\C$ and $X_\beta = {}_\C \C_\R$ {with the canonical bimodule action}. Then the tensor $\R$-algebra $\Lambda \coloneqq \R \mathbb{S}$ is a hereditary $\R$-algebra. It follows from \Cref{prop:LiCompl}\eqref{prop:LiCompl1} that $\Lambda \otimes_\R \C$ is isomorphic to the path algebra of type $A_5$ given by
    \[ \Lambda \otimes_\R \C = \C \left( \begin{tikzcd}
        \plus{1} \arrow[r,"\plus{\alpha}"] & \plus{2} \arrow[r, "\plus{\beta}"] & 3 & \minus{2} \arrow[l, "\minus{\beta}",swap] & \minus{1} \arrow[l, "\minus{\alpha}",swap]
    \end{tikzcd}\right). \]
    Now{,} the preprojective algebra $\Pi(\Lambda_\C)$ of $\Lambda_{\C}$ is {well-known to be} representation infinite but $\tau$-tilting finite. As a consequence of \Cref{thm:JL}\eqref{thm:JL2} and \cref{preprojcommute}, the algebra $\Pi(\Lambda)$ must therefore be representation infinite. Moreover, by \cite[Thm. 7.9]{AHIKM2022}, the algebra $\Pi(\Lambda)$ is $\tau$-tilting finite. Therefore, $\Pi(\Lambda)$ is an example of a representation infinite $\tau$-tilting finite algebra which remains $\tau$-tilting finite under base field extension. 
\end{example}

We conclude the study of $\tau$-tilting finiteness under base field extension with a general result. 

\begin{lemma}\label{lem:taufinbrickcondition}
    Let $K:k$ be a separable algebraic field extension and let $\Lambda$ be a $\tau$-tilting finite $k$-algebra. Let $M \in \mods \Lambda$ and consider an indecomposable direct summand $N \in \add (M_K)$. If $N \in \brick \Lambda_K$ implies $M \in \brick \Lambda$, then $\Lambda_K$ is also $\tau$-tilting finite.
\end{lemma}
\begin{proof}
    By the assumption on $K:k$ and \cite[Prop. 4.13]{Kas00}, every indecomposable $\Lambda_K$-module arises a direct summand of $M_K$ for some $M \in \mods \Lambda$, so in particular, every brick $B \in \brick \Lambda_K$ does. Since $\Lambda$ is $\tau$-tilting finite, there are only finitely many bricks in $\mods \Lambda$ by \cite[Thm. 1.4]{DIJ2019}. Clearly, $M_K$ has finitely many indecomposable direct summands for all $M \in \mods \Lambda$. If the $N \in \brick \Lambda_K$ implies $M \in \brick \Lambda$, then the number of bricks in $\Lambda_K$ is bounded above:
    \[ |\brick \Lambda_K| \leq (\max \{ |B_K|: B \in \brick \Lambda\}) \cdot |\brick \Lambda| < \infty. \]
    Therefore $\Lambda_K$ has finitely many isomorphism classes of bricks, hence is $\tau$-tilting finite by \cite[Thm. 1.4]{DIJ2019}.
\end{proof}

The last condition imposed in \Cref{lem:taufinbrickcondition} generally does not hold, as \Cref{exmp:tauONEloop} and \cref{exmp:Acounterexample} show. In the latter example, the indecomposable projective $\Lambda$-module is not a brick, since applying the scalar extension functor gives the sum of indecomposable projective $\Lambda_K$ modules by \cref{lem:Kasprojs}(4). However, these do not form a semibrick, so that the original projective $\Lambda$-module cannot be a brick by \cref{lem:semibrickslift}. Nonetheless, each indecomposable projective $\Lambda_K$-module is a brick. Moreover, it also does not hold for \cref{exmp:preprojexmp}.

\appendix

\section{Picture groups and faithful group functors over finite fields \\ by Eric J. Hanson}\label{sec:appendix}

In this appendix, we discuss the relationship between faithful group functors for $\tau$-cluster morphism categories and (completed, finitary) Ringel--Hall algebras. This expands upon constructions that previously appeared in \cite{HI21} and \cite{HI21p}. The organisation is as follows. In Section~\ref{sec:pic_groups} we recall the definition of the picture group and define a pair of new groups with larger generating sets (the interval-heart groups). In Section~\ref{sec:functors}, we prove that there exists a group functor from the $\tau$-cluster morphism category of an algebra to its interval-heart group and give a sufficient condition for this functor to be faithful. In Section~\ref{sec:ringel}, we restrict to algebras over finite fields. We then construct a morphism of groups from the interval-heart group to the group of units of the completed Ringel--Hall algebra and use this to verify the sufficient condition from the previous section. Finally, in Section~\ref{sec:remarks}, we discuss the extent to which the results of Section~\ref{sec:ringel} extend to algebras over infinite fields. This includes some clarifying remarks on a construction that appeared in~\cite{HI21p}.

\subsection{Picture groups and interval-heart groups}\label{sec:pic_groups}

Picture groups were first introduced in \cite{ITW16} for representation finite hereditary algebras, and the definition has since been generalised in several other works, see e.g. \cite{BorMot,Bor21,HI21,IgusaTodorovMGS2021,Kai24}. We use the definition and generality from \cite{Kai24}. In what follows, we say that a brick $S$ is an \emph{f-brick} if $\Filt(S)$ is a $\tau$-perpendicular category.

\begin{remark}\label{rem:fbrick}
    By \cite[Thm.~1.1]{BH21} and Theorem~\ref{eq:Ringelbij}, $S$ is an f-brick if and only if there exits a left-finite semibrick $\mathcal{S}$ with $S \in \mathcal{S}$. We note that the term f-brick is also used in \cite{DIJ2019} for the (generally weaker, see  \cite[Example~3.13]{Asai2020}) property that $S$ itself is a left-finite (semi)brick. Finally, we note that if $\Lambda$ is $\tau$-tilting finite, then all bricks are f-bricks by \cite[Thms.~4.1 and~4.2]{DIJ2019}.
\end{remark}

\begin{definition}\label{def:pic_group}\cite[Def. 7.2]{Kai24}
	Let $\Lambda$ be a finite-dimensional algebra. The \emph{picture group} $G(\Lambda)$ is defined as having generators
	$$\{X_S: S \in \fbrick \Lambda\} \cup \{Y_\T: \T \in \ftors \Lambda\}$$
	with a relation $Y_{\U} = X_S Y_{\U}$ whenever there is a cover relation $\U \covered \T$ in $\ftors \Lambda$ labelled by $S$, and the relation $Y_0 = e$.
\end{definition}

It will also be useful to consider the following groups.

\begin{definition}\label{def:int_group}
	Let $\Lambda$ be a finite-dimensional algebra. We define the \emph{interval-heart group} of $\tors \Lambda$, denoted $\intheart(\tors \Lambda)$, as having generators
	$$\{Z_{[\U,\T]}: \U \subseteq \T \in \tors \Lambda\}$$
	with the following relations:
	\begin{enumerate}
		\item For every torsion class $\T \in \tors \Lambda$, there is a relation $Z_{[\T,\T]} = e$.
		\item For every $\V \subseteq \U \subseteq \T \in \tors \Lambda$, there is a relation $Z_{[\U,\T]}Z_{[\V,\U]} = Z_{[\V,\T]}$.
		\item For every pair of closed intervals $[\U,\T]$ and $[\U',\T']$ in $\tors \Lambda$ such that $\U^\perp \cap \T = (\U')^\perp \cap \T'$, there is a relation $Z_{[\U,\T]} = Z_{[\U',\T']}$.
	\end{enumerate}
	We define the \emph{interval-heart group} of $\ftors \Lambda$, denoted $\intheart(\ftors \Lambda)$ analogously by replacing each instance of $\tors \Lambda$ with $\ftors \Lambda$. To avoid confusion, we distinguish the generators of $\intheart(\tors \Lambda)$ from those of $\intheart(\ftors \Lambda)$ by writing $Z_{[\U,\T]}^{\mathsf f} \in \intheart(\ftors \Lambda)$ and $Z_{[\U,\T]} \in \intheart(\tors \Lambda)$.
\end{definition}

\begin{remark}\label{rem:int_group}
	It is clear from the definitions that the association $Z_{[\U,\T]}^{\mathsf f} \mapsto Z_{[\U,\T]}$ extends to a morphism of groups $\iota: \intheart(\ftors \Lambda) \rightarrow \intheart(\tors \Lambda)$. Moreover, by \cite[Thm.~1.2]{DIJ2019}, $\iota$ is an isomorphism if $\Lambda$ is $\tau$-tilting finite.
\end{remark}

The picture group is related to the interval-heart group by the following. We implicitly use the fact that every cover relation in $\ftors\Lambda$ is also a cover relation in $\tors\Lambda$, which follows from \cite[Thm.~1.3]{DIJ2019}.

\begin{proposition}\label{prop:embedding}
	Let $\Lambda$ be a finite-dimensional algebra. Then there is a surjective morphism of groups $\psi: G(\Lambda) \rightarrow \intheart(\ftors \Lambda)$ given on generators as follows.
	\begin{enumerate}
		\item For every $\T \in \ftors \Lambda$, we have $\psi(Y_\T) = Z^{\mathsf f}_{[0,\T]}$.
		\item For every $S \in \fbrick \Lambda$, let $\U \covered \T$ be a cover relation in $\ftors \Lambda$ labelled by $S$. Then $\psi(X_S) = Z^{\mathsf f}_{[\U,\T]}.$
	\end{enumerate}
	Moreover, if $\Lambda$ is $\tau$-tilting finite, then $\psi$ is an isomorphism.
\end{proposition}

\begin{proof}
	We first note that, for any $S \in \fbrick \Lambda$, there exists a cover relation $\U \covered \T$ in $\ftors \Lambda$ labelled by $S$. See e.g. \cite[Lem. 5.2]{BorMot} for a proof.

	Now suppose $\U \covered \T$ and $\U' \covered \T'$ are both cover relations of $\ftors \Lambda$ labelled by the same f-brick $S$. By the definition of the brick labelling (see Section~\ref{subsec:lrfsbrick}), this means $\U^\perp \cap \T = \Filt\{S\} = (\U')^\perp \cap \T'$. Thus $Z^{\mathsf f}_{[\U,\T]} = Z^{\mathsf f}_{[\U',\T']}$, and so $\psi(X_S)$ is well defined. Moreover, one has $$\psi(Y_{\T}) = Z^{\mathsf f}_{[0,\T]} = Z^{\mathsf f}_{[\U,\T]}Z^{\mathsf f}_{[0,\U]} = \psi(X_S)\psi(Y_\U).$$
	Since also $\psi(Y_0) = Z_{[0,0]} = e$, we conclude that $\psi$ is a morphism of groups. The fact that $\psi$ is surjective then follows from observing that
    $$Z_{[\U,\T]}^{\mathsf{f}} = Z_{[0,\T]}^{\mathsf{f}}(Z_{[0,\U]}^{\mathsf{f}})^{-1} = \psi(Y_\T)\psi((Y_\U)^{-1})$$
    for any $\U \subseteq \T \in \ftors\Lambda$.
	
	Now suppose that $\Lambda$ is $\tau$-tilting finite, so that $\tors \Lambda = \ftors \Lambda$ and $|\ftors\Lambda| < \infty$ by \cite[Thm.~1.2 and Cor.~2.9]{DIJ2019}. We define a candidate inverse $\upsilon: \intheart(\ftors \Lambda) \rightarrow G(\Lambda)$ by setting $\upsilon(Z_{[\U,\T]}^{\mathsf f}) = Y_\T (Y_\U)^{-1}$ for every closed interval $[\U,\T]$ in $\ftors\Lambda = \tors\Lambda$. There are now three types of relations to consider:
	\begin{enumerate}
		\item Let $\T \in \ftors \Lambda$. Then $\upsilon(Z_{[\T,\T]}^{\mathsf f}) = Y_\T(Y_\T)^{-1} = e$.
		\item Let $\V \subseteq \U \subseteq \T \in \ftors \Lambda$. Then $\upsilon(Z_{[\U,\T]}^{\mathsf f})\upsilon(Z_{[\V,\U]}^{\mathsf f}) = Y_{\T}(Y_{\U})^{-1}Y_\U(Y_\V)^{-1} = Y_\T(Y_\V)^{-1} = \upsilon(Z_{[\V,\T]}^{\mathsf f})$.
		\item Let $[\U,\T]$ and $[\U',\T']$ in $\tors \Lambda$ such that $\U^\perp \cap \T = (\U')^\perp \cap \T'$. Now let $(S_1,\ldots,S_k)$ be a sequence of bricks labelling a path $\T \rightarrow \U$ in $\mathrm{Hasse}(\ftors \Lambda)$. (Such a path exists since $\ftors\Lambda = \tors \Lambda$ is finite.) By Proposition~\ref{prop:brickpreserve}, we have that $(S_1,\ldots,S_k)$ also labels a path $\T' \rightarrow \U'$. Moreover, since $\Lambda$ is assumed to be $\tau$-tilting finite, all bricks are f-bricks by Remark~\ref{rem:fbrick}. Thus we have
		$$\upsilon(Z_{[\U,\T]}^{\mathsf f}) = Y_\T(Y_\U)^{-1} = X_{S_1}\cdots X_{S_k} = Y_{\T'}(Y_{\U'})^{-1} = \upsilon(Z_{[\U',\T']}^{\mathsf f}).$$
	\end{enumerate}
	We conclude that $\upsilon$ is also a morphism of groups. It is then straightforward to verify that $\upsilon$ and $\psi$ are inverse to one another.
\end{proof}

We conclude this section by showing that the group $\intheart(\tors \Lambda)$ depends only on the abstract lattice structure of $\tors \Lambda$. Note that it is not immediately clear whether the group $\intheart(\ftors \Lambda)$ likewise depends only on the abstract poset structure of $\ftors \Lambda$.

\begin{lemma}\label{lem:invariant_under_isom}
	Let $\Lambda_1$ and $\Lambda_2$ be finite-dimensional algebras (over possibly distinct fields), and suppose there exists a lattice isomorphism $\eta: \tors \Lambda_1 \rightarrow \tors \Lambda_2$. Then the following hold.
	\begin{enumerate}
		\item For any pair of closed intervals $[\U,\T], [\U',\T] \subseteq \tors \Lambda_1$, one has that $\U^\perp \cap \T = (\U')^\perp \cap \T'$ if and only if $\eta(\U)^\perp \cap \eta(\T) = \eta(\U')^\perp \cap \eta(\T')$.
		\item There is an isomorphism $\intheart(\tors \Lambda_1) \rightarrow \intheart(\tors \Lambda_2)$ given on generators by $Z_{[\U,\T]} \mapsto Z_{[\eta(\U),\eta(\T)]}$.
	\end{enumerate}
\end{lemma}

\begin{proof}
	(1) As we now explain, this essentially follows from the proof of \cite[Thm.~3.10]{Enomoto23}.
	
	In \cite[Def. 3.8]{Enomoto23}, Enomoto defines a map $\mathsf{j}\text{-}\mathsf{label}$ from the set of closed intervals of a completely semidistributive lattice to the power set of the lattice's set of ``completely join-irreducible elements''. It can be verified directly from the definition that this map commutes with lattice isomorphisms. More precisely, in the setup of the theorem, one has that $$\{\eta(x): x \in \mathsf{j}\text{-}\mathsf{label}[\U,\T]\} = \{\mathsf{j}\text{-}\mathsf{label}[\eta(\U),\eta(\T)]\},$$
	and likewise for the primed version of $\U$ and $\T$. Now the commutative diagram \cite[(3.1)]{Enomoto23} (found in the proof of \cite[Thm.~3.10]{Enomoto23}) shows that $\U^\perp \cap \T = (\U')^\perp \cap \T'$ if and only if $\mathsf{j}\text{-}\mathsf{label}[\U,\T] = \mathsf{j}\text{-}\mathsf{label}[\U',\T']$. The result then follows.
	
	(2) This is an immediate consequence of (1) and the presentation of the interval-heart group.
\end{proof}

As a consequence of Lemma~\ref{lem:invariant_under_isom} and Proposition~\ref{prop:embedding} we obtain the following. Note that this can also be deduced from \cite[Thm.~4.10]{HI21} and \cite[Cor.~1.2]{Kai24}.

\begin{corollary}\label{cor:invariant_under_isom}
	Let $\Lambda_1$ and $\Lambda_2$ be $\tau$-tilting finite algebras (over possibly distinct fields), and suppose there exists a lattice isomorphism $\eta: \tors \Lambda_1 \rightarrow \tors \Lambda_2$. Then the picture groups $G(\Lambda_1)$ and $G(\Lambda_2)$ are isomorphic.
\end{corollary}

\subsection{Group functors}\label{sec:functors}

Recall that $\mathfrak{W}(\Lambda)$ denotes the $\tau$-cluster morphism category of $\Lambda$. Considering the picture group $G(\Lambda)$ as a groupoid with one object, 
a functor $\mathfrak{W}(\Lambda) \rightarrow G(\Lambda)$ is constructed in \cite[Thm.~4.15]{HI21} in the $\tau$-tilting finite case. Furthermore, this functor is shown to be faithful when the picture group $G(\Lambda)$ satisfies a technical condition which essentially says that distinct generators of the group must correspond to distinct group elements (see Definition~\ref{def:distinct_generators} below).
 In extending this to algebras which are not $\tau$-tilting finite, we find it useful to work with the interval-heart group $\intheart(\ftors \Lambda)$ and the realisation of the $\tau$-cluster morphism category from \cite{Kai24} that was recalled in Definition~\ref{defn:latticedef}.

\begin{theorem}\label{thm:functor}
	Let $\Lambda$ be a finite-dimensional algebra. Then
	\begin{enumerate}
		\item There is a group functor $\Gamma: \mathfrak{W}(\Lambda) \rightarrow \intheart(\ftors \Lambda)$ given as follows. Let $\rho$ be a morphism of $\mathfrak{W}(\Lambda)$, represented by a containment $[\U_1,\T_1] \leq [\U_2,\T_2]$ between $\tau$-perpendicular intervals. Then $\Gamma(\rho) = Z^{\mathsf f}_{[\U_1,\U_2]}$.
	\item Suppose that, for any closed intervals $[\U,\T]$ and $[\U',\T']$ in $\ftors \Lambda$, we have that $Z^{\mathsf f}_{[\U,\T]} = Z^{\mathsf f}_{[\U',\T']}$ if and only if $\U^\perp \cap \T = (\U')^\perp \cap \T'$. Then the functor $\Gamma$ is faithful.
	\end{enumerate}
\end{theorem}

\begin{proof}
	(1) We first show that $\Gamma(\rho)$ does not depend on the representative of the equivalence class. Suppose that $[\U_1,\T_1] \leq [\U_2,\T_2]$ and $[\U_1',\T_1'] \leq [\U_2',\T_2']$ are representatives of a morphism $\rho$. This means that the domain of $\rho$ is $\U_1^\perp \cap \T_1 = \U_2^\perp \cap \T_2$ and that $[\U_2,\T_2] \cap \U_1^\perp \cap \T_1 = [\U'_2,\T'_2] \cap (\U_1')^\perp \cap \T'_1$. Since $\U_2 \subseteq \T_1$ and $\U_2' \subseteq \T_1'$, the implication $(1 \implies 3$) of Lemma~\ref{lem:morphismlifts} then implies that $\U_1^\perp \cap \U_2 = (\U_1')^\perp \cap \U_2'$. We conclude that $\Gamma(\rho)$ does not depend on the chosen representative of $\rho$.
	
	We now show that $\Gamma$ preserves composition. Let $\rho_1$ and $\rho_2$ be composable morphisms. By Lemma~\ref{lem:composition}, this means there exist $\tau$-perpendicular intervals $[\U_1,\T_1] \leq [\U_2,\T_2] \leq [\U_3,\T_3]$ such that $[\U_1,\T_1] \leq [\U_2,\T_2]$ is a representative of $\rho_1$, $[\U_2,\T_2] \leq [\U_3,\T_3]$ is a representative of $\rho_2$, and $[\U_1,\T_1] \leq [\U_3,\T_3]$ is a representative of $\rho_2 \circ \rho_1$. We then compute
	$$\Gamma(\rho_2 \circ \rho_1) = Z^{\mathsf f}_{[\U_1,\U_3]} = Z^{\mathsf f}_{[\U_2,\U_3]}Z^{\mathsf f}_{[\U_1,\U_2]} = \Gamma(\rho_2) \Gamma(\rho_1).$$
	
	(2) Let $\rho, \rho': \W \rightarrow \W'$ be morphisms in $\T(\Lambda)$ with the same domain and codomain. Choose representatives $[\U_1,\T_1] \leq [\U_2,\T_2]$ and $[\U_1',\T_1'] \leq [\U_2',\T_2']$ of $\rho$ and $\rho'$, respectively.
    Similar to the proof of (1), the equivalence $(1\iff 3)$ of Lemma~\ref{lem:morphismlifts} then implies that $\rho = \rho'$ if and only if $\U_1^\perp \cap \U_2 = (\U'_1)^\perp \cap \U'_2$. Now by assumption, the second equality holds if and only if $\Gamma(\rho) = \Gamma(\rho')$. We conclude that $\Gamma$ is faithful.
\end{proof}

Motivated by Theorem~\ref{thm:functor}(2) and \cite[Thm.~4.15]{HI21p}, we consider the following definitions.

\begin{definition}\label{def:distinct_generators}
    Let $\Lambda$ be a finite-dimensional algebra.
    \begin{enumerate}
        \item We say that $\intheart(\tors\Lambda)$ \emph{has heart-controlled generators} if, for any closed intervals $[\U,\T]$ and $[\U',\T']$ in $\tors \Lambda$, we have that $Z_{[\U,\T]} = Z_{[\U',\T']}$ if and only if $\U^\perp \cap \T = (\U')^\perp \cap \T'$. 
        \item We say that $\intheart(\ftors\Lambda)$ \emph{has heart-controlled generators} if, for any closed intervals $[\U,\T]$ and $[\U',\T']$ in $\ftors \Lambda$, we have that $Z^{\mathsf{f}}_{[\U,\T]} = Z^{\mathsf{f}}_{[\U',\T']}$ if and only if $\U^\perp \cap \T = (\U')^\perp \cap \T'$.
        \item We say that $G(\Lambda)$ \emph{has distinct generators} if both of the following hold.
        \begin{enumerate}
            \item Let $\U, \T \in \ftors \Lambda$. Then $Y_\U = Y_\T \in G(\Lambda)$ if and only if $\U = \T$.
            \item Let $S, T \in \fbrick \Lambda$. Then $X_S = X_T \in G(\Lambda)$ if and only if $S = T$. Moreover, $X_S \neq e \in G(\Lambda)$.
        \end{enumerate}
    \end{enumerate}
\end{definition}

 In Section~\ref{sec:ringel}, we will show that $\intheart(\tors\Lambda)$ has heart-controlled generators for any finite-dimensional algebra $\Lambda$ over a finite field. We defer our discussion of those classes of algebras for which $G(\Lambda)$ has previously been shown to have distinct generators to Section~\ref{sec:remarks}.

Using the definitions of the groups and the morphisms $\iota$ (from Remark~\ref{rem:int_group}) and $\psi$ (from Proposition~\ref{prop:embedding}), we obtain the following implications between the three parts of Definition~\ref{def:distinct_generators}.

\begin{proposition}\label{prop:distinct_generators}
    Let $\Lambda$ be a finite-dimensional algebra. Then the following hold.
    \begin{enumerate}
        \item If $\intheart(\tors\Lambda)$ has heart-controlled generators then so does $\intheart(\ftors\Lambda)$. \item If $\intheart(\ftors\Lambda)$ has heart-controlled generators then $G(\Lambda)$ has distinct generators.
    \end{enumerate}
\end{proposition}

We conclude this section by recording the following conjecture. Note that, by Theorem~\ref{thm:functor} and Proposition~\ref{prop:distinct_generators}, the faithfulness of $\Gamma$ follows from $\intheart(\tors\Lambda)$ having heart-controlled generators.

\begin{conjecture}\label{conj:faithful}
    Let $\Lambda$ be a finite-dimensional algebra. Then $\intheart(\tors\Lambda)$ has heart-controlled generators and the functor $\Gamma$ from Theorem~\ref{thm:functor} is a faithful group functor $\mathfrak{W}(\Lambda) \rightarrow \intheart(\ftors \Lambda)$.
\end{conjecture}

\subsection{Completed Ringel--Hall algebras}\label{sec:ringel}

In this section, we prove that the interval-heart groups of finite-dimensional over finite fields have heart-controlled generators. Thus, by Theorem~\ref{thm:functor}, we conclude that the $\tau$-cluster morphism categories of these algebras admit faithful group functors. Our argument uses the machinery of completed (finitary) Ringel--Hall algebras. We introduce these following exposition similar to that of \cite[Sec.~7]{Treffinger25}.

Fix a finite-dimensional algebra $\Lambda$ over a finite field $K$, and fix a skeleton $[\mods \Lambda]$ of the category $\mods \Lambda$. For a subcategory $\calC \subseteq \mods\Lambda$ which is closed under isomorphisms, we denote $[\calC] := \calC \cap [\mods\Lambda]$. For $M, N, E \in \mods \Lambda$, we let $c_{M,N}^E$ denote the number of submodules $M' \subseteq E$ such that $M' \cong N$ and $E/M' \cong N$. The \emph{completed (finitary) Ringel--Hall algebra} of $\Lambda$ is the $\mathbb{Z}$-algebra $\hall(\Lambda)$ whose elements are formal series
$$\sum_{[M] \in [\mods \Lambda]} a_{[M]} [M]$$
with multiplication given by
$$[M][N] = \sum_{[E] \in [\mods \Lambda]} c_{M,N}^E [E].$$
By \cite[Prop.~1]{ringel_hall}, this is an associative algebra with identity $[0]$.

For a subcategory $\mathcal{C} \subseteq \mods\Lambda$ which is closed under isomorphisms, we denote
	$$E_\mathcal{C} := \sum_{[M]\in [\mathcal{C}]}[M] \in \widehat{\mathscr{H}}(\Lambda).$$

The following result of Treffinger forms the basis or our argument.

\begin{proposition}\label{thm:treffinger}\cite[Thm. 7.1]{Treffinger25}
	Let $\Lambda$ be a finite-dimensional algebra over a finite field and let
	$$0 = \T_0 \subsetneq \T_1 \subsetneq \cdots \subsetneq \T_m = \mods \Lambda$$
	be a chain of torsion classes. Then
	$$E_{\mods \Lambda} = E_{\T_{m-1}^\perp \cap \T_m} E_{\T_{m-2}^\perp \cap \T_{m-1}}\cdots E_{\T_1^\perp \cap \T_2} E_{\T_0^\perp \cap \T_1}.$$
\end{proposition}

Before continuing, we establish some additional notation.
	\begin{itemize}
		\item Let $F \in \widehat{\mathscr{H}}(\Lambda)$. For $M \in \mods \Lambda$, we denote by $F_{[M]}$ the coefficient of $M$ in $F$. We denote by $\mathsf{supp}(F)$ the subcategory of objects $M \in \mods\Lambda$ so that $F_{[M]} \neq 0$. 
		\item For $\ell \in \mathbb{N}$, we denote by $\mods_\ell \Lambda \subseteq \mods \Lambda$ the subcategory consisting of modules of length $\ell$.
		\item For $\ell \in \mathbb{N}$, we denote by $\mathcal{O}(\ell)$ the two-sided ideal of $\widehat{\mathscr H}(\Lambda)$ generated by the symbols $[M]$ for $M \in \mods_\ell \Lambda$.
	\end{itemize}

The following lemma is critical.

\begin{lemma}\label{lem:invertible}
	Let $\mathcal{C}\subseteq \mods \Lambda$ be a subcategory which is closed under isomorphisms. Then $E_\mathcal{C}$ is invertible in $\widehat{\mathscr{H}}(\Lambda)$ if and only if $0 \in \mathcal{C}$.
\end{lemma}

\begin{proof}
Suppose first that $0 \in \mathcal{C}$. We symmetrically define two elements $F, F' \in \widehat{\mathscr H}(\Lambda)$ as follows:
\begin{enumerate}
	\item Define $F_0 = [0] = F'_0$.
	\item For $0 \neq \ell \in \mathbb{N}$, denote $G_\ell = \sum_{j = 0}^{\ell-1} F_j$ and $G'_\ell = \sum_{j = 0}^{\ell-1} F'_j$.  Then define
	$$F_\ell = -\sum_{[M] \in [\mods_\ell \Lambda]} (E_\mathcal{C}G_\ell)_{[M]}[M],\qquad F'_\ell = -\sum_{[M] \in [\mods_\ell \Lambda]} (E_\mathcal{C} G'_\ell)_{[M]}[M]$$
	\item Define $F = \sum_{\ell = 0}^\infty F_\ell$ and $F' = \sum_{\ell = 0}^\infty F'_\ell$.
\end{enumerate}
We claim that $E_\mathcal{C}F = [0] = F' E_\mathcal{C}$. In particular, this will imply that $F = F'$, and that this is the inverse of $E_\mathcal{C}$. We will show only that $E_\mathcal{C}F = [0]$, as the proof that $F' E_\mathcal{C} = [0]$ is entirely analogous.

We will demonstrate by induction on $\ell$ that $E_\mathcal{C} G_\ell \in [0] + \mathcal{O}(\ell)$ for all $\ell$. Since $E_\mathcal{C}(F - G_\ell) \in \mathcal{O}(\ell)$ and $\bigcap_{\ell = 1}^\infty \mathcal{O}(\ell) = 0$, this will imply the claim. For $\ell = 1$, we have $E_\mathcal{C}Y_1 = E_\mathcal{C} = [0] + \mathcal{O}(1)$ since $0 \in \mathcal{C}$. For $\ell > 1$, we note that
$$E_\mathcal{C}F_{\ell-1} = -E_\mathcal{C} \sum_{[M] \in [\mods_{\ell-1}\Lambda]} (E_\mathcal{C} G_{\ell-1})_{[M]} [M] \in -\sum_{[M] \in [\mods_{\ell-1}\Lambda]} (E_\mathcal{C} G_{\ell-1})_{[M]} [M] + \mathcal{O}(\ell)$$
again because $0 \in \mathcal{C}$. By the induction hypothesis, we then have
\begin{eqnarray*}
	E_\mathcal{C} G_\ell &=& E_\mathcal{C}(G_{\ell-1} + F_{\ell-1})\\
		&\in& [0] + \sum_{[M] \in [\mods_{\ell-1}\Lambda]} (E_\mathcal{C} G_{\ell-1})_{[M]} [M] + E_\mathcal{C}F_{\ell-1} + \mathcal{O}(\ell)\\
		&=& [0] + \mathcal{O}(\ell).
\end{eqnarray*}
Now suppose $0 \notin \mathcal{C}$. We define an element $F'' \in \widehat{\mathscr H}(\Lambda)$ as follows:
\begin{enumerate}
	\item Define $F''_1 = \sum_{[M] \in [\mods_1 \Lambda]} [M]$.
	\item For $1 < \ell \in \mathbb{N}$, denote $G''_\ell = \sum_{j = m}^{\ell-1}F''_j$. Then define
		$$F''_\ell = -\sum_{[M] \in [\mods_\ell \Lambda]}(E_\mathcal{C}G''_\ell)_{[M]} [M].$$
	\item Define $F'' = \sum_{\ell = 1}^\infty F_\ell$.
\end{enumerate}
Similar to the case where $0 \in \mathcal{C}$, it can be shown that $E_\mathcal{C}G''_\ell \in \mathcal{O}(\ell)$ for all $\ell > 1$. This means $E_\mathcal{C}F'' = 0$, and so $E_\mathcal{C}$ is not invertible.
\end{proof}

We now prove the main technical result of this section. We denote by $\widehat{\mathscr H}(\Lambda)^*$ the group of units of $\widehat{\mathscr H}(\Lambda)$.

\begin{proposition}\label{thm:MapToHall}
	Let $\Lambda$ be a finite-dimensional algebra over a finite field. Then there is a morphism of groups $\phi:\intheart(\tors \Lambda)\rightarrow \widehat{\mathscr H}(\Lambda)^*$ given by $\phi(Z_{[\U,\T]}) = E_{\U^\perp \cap \T}$.
\end{proposition}

\begin{proof}
	We note that for every interval $[\U,\T]$ in $\tors\Lambda$, we have that $\phi(Z_{[\U,\T]}) \in \widehat{\mathscr H}(\Lambda)^*$ by Lemma~\ref{lem:invertible}. Thus we need only show that $\phi$ is well-defined; i.e., that it preserves the relations.
	
	For the first type of relation, let $\T \in \tors \Lambda$. Then $\T^\perp \cap \T = 0$, and so $\phi(Z_{[\T,\T]}) = e_{\{0\}} = [0]$ as desired.
	
	For the second type of relation, let $\V \subseteq \U \subseteq \T \in \tors \Lambda$. By Proposition~\ref{thm:treffinger}, we then have
	\begin{eqnarray*}
		\phi(Z_{[\T, \mods \Lambda]}) \phi(Z_{[\U,\T]}) \phi(Z_{[\V,\U]})\phi(Z_{[0,\V]}) &=& E_{\T^\perp}E_{\U^\perp \cap \T}E_{\V^\perp \cap \U}E_\V\\
		&=& E_{\mods \Lambda}\\
		&=& E_{\T^\perp}E_{\V^\perp \cap \T}E_\V\\
		&=& \phi(Z_{[\T,\mods \Lambda]}) \phi(Z_{[\V,\T]}) \phi(Z_{[0,\V]}). 
	\end{eqnarray*}
	By Lemma~\ref{lem:invertible}, this implies that $\phi(Z_{[\U,\T]})\phi(Z_{[\V,\U]}) = \phi(Z_{[\V,\T]})$ as desired.
	
	Finally, the fact that if $\U^\perp \cap \T = (\U') \cap \T'$ then $\phi(Z_{[\U,\T]}) = \phi(Z_{[\U',\T']})$ follows immediately from the definition of $\phi$.
\end{proof}

\begin{conjecture}
	The morphism $\phi$ in Proposition~\ref{thm:MapToHall} is injective.
\end{conjecture}

As a consequence of Proposition~\ref{thm:MapToHall}, we obtain the following.

\begin{theorem}\label{thm:finite_field}
    Let $\Lambda$ be a finite-dimensional algebra over a finite field. Then $\intheart(\tors\Lambda)$ has heart-controlled generators. Consequently, the functor $\Gamma$ from Theorem~\ref{thm:functor} is a faithful group functor $\mathfrak{W}(\Lambda) \rightarrow \intheart(\ftors \Lambda)$.
\end{theorem}

\begin{proof}
    Let $[\U,\T]$ and $[\U',\T']$ be closed intervals in $\tors \Lambda$. It is clear from the definition that if $\U^\perp \cap \T = (\U')^\perp \cap \T'$ then $Z_{[\U,\T]} = Z_{[\U',\T']} \in \intheart(\tors\Lambda)$. Thus suppose that $Z_{[\U,\T]} = Z_{[\U',\T']} \in \intheart(\tors\Lambda)$. Then $E_{\U^\perp \cap \T} = \phi(Z_{[\U,\T]}) = \phi(Z_{[\U',\T']}) = E_{(\U')^\perp \cap \T'}$ in $\widehat{\mathscr H}(\Lambda)^*$ by Proposition~\ref{thm:MapToHall}. It follows that $\U^\perp \cap \T = (\U')^\perp \cap \T'$, as desired.

    We have shown that $\intheart(\tors\Lambda)$ has heart-controlled generators. Now the fact that $\Gamma$ is a faithful group functor follows from Theorem~\ref{thm:functor} and Proposition~\ref{prop:distinct_generators}.
\end{proof}

\subsection{Remarks on infinite fields}\label{sec:remarks}

We conclude this appendix with some remarks on the extent to which Theorem~\ref{thm:finite_field} can be extended to algebras over infinite fields. Our starting point is the following.

\begin{corollary}\label{cor:field_change}
	Let $\Lambda$ be a finite-dimensional algebra over an arbitrary field, and suppose there exists a lattice isomorphism $\eta: \tors\Lambda \rightarrow \tors\Lambda'$ for some finite-dimensional algebra $\Lambda'$ over a finite field. Then $\intheart(\tors\Lambda)$ has heart-controlled generators. Consequently, the functor $\Gamma$ from Theorem~\ref{thm:functor} is a faithful group functor $\mathfrak{W}(\Lambda) \rightarrow \intheart(\ftors\Lambda)$.
\end{corollary}

\begin{proof}
    Let $[\U,\T], [\U',\T]$ be closed intervals in $\tors \Lambda$. It is clear than if $\U^\perp \cap \T = (\U')^\perp \cap \T'$ then $Z_{[\U,\T]} = Z_{[\U',\T']} \in \intheart(\tors\Lambda)$. Thus suppose that $Z_{[\U,\T]} = Z_{[\U',\T']} \in \intheart(\tors \Lambda)$. 
    Then, by Lemma~\ref{lem:invariant_under_isom}(2), we have $Z_{[\eta(\U),\eta(\T)]} = Z_{[\eta(\U'),\eta(\T')]} \in \intheart(\tors \Lambda_1)$. Theorem~\ref{thm:finite_field} then implies that $\eta(\U)^\perp \cap \eta(\T) = \eta(\U')^\perp \cap \eta(\T')$. By Lemma~\ref{lem:invariant_under_isom}(1), we conclude that $\U^\perp \cap \T = (\U')^\perp \cap \T'$.

    We have shown that $\intheart(\tors\Lambda)$ has heart-controlled generators. Now the fact that $\Gamma$ is a faithful group functor follows from Theorem~\ref{thm:functor} and Proposition~\ref{prop:distinct_generators}.
\end{proof}

\begin{remark}
    Suppose that $\Lambda$ is $\tau$-tilting finite in the setup of Corollary~\ref{cor:field_change}. Then the fact that $\mathfrak{W}(\Lambda)$ admits a faithful group functor can also be deduced from \cite[Sec.~7]{Kai24}.
\end{remark}

We now give several examples demonstrating how one can use Corollary~\ref{cor:field_change} in practice. Note that the ($\tau$-)cluster morphism categories of representation finite (and also tame) hereditary algebras were previously shown to admit faithful group functors in \cite[Thm.~3.7]{IgusaTodorov22}.

\begin{proposition}\label{prop:preproj_distinct_generators}
    Let $\Lambda$ be a representation finite hereditary algebra over an arbitrary field $K$, and let $\Pi(\Lambda)$ denote the preprojective algebra of $\Lambda$. Then $\intheart(\tors\Lambda)$ and $\intheart(\tors\Pi(\Lambda))$ both have heart-controlled generators. Consequently, the functor $\Gamma$ from Theorem~\ref{thm:functor} is a faithful group functor $\mathfrak{W}(\Lambda) \rightarrow \intheart(\ftors\Lambda)$ (resp. $\mathfrak{W}(\Pi(\Lambda)) \rightarrow \intheart(\ftors\Pi(\Lambda))$).
\end{proposition}

\begin{proof}
    By \cite[Thm.~C]{DlabRingel75}, we have that $\Lambda$ is Morita equivalent to the tensor algebra $K\mathbb{S}$ of a Dynkin $K$-species $\mathbb{S}$. The preprojective algebras $\Pi(\Lambda)$ and $\Pi(K\mathbb{S})$ are then also Morita equivalent by \cite[Thm.~2.3 and Cor.~5.5]{CB99}. Now let $K'$ be a finite field such that there exists a $K'$-species $\mathbb{S}'$ with the same underlying modulated quiver as $\mathbb{S}$.

    In the simply-laced case, it follows from \cite[Thm.~4.3]{IT09} that $\tors(K\mathbb{S})(\cong\tors\Lambda)$ and $\tors(K'\mathbb{S}')$ are both isomorphic to the \emph{Cambrian lattice} determined by the (modulated) quiver underlying $\mathbb{S}$. The analogous result can also be deduced outside the simply-laced case by combining \cite[Thm.~1]{GLS2020} and \cite[Cor.~8.17]{Gyoda2022}.
    
    Similarly, \cite[Thm. 7.9]{AHIKM2022} and Theorem~\ref{thm:AIRftorsbij} imply that $\ftors(\Pi(K\mathbb{S}))(\cong\ftors(\Pi(\Lambda)))$ and $\ftors(\Pi(K'\mathbb{S}'))$ are both anti-isomorphic to the \emph{weak order} on the Coxeter group determined by the modulated graph underlying $\mathbb{S}$. Now since $\mathbb{S}$ is Dynkin, it is well-known that this Coxeter group is finite. Thus $\Pi(\Lambda)$ and $\Pi(K'\mathbb{S}')$ are $\tau$-tilting finite and, by \cite[Thm.~1.2]{DIJ2019}, one has
    $$\tors(\Pi(\Lambda)) = \ftors(\Pi(\Lambda)) \cong \ftors(\Pi(K'\mathbb{S}')) = \tors(\Pi(K'\mathbb{S}')).$$
    Corollary~\ref{cor:field_change} now implies the result.
\end{proof}

For the next example, we refer to \cite{ButlerRingel} for the definition of a string algebra and for results concerning string and band modules that we apply without proof.

\begin{proposition}\label{prop:string}
    Let $\Lambda$ be a representation finite string algebra over an arbitrary field $K$. Then $\intheart(\tors\Lambda)$ has heart-controlled generators. Consequently, the functor $\Gamma$ from Theorem~\ref{thm:functor} is a faithful group functor $\mathfrak{W}(\Lambda) \rightarrow \intheart(\ftors\Lambda)$.
\end{proposition}

\begin{proof}
    By the definition of a string algebra, there exists a quiver $Q$ and a set of paths $p$ in $Q$ such that $\Lambda = KQ/(p)$. The definition then implies that $K'Q/(p)$ is also a string algebra for any field $K'$. In particular, we can take $K'$ to be finite.

   The indecomposable modules over a string algebra are classified (up to isomorphism) as ``string modules'' and ``band modules''. It is well known that a string algebra is representation infinite if and only if it admits a band module. Moreover, while the classification of band modules depends on the field, the existence of a band module depends only of $Q$ and $p$. Thus $K'Q/(p)$ is also representation finite and all indecomposable modules (over either $\Lambda = KQ/(p)$ or $K'Q/(p)$) are string modules.

   Unlike band modules, the classification of string modules depends only on $Q$ and $p$, not on the choice of base field. The same is also true for the dimension of the Hom-space between string modules by a result of \cite{CrawleyBoevey}. Putting this together, we conclude that there exists a bijection $\omega: \mathrm{ind}(\mods\Lambda) \rightarrow \mathrm{ind}(\mods K'Q/(p))$ such that $\dim_K\Hom_\Lambda(X,Y) = \dim_{K'}\Hom_{K'Q/(p)}(\omega(X),\omega(Y))$ for all $X$ and $Y$.

   Now it is well known that a subcategory $\T$ is a torsion class if and only if $\T = {}^\perp(\T^\perp)$. Thus $\omega$ extends to an isomorphism $\tors(\Lambda) \rightarrow \tors(K'Q/(p))$. Corollary~\ref{cor:field_change} then implies the result.
\end{proof}

Special cases of representation finite string algebras include $\tau$-tilting finite gentle algebras (see \cite[Thm.~1.1]{plamondon} or \cite[Thm.~7.1]{mousavand}) and the ``Nakayama-like algebras'' studied in \cite{HI21p}. Faithful group functors for these classes of algebras were previously established in \cite[Sec.~5]{HI21p}. We conclude with some clarifying comments about this result.

In \cite[Sec.~5.2]{HI21p}, we introduced what we called the ``power series 0-Hall algebra'' of a finite-dimensional algebra $\Lambda$ (over a field $K$). We then used this to show that the picture groups of Nakayama-like algebras and of $\tau$-tilting finite gentle algebras whose quivers have no loops or 2-cycles have distinct generators. (We then used \cite[Thm.~4.15]{HI21p} to conclude that the $\tau$-cluster morphism categories of these algebras admit faithful group functors.) Implicit in the construction of the power series 0-Hall algebra are the assumptions that (1) if $K$ is finite then $\Lambda$ admits Hall polynomials, and (2) if $K$ is infinite then one can replace $K$ with a finite field and use the resulting algebra to understand the representation theory of $\Lambda$. Assumption (2) can be made precise by, for example, requiring the existence of a bijection $\eta$ as in the proof of Proposition~\ref{prop:string}. In particular, this means assumption (2) is satisfied by $\tau$-tilting finite gentle algebras and by Nakayama-like algebras. On the other hand, the problem of determining whether the algebras in these families satisfy assumption (1) seems to be open, see e.g. \cite{Hall_gentle} for some work in this direction. To avoid this existence problem, one could instead use Proposition~\ref{prop:string} to deduce \cite[Cor.~5.13]{HI21p}.

Finally, we recall that picture groups of Nakayama algebras were shown to have distinct generators in \cite[Sec.~4]{HI21} using a different, but morally related, construction. While this construction does not rely on the existence of Hall polynomials (and is independent of the base field), we note that Nakayama algebras over finite fields are known to admit Hall polynomials \cite{Hall_nakayama}.
\providecommand{\etalchar}[1]{$^{#1}$}

\end{document}